\documentclass[10pt,dvipsnames, reqno]{amsart}
\usepackage[T1]{fontenc}
\usepackage{amsmath,amssymb,amsfonts, mathtools}
\usepackage{bbm}
\usepackage[mathscr]{eucal}
 \usepackage[margin=1.5in]{geometry}
\usepackage[all]{xy}
\usepackage{xparse}
\usepackage{fullpage}

\allowdisplaybreaks

\usepackage{tikz, tikz-cd}%
\usetikzlibrary{matrix,arrows,decorations.pathmorphing,cd,patterns,calc,backgrounds,patterns}
\tikzset{dummy/.style= {circle,fill,draw,inner sep=0pt,minimum size=1.2mm}}
\tikzset{vertex/.style={fill, circle, minimum size=.1cm, inner sep=0pt}}

 \usepackage{comment}
\usetikzlibrary{matrix,arrows}
\usepackage{extarrows}
\usepackage{url}

\newcommand{\C}{\mathcal{C}}
\newcommand{\Cc}{\mathsf{C}}
\renewcommand{\Mc}{\mathsf{M}}
\newcommand{\Wc}{\mathsf{W}}
\newcommand{\op}{\mathsf{op}}
\newcommand{\id}{\mathsf{id}}
\newcommand{\bI}{\mathbb{I}}
\newcommand{\N}{\mathcal{N}}
\newcommand{\cotor}{\mathrm{CoTor}}

\newcommand{\dual}{\star}

\newcommand{\dcomod}{\mathcal{D}\bicomod{}{}}

\newcommand{\Ainf}{{\mathbb{E}_1}}

\newcommand{\Ck}{{\mathrm{Ch}_\mathbbm{k}^{\geq 0}}}
\newcommand{\comod}{\mathrm{CoMod}}

\newcommand{\tr}{\mathrm{tr}}
\newcommand{\cotr}{\mathrm{cotr}}

\newcommand{\ncoTHH}{\mathrm{coTHH}}
\newcommand{\coTHH}{\mathrm{coHH}}
\newcommand{\THH}{\mathrm{THH}}
\newcommand{\coHH}{\mathrm{coHH}}
\newcommand{\HH}{\mathrm{HH}}

\newcommand{\holim}{\mathrm{holim}}

\renewcommand{\hom}{\mathrm{Hom}}
\newcommand{\Endo}{\mathrm{End}}

\newcommand{\ev}{\mathrm{ev}}

\newcommand{\ccobars}[1]{\Omega^\bullet (#1, C, C)}
\newcommand{\ccobar}[1]{\Omega(#1, C, C)}
\newcommand{\cobars}[2]{\Omega^\bullet(#1, C, #2)}
\newcommand{\cobar}[2]{\Omega(#1, C, #2)}

\newcommand{\pull}{\arrow[dr, phantom, "\lrcorner", very near start]}

\newcommand{\DDk}{\mathcal{D}^{\geq 0}(\mathbbm{k})}

\newcommand{\bicomod}[2]{{}_{#1}\mathrm{CoMod}_{#2}}
\newcommand{\bimod}[2]{{}_{#1}\mathrm{Mod}_{#2}}

\newcommand{\dcotensor}{\widehat{\square}}

\def\Ho{\textup{Ho}}
\def\Sp{\textup{Sp}}
\newcommand{\lang}{{\langle \langle}}
\newcommand{\rang}{{\rangle \rangle}}

\numberwithin{equation}{section}

\theoremstyle{plain}   
\newtheorem{thm}[equation]{Theorem} 
\newtheorem{cor}[equation]{Corollary}
\newtheorem{lemma}[equation]{Lemma}
\newtheorem{prop}[equation]{Proposition}

\theoremstyle{definition}
\newtheorem{defn}[equation]{Definition}

\newtheorem{rem}[equation]{Remark}
\newtheorem{ex}[equation]{Example}

\usepackage[pagebackref, colorlinks, citecolor=DarkOrchid, linkcolor=black, urlcolor=NavyBlue]{hyperref}
\hypersetup{linktoc=all,} 

\hypersetup{
  colorlinks   = true, 
  urlcolor     = blue, 
  linkcolor    = PineGreen, 
colorlinks, citecolor    = DarkOrchid 
}

\usepackage[textwidth=25mm, textsize=tiny]{todonotes}

\title{Trace methods for coHochschild homology}
\author[S. Klanderman]{Sarah Klanderman}
\author[M.\ P\'eroux]{Maximilien P\'eroux}

\address{Department of Mathematical and Computational Sciences, Marian University, 3200 Cold Spring Road, Indianapolis, IN, 46222, USA}
\email{sklanderman@marian.edu}

\address{Department of Mathematics, Michigan State University, 619 Red Cedar Rd, East Lansing, MI 48824, USA}
\email{peroux@msu.edu}
    
\keywords{coalgebra, comodule, coHochschild homology, trace, bicategory, $K$-theory, Morita--Takeuchi equivalence.}
\subjclass[2020]{Primary: 16E40, 16T15, 18N10, 19A99, 55P43, 55U15. Secondary: 16D90, 18F30, 18M70, 18N60, 57T30.}

\begin{document}

\begin{abstract}
Hochschild homology is a classical invariant of rings that plays an important role because of its connection to algebraic $K$-theory via the Dennis trace.
At level zero, the Dennis trace is induced by the Hattori--Stallings trace.
In this paper, we introduce new algebraic $K$-theories of coalgebras and obtain coalgebraic refinements of the Hattori--Stallings trace that connect these algebraic $K$-theories to coHochschild homology---the invariant analogous to Hochschild homology but for coalgebras.
We employ bicategorical methods of Ponto to show that coHochschild homology is a shadow. Consequently, we obtain that coHochschild homology is Morita--Takeuchi invariant.
\end{abstract}
\maketitle

\setcounter{tocdepth}{1}
\tableofcontents

\section{Introduction}

Trace methods are essential for computations in algebraic $K$-theory. The \textit{Dennis trace} $K_*(R)\rightarrow \HH_*(R)$ relates $K$-theory to the more computable \textit{Hochschild homology} (HH). B\"okstedt refined this by introducing \emph{topological Hochschild homology} (THH), yielding the trace $K(R)\rightarrow \THH(R)$, which has been instrumental in computations, for instance \cite{trace1, trace2, tch}.

\textit{CoHochschild homology} (coHH) is the invariant of coalgebras analogous to Hoch\-schild homology that was first introduced by Doi \cite{doi1981homological}.
Based on previous work by Hess--Parent--Scott in chain complexes \cite{hess2009cohochschild}, Hess--Shipley extended Doi's definition to the setting of stable homotopy theory and introduced \textit{topological coHochschild homology} (coTHH) as the spectral refinement of coHH in \cite{HScothh}. 
Given a connected space $X$, there is an equivalence of spectra \[\THH(\Sigma^\infty_+ \Omega X)\simeq \Sigma^\infty_+\mathcal{L}X,\] where $\mathcal{L}X$ is the free loop space on $X$, see \cite[7.3.11]{loday1998cyclichomology} or \cite[IV.3.3]{tch}.
Building on work of Malkiewich \cite[2.22]{Cary}, Hess--Shipley show that
further connectivity conditions on $X$ give an equivalence of spectra \[\ncoTHH(\Sigma^\infty_+X)\simeq \Sigma^\infty_+\mathcal{L}X,\] see \cite[3.7]{HScothh}. 
These equivalences of spectra have led to new computations of the homology of the free loop space \cite{loop}, using a\textit{ coB\"okstedt spectral sequence} for coTHH constructed in \cite{bohmann2018computational}.

Given the prominent role played by trace methods in computations of algebraic $K$-theory, it is natural to extend trace methods to coTHH.
Combining the equivalences above with the Dennis trace, Hess--Shipley obtained a new trace to coTHH \cite{HScothh} for $X$ simply connected:
\[K(\Sigma^\infty_+ \Omega X)\longrightarrow \THH(\Sigma^\infty_+ \Omega X)\simeq \ncoTHH(\Sigma^\infty_+ X).\]
Additionally, work of the second author and Bay{\i}nd{\i}r \cite[4.4]{dualitySW} shows that applying Spanier--Whitehead duality on the Dennis trace determines a pairing 
$K(C^\vee)\wedge \ncoTHH(C) \longrightarrow \mathbb{S}$,
that can be reformulated as 
\begin{equation}\label{eq: dualizable dennis trace}
   K(C^\vee)\longrightarrow \ncoTHH(C)^\vee, 
\end{equation}
where $C^\vee$ is the Spanier--Whithead dual of a coalgebra spectrum $C$ with some finiteness conditions.

However, both of the coTHH traces mentioned above are not internal to the coalgebraic setting; they are directly induced by combining the usual Dennis trace  with an identification of coTHH with THH.
Moreover, there has been no clear candidate of an algebraic $K$-theory for coalgebras without using a duality on coalgebras into rings.
In this paper, as a first step towards trace methods in the coalgebraic setting, we introduce new algebraic $K$-theories for coalgebras and construct analogues of the Dennis trace on coHochschild homology at level zero. Our results lay the foundational framework for building higher level traces on coHochschild homology as well as spectral refinements to coTHH.

\subsection*{Review of trace methods}
If $A$ is a connective ring spectrum, the Dennis trace $K(A)\rightarrow \THH(A)$ at level zero is a homomorphism $K_0(\pi_0(A))\rightarrow \HH_0(\pi_0(A))$, induced by $K_0(R)\rightarrow \HH_0(R)$ for a (discrete) ring $R$. 
In order to obtain trace maps in the coalgebraic setting, we must first determine what happens at level zero for coHochschild homology.
In the algebraic setting, the homomorphism $K_0(R)\rightarrow \HH_0(R)$ arises from the usual trace of  matrices, which is additive and cyclic.
These properties generalize to endomorphisms on dualizable objects in symmetric monoidal categories \cite{doldpuppe, traceinSMon}. In non-symmetric contexts, Hattori and Stallings define the \textit{Hattori--Stallings trace}, a homomorphism $\Endo_R(M)\rightarrow R/[R,R]$ for a finitely generated projective $R$-module $M$ \cite{hattori, stallings}. The group $R/[R,R]=\operatorname{HH}_0(R)$ serves as the universal target for trace maps $\mathcal{M}_n(R)\rightarrow A$. Evaluating this trace on identity endomorphisms induces $K_0(R)\rightarrow \HH_0(R)$.

Ponto generalized this approach by introducing the notion of traces in \textit{shadowed bicategories} \cite{ponto2008fixed}. 
The key observation is that given (not-necessarily commutative) rings $R$ and $S$, an $(R,S)$-bimodule $M$ and an $(S,R)$-bimodule $N$, the relative tensor products $M\otimes_S N$ and $N\otimes_R M$ are \textit{not} isomorphic.
In fact, they do not even have the same algebraic structure, as the former is an $(R,R)$-bimodule while the latter is an $(S,S)$-bimodule.
However, the two bimodules become equivalent after applying (relative) Hochschild homology:
\[
\HH_0(R, M\otimes_S N)\cong \HH_0(S, N\otimes_R M).
\]
Considering instead the derived tensor product shows that higher levels of Hochschild homology $\HH_*$ are also shadows, and replacing rings by ring spectra shows that THH is a shadow. In fact, THH is the universal cyclic invariant, and the universal shadow in a sense made precise in \cite{nima}. 
Every shadowed bicategory comes equipped with a notion of a bicategorical trace that is cyclic, and the bicategorical trace on dualizable bimodules recovers the Hattori--Stallings trace in the underived case.

\subsection*{Results}
 We dualize the construction of trace maps in the coalgebraic setting, using the framework of bicategorical shadows. In doing so, we are confronted with several obstacles.
 For instance, given coalgebras $C$ and $D$, the \textit{cotensor product} $M\square_D N$ of a $(C,D)$-bicomodule $M$ with a $(D,C)$-bicomodule $N$ need not to be a $(C,C)$-bicomodule.  
Even worse, the cotensor product is not even associative in general when considering coalgebras over $\mathbb{Z}$ \cite[IV.2.5]{families}. 
To remedy this, we assume that $\mathbbm{k}$ is a ring of global dimension zero, as this guarantees we have a nice associative relative cotensor product on bicomodules.

\begin{thm}[{Theorem \ref{thm: coHH zero is a shadow}}]\label{theorem: main theorem}
Let $\mathbbm{k}$ be a commutative ring with global dimension zero.
The zeroth coHochschild homology $\coHH_0$ defines a shadow on the bicategory of $\mathbbm{k}$-coalgebras and bicomodules with their relative cotensor products. 
\end{thm}

In particular, given $\mathbbm{k}$-coalgebras $C$ and $D$, a $(C,D)$-bicomodule $M$ and a $(D,C)$-bicomodule $N$, there is a natural isomorphism of $\mathbbm{k}$-modules
\[
\coHH_0(M\square_D N, C) \cong \coHH_0(N\square_C M, D).
\]
Given a right dualizable $(C,D)$-bicomodule $M$ and an endomorphism $f\colon M\rightarrow M$ of $(C,D)$-bicomodules, the fact that $\coHH_0$ is a shadow induces a $\mathbbm{k}$-linear homomorphism
\begin{equation}\label{eq: bicategorical trace}
 \tr_C^D(f)\colon \coHH_0(C) \longrightarrow \coHH_0(D), 
\end{equation}
that is cyclic (Proposition \ref{prop: cyclic trace}). Considering either $C=\mathbbm{k}$ or $D=\mathbbm{k}$ leads to two particularly interesting traces.

\subsubsection*{Case 1: $D=\mathbbm{k}$ in \eqref{eq: bicategorical trace}}
In this case, if $M$ is a finitely cogenerated and injective left $C$-comodule (see Definition \ref{def: finitely cogenerated}), then it is right dualizable as a $(C, \mathbbm{k})$-bicomodule.
Given a $C$-colinear endomorphism $f$ of $M$, the assignment $f\mapsto \tr_C^\mathbbm{k}(f)$ determines a $\mathbbm{k}$-linear homomorphism that we call the\textit{ Hattori--Stallings cotrace}:
\begin{equation*}\label{eq: hattori-stallings cotrace intro}
  {}_C\Endo(M)\rightarrow \hom_\mathbbm{k}(\coHH_0(C), \mathbbm{k}).  
\end{equation*}
We shall write $ \hom_\mathbbm{k}(\coHH_0(C), \mathbbm{k})$ as $\coHH_0(C)^*$.
The class of finitely cogenerated and injective left $C$-comodules is the dual analogue of the class of finitely generated and projective right $R$-modules considered when computing the algebraic $K$-theory $K(R)$ of a ring $R$. 
Therefore, this analogy leads us to introduce the following construction.

\begin{defn}\label{def: the first K-theory}
Let $\mathbbm{k}$ be a commutative ring with global dimension zero and let $C$ be a $\mathbbm{k}$-coalgebra.
Let $\bicomod{C}{}^\mathrm{fc, inj}$ denote the category of finitely cogenerated and injective left $C$-comodules, viewed as a Waldhausen category whose cofibrations are monomorphisms and whose weak equivalences are isomorphisms.
Let $K^c(C)$ denote the resulting algebraic $K$-theory spectrum associated to the Waldhausen category $\bicomod{C}{}^\mathrm{fc, inj}$, as in \cite{waldy} or \cite[IV.8.5]{Kbook}.
\end{defn}

Just as the Hattori--Stallings trace $\Endo_R(M)\rightarrow \HH_0(R)$ led to a homomorphism $K_0(R)\rightarrow \HH_0(R)$ of abelian groups by considering the trace on the identity, the Hattori--Stallings cotrace ${}_C\Endo(M)\rightarrow \coHH_0(C)^*$ on the identity (i.e., the \textit{corank}) induces the following map.

\begin{thm}[{Theorem \ref{thm: the corank as a theorem}}]\label{thm: first trace}
Let $\mathbbm{k}$ be a commutative ring with global dimension zero, and let $C$ be a $\mathbbm{k}$-coalgebra.
The Hattori--Stallings corank induces a homomorphism of abelian groups
\[
K_0^c(C) \longrightarrow \coHH_0(C)^*.
\]
\end{thm}

The cotrace $K_0^c(C) \rightarrow \coHH_0(C)^*$ is analogous
to the Dennis trace $K_0(C^*)\rightarrow \coHH_0(C)^*$ obtained when $C$ is finitely generated as $\mathbbm{k}$-modules, a particular case of \eqref{eq: dualizable dennis trace} from \cite{dualitySW}.
We view Theorem \ref{thm: first trace} as a first step in constructing the Dennis trace entirely in the coalgebraic setting, with no finiteness assumption on $C$.

\subsubsection*{Case 2: $C=\mathbbm{k}$ in \eqref{eq: bicategorical trace}}
In this case, relabeling $D$ as $C$, if $M$ is a right $C$-comodule that is finitely generated as a $\mathbbm{k}$-module, then it is right dualizable as a $(\mathbbm{k}, C)$-bicomodule.
Given a $C$-colinear endomorphism $f$ on $M$,
the assignment $f\mapsto \tr^C_\mathbbm{k}(f)(1)$
determines a $\mathbbm{k}$-linear homomorphism that we call the \textit{colinear Hattori--Stallings trace}:
\begin{equation*}\label{eq: colinear hattori-stallings}
  \Endo_C(M)\longrightarrow \coHH_0(C).  
\end{equation*}
The class of right $C$-comodules that are finitely generated as $\mathbbm{k}$-modules is similar to classes that have been considered in algebraic $K$-theory before in the literature \cite{HSwaldausen, pertower}. 

\begin{defn}\label{def: the second algebraic K-theory}
Let $\mathbbm{k}$ be a commutative ring with global dimension zero, and let $C$ be a $\mathbbm{k}$-coalgebra. Let $\bicomod{}{C}^\mathrm{fg}$ denote the category of right $C$-comodules that are finitely generated (and automatically projective) as $\mathbbm{k}$-modules. 
It is a Waldhausen category whose cofibrations are monomorphisms and whose weak equivalences are isomorphisms.
Let  
$G^c(C)$ denote the resulting algebraic $K$-theory spectrum $K(\bicomod{}{C}^\mathrm{fg})$.
\end{defn}

We show in Proposition \ref{prop: G-theory} that this algebraic $K$-theory is similar to the $G$-theory of rings.
A particular example of interest is the group ring $\mathbbm{k} \Gamma$ for $\Gamma$ a finite group, as $G(\mathbbm{k} \Gamma)$ is the $K$-theory of representations of $\Gamma$ in the category of $\mathbbm{k}$-modules, also known as Swan theory \cite{Swan}.  
The colinear Hattori--Stallings trace $\Endo_C(M)\rightarrow \coHH_0(C)$ on the identity defines a colinear rank. 
In \cite{coalgKtheory}, together with Gerhardt and Sor\'e, the second author show that the colinear rank is dual to the character of finite dimensional representations.

\begin{thm}[{Theorem \ref{thm: the second trace on K-theory}}]\label{Thm: second trace}
Let $\mathbbm{k}$ be a commutative ring with global dimension zero, and let $C$ be a $\mathbbm{k}$-coalgebra.
The colinear Hattori--Stallings rank induces a homomorphism of abelian groups
\[
G^c_0(C)\rightarrow \coHH_0(C).
\]
\end{thm}

In this paper, we also extend our shadow framework of Theorem \ref{theorem: main theorem} to the derived setting.
Our approach dualizes the arguments which show Hochschild homology is a shadow \cite{ponto2008fixed}, using an adaptation of the so-called Dennis--Waldhausen Morita argument. 
However, we are again faced with the following challenges. 

First, obtaining model structures on categories of comodules and coalgebras is a difficult problem \cite{left1,hkrs, left3}. Even in the case when model structures do exist, they may not be well-behaved \cite{perouxshipley, dkcoalg}. 
Nevertheless, the second author showed that there is a model structure on connective comodules over simply connected differential graded $\mathbbm{k}$-coalgebras up to quasi-isomorphism representing homotopy coherent connective comodules in $H\mathbbm{k}$-spectra \cite{connectivecomod}.

Second, constructing a derived cotensor product is challenging but necessary for a shadow structure on derived bicomodules. While the two-sided bar construction models the derived tensor product, a two-sided cobar construction should model the derived cotensor product. However, topological and algebraic cobar constructions each have issues: the former lacks comodule preservation, while the latter fails to be invariant under quasi-isomorphisms, see Remark \ref{rem: topological cobar vs algebraic cobar}. Restricting to simply connected dg-coalgebras, as in \cite{pertower, connectivecomod}, offers a workaround, though subtleties remain, such as totalization not commuting with tensoring, which we resolve here.

\begin{thm}[{Theorem \ref{theorem: coHH is shadow (derived)}}]\label{thm: intro derived shadow}
Let $\mathbbm{k}$ be a commutative ring with global dimension zero.
CoHochschild homology $\coHH$ defines a shadow on the bicategory of simply connected differential graded $\mathbbm{k}$-coalgebras and bicomodules with their derived relative cotensor product. 
\end{thm}

One of the defining properties of $K$-theory and Hochschild homology is invariance under Morita equivalences --- rings that have equivalent (possibly derived) categories of left modules. 
The dual notion for coalgebras and comodules is called \textit{Morita--Takeuchi equivalence}. 
Despite the pathological nature of model categories of comodules discussed above, we are able to show that coHochschild homology is invariant under Morita--Takeuchi equivalence in the derived context using our shadow framework.

\begin{thm}[{Theorem \ref{thm: coHH is morita invariant}}]\label{THEOREM E}
Let $C$ and $D$ be homotopically  Morita--Takeuchi equivalent simply connected differential graded $\mathbbm{k}$-coalgebras.
Then there is an isomorphism \[\coHH(C)\cong \coHH(D).\] 
\end{thm}

\subsection*{Vista}
The present paper is the first step of a program to extend the Dennis trace $K(R)\rightarrow \THH(R)$ to the coalgebraic setting. 
Our results propose natural candidates for the algebraic $K$-theory of a coalgebra.
The Hess--Shipley identification
in \cite{HSwaldausen} suggests that these new $K$-theories for coalgebras ``deloop'' the usual algebraic $K$-theories of rings.
In fact, given any ring spectrum $R$, its bar construction $BR$ is a coalgebra and modules over $R$ are equivalent to certain comodules over $BR$ under a Koszul duality \cite{gijs}. 
Therefore, obtaining computational tools to determine the algebraic $K$-theories of a coalgebra can shed new light on the usual $K$-theory of rings. Additionally, in forthcoming work with Brazelton, Calle, Chan, and Keenan, the second author intends to extend current methods in this paper to the bicategory of bicomodules over the coalgebra spectrum $\Sigma^\infty_+ X$.

In \cite{coalgKtheory}, Gerhardt, Sor\'e, and the second author show how the algebraic $K$-theories $K^c(C)$ and $G^c(C)$ relate further to the usual algebraic $K$-theory and $G$-theory of a ring. They also show the compatibility of the corank and colinear rank with the usual Hattori--Stallings trace and the character of finite dimensional representations.
In forthcoming work with Agarwal and Mehrle, the second author intends to extend the trace of Theorems \ref{thm: first trace} to higher level
\(
K_n^c(C) \longrightarrow (\coHH_n(C))^*,
\) for all $n \geq 0$.
 
\subsection*{Organization}
We finish this section with notations and definitions that are used throughout this paper. 
In section \ref{section :universal cotrace}, we construct the Hattori--Stallings cotrace and the colinear Hattori--Stallings trace  using only algebraic methods and motivation. This section does not require any knowledge of bicategories, shadows, or homotopy theory.
In order to be a self-contained paper, we will recall all necessary bicategorical definitions in later sections.
We introduce the different bicategories of bicomodules 
in Section \ref{section: bicategory of comodules}. Our key result, that coHochschild homology defines a bicategorical shadow, is proved in Section \ref{section: coHH as a shadow}. We then explore the notion of duality in bicomodules in Section \ref{section: duality}, which has been under-documented in the literature thus far.
In Section \ref{section: traces}, we provide a bicategorical description of the traces that appear in Section \ref{section :universal cotrace}.
In Section \ref{section: morita}, we show how Morita--Takeuchi equivalences are part of the bicategorical structures. In Appendix \ref{chap: appendix}, we carefully detail how to represent the derived cotensor product of bicomodules and explore what fibrant bicomodules are.

\subsection*{Acknowledgments}  
We thank Jonathan Campbell and Cary Malkiewich for discussions that sparked some of the initial ideas of this paper, as well as Teena Gerhardt and Gabriel Angelini-Knoll. 
We also thank Haldun Özgür Bayındır, Thomas Brazelton, Maxine Calle, David Chan, Andres Mejia, Svetlana Makarova, Mona Merling, and Manuel Rivera for helpful and enlightening conversations.
We also thank Maxine Calle and David Chan for their valuable feedback on earlier versions of this paper.

\subsection*{Notation} 
We begin by establishing notation that we use throughout this paper.
\begin{enumerate}
    \item The letter $\mathbbm{k}$ shall always denote a commutative ring with global dimension zero, i.e., a finite product of fields, or said differently, a commutative semisimple ring. Every $\mathbbm{k}$-module is projective and injective.
    \item Let $\bimod{}{\mathbbm{k}}$ denote the category of $\mathbbm{k}$-modules. We write the relative tensor product $\otimes_\mathbbm{k}$ as $\otimes$.
    \item Given $\mathbbm{k}$-modules $M$ and 
    $C$, and $\mathbbm{k}$-linear homomorphisms $\Delta\colon C\rightarrow C\otimes C$ and $\rho\colon M\rightarrow M\otimes C$, we may employ the \textit{Sweedler notation} in coalgebraic contexts:
\[
\Delta(c)=\sum_{i=1}^n {c_{(1)}}_i\otimes {c_{(2)}}_i=:\sum_{(c)}c_{(1)}\otimes c_{(2)},\, \text{and}\,  \rho(m)=\sum_{i=1}^n {m_{(0)}}_i\otimes {m_{(1)}}_i =:\sum_{(m)}m_{(0)}\otimes m_{(1)}.
\]
    \item Let $\Ck$ be the category of non-negative chain complexes of $\mathbbm{k}$-modules (graded homologically). The category is endowed with a symmetric monoidal structure. The tensor product of two chain complexes $X$ and $Y$ is defined by
\[
(X\otimes Y)_n=\bigoplus_{i+j=n} X_i\otimes_\mathbbm{k} Y_j,
\]
with differential given on homogeneous elements by
\[
d(x\otimes y)=dx \otimes y + (-1)^{\vert x \vert}x\otimes dy.
\]
We denote the tensor simply as $\otimes$. The monoidal unit is denoted $\mathbbm{k}$, which is the chain complex $\mathbbm{k}$ concentrated in degree zero.
\item We write $\hom_\mathbbm{k}(-,-)$ for the internal hom for both $\bimod{}{\mathbbm{k}}$ and $\Ck$. We write $M^*=\hom_\mathbbm{k}(M,\mathbbm{k})$, the linear dual.
\item Given any symmetric monoidal category $(\Cc, \otimes)$, we write $\tau\colon X\otimes Y\stackrel{\cong}\rightarrow Y\otimes X$, the natural symmetric isomorphism.
\end{enumerate}

\subsection*{General definitions}
We recall here the categorical definitions of coalgebras and comodules, and their related concepts.

\begin{defn}
Let $(\Cc, \otimes, \bI)$ be a symmetric monoidal category.
A \emph{coalgebra} $(C, \Delta, \varepsilon)$ in $\Cc$ consists of an object $C$ in $\Cc$ together with a \emph{coassociative} comultiplication $\Delta\colon C\rightarrow C\otimes C$, such that the following diagram commutes
\[
\begin{tikzcd}[column sep =large]
C\ar{r}{\Delta} \ar{d}[swap]{\Delta} &C\otimes C\ar{d}{\id_C\otimes \Delta}\\
C\otimes C \ar{r}{\Delta\otimes \id_C} & C\otimes C\otimes C,
\end{tikzcd}
\]
and that admits a \emph{counit} morphism $\varepsilon \colon C\rightarrow \mathbb{I}$ such that we have the following commutative diagram
\[
\begin{tikzcd}
C\otimes C \ar{r}{\id_C\otimes \varepsilon} & C\otimes \mathbb{I} \cong C \cong \mathbb{I}\otimes C & C\otimes C \ar{l}[swap]{\varepsilon \otimes \id_C}\\
& C.\ar[equals]{u}\ar[bend left]{ul}{\Delta} \ar[bend right]{ur}[swap]{\Delta} &
\end{tikzcd}
\]
The coalgebra is \emph{cocommutative} if the following diagram commutes
\[
\begin{tikzcd}
C\otimes C \ar{rr}{\tau}[swap]{\cong} && C\otimes C.\\
& C\ar{ul}{\Delta}\ar{ur}[swap]{\Delta} &
\end{tikzcd}
\]
Given a coalgebra $(C, \Delta, \varepsilon)$, we can define its opposite coalgebra $C^\op$ to be the coalgebra $(C, \tau \circ \Delta, \varepsilon)$. If $C$ is cocommutative, then $C=C^\op$.
\end{defn}

\begin{defn}\label{def: bicomodules}
Let $(\mathsf{C}, \otimes, \mathbb{I})$ be a symmetric monoidal category, and let $(C, \Delta, \varepsilon)$ be a coalgebra in $\mathsf{C}$. A \emph{right $C$-comodule} $(M, \rho)$ is an object $M$ in $\Cc$ together with a \emph{coassociative} and \emph{counital} right coaction morphism $\rho:M\rightarrow M\otimes C$ in $\Cc$, i.e., the following diagrams commute
\[
\begin{tikzcd}
M\ar{r}{\rho}\ar{d}[swap]{\rho} & M\otimes C \ar{d}{\rho \otimes \id_C}  & & M \ar{r}{\rho}\ar[equals]{ddr} & M\otimes C \ar{d}{ \id_M \otimes \varepsilon}\\
M\otimes C \ar{r}[swap]{ \id_M\otimes \Delta} &M\otimes C\otimes C, & & & M\otimes \bI  \ar{d}{\cong} \\[-5pt]
  & & &  & M. 
\end{tikzcd}
\]
We can similarly define \emph{a left $C$-comodule} $(M, \lambda)$. In fact, notice that a left $C$-comodule is a right $C^\op$-comodule. Given two coalgebras $C$ and $D$ in $\Cc$, a \textit{$(C, D)$-bicomodule} $(M, \lambda, \rho)$ is a left $C$-comodule $(M, \lambda)$ and a right $D$-comodule $(M, \rho)$ such that the following diagram commutes
\[
\begin{tikzcd}
M \ar{r}{\rho} \ar{d}[swap]{\lambda} & M\otimes D\ar{d}{\lambda\otimes \id_D}\\
C\otimes M \ar{r}{\id_C\otimes \rho} & C \otimes M \otimes D.
\end{tikzcd}
\]
A {morphism $(M, \rho_M)\rightarrow (N, \rho_N)$ of right $C$-comodules} is a morphism $f\colon M\rightarrow N$ in $\Cc$ such that the following diagram commutes
\[
\begin{tikzcd}
M\ar{r}{f}\ar{d}[swap]{\rho_M} & N\ar{d}{\rho_N}\\
M\otimes C \ar{r}{f\otimes \id_C} & N\otimes C.
\end{tikzcd}
\]
This morphism will be referred as a \emph{(right) $C$-colinear homomorphism}. A \textit{$(C,D)$-bicolinear homomorphism} is a morphism that is both a left $C$-colinear map and a right $D$-colinear homomorphism.
The notation $\bicomod{}{C}(\Cc)$ will denote the category of right $C$-comodules in $\Cc$. Similarly, we use $\bicomod{C}{}(\Cc)$ to denote the category of left $C$-comodules in $\Cc$, and $\bicomod{C}{D}(\Cc)$ for the category of $(C, D)$-bicomodules. Notice that we have isomorphisms of categories $\bicomod{C}{\bI}(\Cc)\cong\bicomod{C}{}(\Cc)$ and $\bicomod{\bI}{C}(\Cc)\cong \bicomod{}{C}(\Cc)$.
\end{defn}

\begin{defn}
We say a coalgebra $(C, \Delta, \varepsilon)$ is \emph{flat} in a symmetric monoidal category $(\Cc, \otimes, \bI)$ if the induced functor $C\otimes -:\Cc\rightarrow \Cc$ preserves equalizers when they exist.
\end{defn}

\begin{rem}\label{rem: bicomodules are right comodules (ordinary)}
Given a symmetric monoidal category $(\Cc, \otimes, \bI)$, and coalgebras $C$ and $D$ in $\Cc$,
note that we obtain an isomorphism of categories
\[
\bicomod{C}{D}(\Cc) \cong \bicomod{}{C^\op\otimes D}(\Cc).
\]
In particular, given a $(C,D)$-bicomodule $(M, \lambda, \rho)$, we obtain a right $(C^\op\otimes D)$-comodule via
\[
\begin{tikzcd}
M\ar{r}{\rho} & M\otimes D \ar{r}{\lambda\otimes \id_D} & C\otimes M \otimes D \ar{r}{\tau \otimes \id_D} & M\otimes C \otimes D.
\end{tikzcd}
\]
\end{rem}

\begin{defn}
Let $(\Cc,\otimes, \mathbb{I})$ be a symmetric monoidal category, let $C, D,$ and $E$ be coalgebras in $\Cc$, and let $(M, \lambda_M, \rho_M)$ be a $(D, C)$-bicomodule and $(N, \lambda_N, \rho_N)$ be a $(C, E)$-bicomodule. 
Define the \emph{cotensor product} $M \square_C N$ to be the following equalizer in $\Cc$:
\[
\begin{tikzcd}
M\square_C N \ar{r} & M\otimes N \ar[shift left]{r}{\rho_M\otimes \id_N} \ar[shift right]{r}[swap]{\id_M\otimes \lambda_N} & [20pt] M\otimes C \otimes N.
\end{tikzcd}
\]
If $D$ and $E$ are flat coalgebras, then
$M\square_C N$ is a $(D, E)$-bicomodule via
\[
\begin{tikzcd}[row sep=large]
M\square_C N \ar{r} \ar[dashed]{d}{\exists !} & M\otimes N \ar{d}{\lambda_X\otimes \rho_Y}\ar[shift left]{r}{\rho_M\otimes \id_N} \ar[shift right]{r}[swap]{\id_M\otimes \lambda_N} & [4em] M\otimes C \otimes N. \ar{d}{\lambda_M\otimes \id_C\otimes \rho_N} \\
D\otimes (M\square_C N)\otimes E \ar{r} & D\otimes M\otimes N\otimes E\ar[shift left]{r}{\id_D\otimes \rho_M\otimes \id_N\otimes \id_E} \ar[shift right]{r}[swap]{\id_D\otimes \id_M\otimes \lambda_N\otimes \id_E} & D\otimes M\otimes C \otimes N\otimes E.
\end{tikzcd}
\]
\end{defn}

\begin{defn}\label{defn: Morita-Takeuchi equivalence}
    Given a symmetric monoidal category $(\Cc,\otimes, \mathbb{I})$, we say two coalgebras $C$ and $D$ in $\Cc$ are \textit{Morita--Takeuchi equivalent} if the categories ${}_C\comod(\Cc)$ and ${}_D\comod(\Cc)$ are equivalent.
\end{defn}

\begin{defn}
    Given a closed symmetric monoidal category $(\Cc,\otimes, \mathbb{I}, [-,-])$, and a coalgebra $C$, we can provide an enrichment of $\comod_C(\Cc)$ as follows. Given right $C$-comodules $M$ and $N$, define $\hom_C(M,N)$ as the equalizer in $\Cc$
    \[
    \begin{tikzcd}
    \hom_C(M,N) \ar{r} & {[M,N]} \ar[shift left]{r} \ar[shift right]{r} & {[M, N\otimes C]},
    \end{tikzcd}
    \]
    where the first parallel map is induced  by the coaction $N\rightarrow N\otimes C$ while the second is defined by forgetful-cofree adjointness. We denote $\Endo_C(M)=\hom_C(M,M)$. We can similarly define ${}_C\hom(-,-)$ and ${}_C\Endo(-)$ for left $C$-comodules.
\end{defn}

\begin{defn}\label{def: cotor}
    Let $\mathsf{A}$ be an abelian symmetric monoidal category with enough injective objects. Suppose $C$ is a flat coalgebra in $\mathsf{A}$. Then $\bicomod{C}{}(\mathsf{A})$ also has enough injective objects. If $M$ is a right $C$-comodule in $\mathsf{A}$, then the functor $M\square_C-\colon \bicomod{C}{}(\mathsf{A})\rightarrow \mathsf{A}$ is left exact between abelian categories with enough injectives and thus we can right derive it. We denote by $\cotor^i_C(M,-)$ the  $i^{th}$ right derived functor.
    Similarly, if $N$ is a left $C$-comodule, we can right derive the functor $-\square_C N\colon\bicomod{}{C}(\mathsf{A})\rightarrow \mathsf{A}$ and obtain functors $\cotor^i_C(-, N)$. Per usual, we can check that $\cotor^i_C(M,N)$ is unambiguous. 
\end{defn}

\section{Coalgebraic variations of the Hattori--Stallings trace}\label{section :universal cotrace}

We introduce the Hattori--Stallings cotrace (Definition \ref{def: the hattori-stallings cotrace}) and the colinear Hattori--Stallings trace (Definition \ref{def: colinear Hattori--Stallings trace}) using only (co)algebraic methods.
We then obtain trace maps on $K$-theories in Theorem \ref{thm: the corank as a theorem} and Theorem \ref{thm: the second trace on K-theory}. 
We first begin by briefly recalling how the Hattori--Stallings trace is built so that the reader sees the analogies between the algebraic setting and coalgebraic setting.

\subsection{The Hattori--Stallings trace}

We recall \cite{hattori, stallings}.
Let $R$ be a ring, not necessarily commutative, and let $A$ be an abelian group. Denote the ring of square matrices of size $n$ with coefficients in $R$ by $\mathcal{M}_n(R)$.

\begin{defn}\label{def: trace function on a ring}
A \textit{trace function on $R$ with values in $A$} is a collection of $\mathbb{Z}$-linear homomorphisms denoted $T_n\colon \mathcal{M}_n(R)\rightarrow A$ for each $n\geq 1$, such that $T_n(MN)=T_m(NM)$, where we chose $M$ to be a $n\times m$-matrix, and $N$ a $m\times n$-matrix.
\end{defn}

\noindent 
Each map $T_n\colon\mathcal{M}_n(R)\rightarrow A$ is entirely determined by the map $T_1\colon R=\mathcal{M}_1(R)\rightarrow A$ \cite[1.4,1.5]{stallings}, as $T_n(M)=\sum_{i=1}^n T_1(m_{ii})$, where $M=[m_{ij}]$.
Therefore we can rephrase a \textit{trace} on $R$ with values in $A$ to be a $\mathbb{Z}$-linear homomorphism $T\colon R\rightarrow A$ such that $T(rs)=T(sr)$, for all $r,s\in R$, i.e., the following diagram commutes
\[
\begin{tikzcd}
    R\otimes_\mathbb{Z} R \ar[shift left]{r}{m}
 \ar[shift right]{r}[swap]{m\circ \tau} & R \ar{r}{T} & A,\end{tikzcd}
\]
where $m$ denotes the multiplication on $R$.

If we let $[R,R]=\{rs-sr| r,s\in R\}$ be the subgroup of commutators, then we obtain a universal trace $\tr\colon R\rightarrow R/[R,R]$ by the quotient homomorphism, and every trace $T\colon R\rightarrow A$ is uniquely determined by a $\mathbb{Z}$-linear homomorphism $R/[R,R]\rightarrow A$:
\[
\begin{tikzcd}
    R \ar{r}{T} \ar{d}[swap]{\tr} & A.\\
    R/{[R, R]} \ar[dashed]{ur}[swap]{\exists !}
\end{tikzcd}
\]
Given $M$ a finitely generated free right $R$-module, then an $R$-linear endomorphism $f\colon M\rightarrow M$ is represented by a square matrix on $R$, and thus any trace $T\colon R\rightarrow A$ defines a trace $T_M\colon\Endo_R(M)\rightarrow A$ as we have the isomorphism of rings $\Endo_R(M)\cong \mathcal{M}_n(R)$, where $n$ is the rank of $M$.

More generally, if $M$ is a finitely generated projective right $R$-module, then there exists another projective module $N$ such that $M\oplus N$ is a free $R$-module. 
Therefore if we extend an $R$-linear endomorphism $f\colon M\rightarrow M$ to the endomorphism $f+0\colon M\oplus N\rightarrow M\oplus N$, we again obtain a square matrix, and thus any trace $T\colon R\rightarrow A$ defines a trace $T_M\colon\Endo_R(M)\rightarrow A$. 

\begin{defn}\label{def: hattori-stallings definition}
Let $R$ be a ring, and $M$ be a finitely generated projective right $R$-module.
The \textit{Hattori--Stallings trace} on $M$ is the trace induced by the universal trace $\tr\colon R\rightarrow R/[R, R]$ on the endomorphisms of $M$,
\[
\tr_M\colon \Endo_R(M)  \longrightarrow R/[R,R].
\]
The \textit{Hattori--Stallings rank} of $M$ is then the trace of the identity endomorphism on $M$. The abelian group $R/[R,R]$ is isomorphic to $\HH_0(R)=R\otimes_{R\otimes R^\op} R$, the zeroth \textit{Hochschild homology of $R$}.
\end{defn}

The definition above depends on the choice of generators of $M$ and a section $R^{\oplus n}\rightarrow M$. However, the Hattori--Stallings trace can be defined coordinate-free as follows. 
Given a finitely generated projective right $R$-module $M$, we have an isomorphism
\[
M\otimes_R \hom_R(M, R) \stackrel{\cong}\longrightarrow \hom_R(M, M)=\Endo_R(M),
\]
and thus the Hattori--Stallings trace is the composition
\[
\begin{tikzcd}
    \Endo_R(M) & \ar{l}[swap]{\cong} M\otimes_R \hom_R(M,R) \ar{r}{\ev} & R \ar{r}{\tr} &\HH_0(R).
\end{tikzcd}
\]

Let $\mathrm{P}(R)$ be the set of isomorphism classes of finitely generated projective right $R$-modules. It is a commutative monoid with respect to direct sum. 
Any trace $T\colon R\rightarrow A$ defines an additive homomorphism with its rank,
\begin{align*}
    \mathrm{P}(R) & \longrightarrow A\\
    [M] & \longmapsto T_M(\id_M).
\end{align*}
Let $K_0(R)$ be the group completion of 
$\mathrm{P}(R)$. Then the additive map factors through a $\mathbb{Z}$-linear homomorphism
\[
T\colon K_0(R) \longrightarrow A.
\]
Choosing $T$ to be the universal trace $\tr\colon R\rightarrow \HH_0(R)$ defines the \textit{Hattori--Stallings rank} $K_0(R)\rightarrow \HH_0(R)$, which is the zeroth level of the Dennis trace $K_*(R)\rightarrow \HH_*(R)$.

\subsection{The Hattori--Stallings cotrace}
We wish to dualize the approach of Hattori and Stallings in the coalgebraic context. Just as the group $\HH_0(R)$ is the universal home for traces, we introduce a group $\coHH_0(C)$ as the universal home for so-called cotraces.
Throughout the rest of the section, let $(C, \Delta, \varepsilon)$ be a $\mathbbm{k}$-coalgebra, not necessarily cocommutative. 

\begin{defn}\label{def: cotrace on a coalgebra}
Let $V$ be a $\mathbbm{k}$-module.
A $\mathbbm{k}$-linear homomorphism $T\colon V\rightarrow C$ is said to be a \textit{cotrace on $C$ from $V$} if it is \textit{cocyclic} in the sense that the following diagram commutes:
\[
\begin{tikzcd}
V \ar{r}{T} & C \ar[shift left]{r}{\Delta} \ar[shift right]{r}[swap]{\tau\circ \Delta} & C\otimes C.
\end{tikzcd}
\]   
In other words, using the Sweedler notation $\Delta(c)=\sum_{(c)}c_{(1)}\otimes c_{(2)}$, the $\mathbbm{k}$-linear homomorphism $T\colon V\rightarrow C$ is a cotrace if for all $v\in V$:
\[
\sum_{(T(v))} T(v)_{(1)}\otimes T(v)_{(2)} = \sum_{(T(v))} T(v)_{(2)}\otimes T(v)_{(1)}.
\] 
\end{defn}

\begin{defn}
The \textit{cocommutator} submodule $\lang C\rang $ is the kernel in $\mathbbm{k}$-modules
\[
\begin{tikzcd}
\lang C\rang \ar[hook]{r}{\cotr} & C \ar[shift left]{r}{\Delta} \ar[shift right]{r}[swap]{\tau\circ \Delta} & C\otimes C.
\end{tikzcd}
\]
The \textit{universal cotrace on $C$} is the resulting cotrace $\cotr\colon\lang C \rang \hookrightarrow C$. Every cotrace $T\colon V\rightarrow C$ is uniquely determined by a $\mathbbm{k}$-linear homomorphism $V\rightarrow \lang C \rang$:
\[
\begin{tikzcd}
    & \lang C \rang \ar[hook]{d}{\cotr}\\
    V \ar[dashed]{ur}{\exists !} \ar{r}[swap]{T} & C
\end{tikzcd}
\]
Notice that $\lang C \rang$ is isomorphic to the zeroth coHochschild homology $\coHH_0(C):=C\square_{C\otimes C^{\op}} C$.
We recall the definition of $\coHH_q(C)$ more generally in Definition \ref{def: cohochschild homology underived}.
\end{defn}

Recall that any trace $T\colon R\rightarrow A$ can be extended on square matrices $T_n\colon \mathcal{M}_n(R)\rightarrow A$  by assigning $T_n(M)=\sum_{i=1}^n T(m_{ii})$, where $M=[m_{ij}]$ . 
We would like to have a similar property for cotraces.

\begin{defn}\label{def: comatrices}
Given $n,m\geq 1$, define the \textit{$(n,m)$-comatrices $\mathcal{M}_{n\times m}^c(\mathbbm{k})$} to be usual $\mathbbm{k}$-module $\mathcal{M}_{n\times m}(\mathbbm{k})$ of matrices of size $n\times m$ with coefficients in $\mathbbm{k}$, but endowed with the following two operations. 
Let $\{E_{ij}\}$ be the canonical basis of $\mathcal{M}_{n\times m}(\mathbbm{k})$ composed of the elementary matrices. 
Define for all $r\geq 1$ the $\mathbbm{k}$-linear homomorphism
\begin{align*}
    \Delta\colon \mathcal{M}^c_{n\times m}(\mathbbm{k}) & \longrightarrow \mathcal{M}^c_{n\times r}(\mathbbm{k})\otimes \mathcal{M}^c_{r\times m}(\mathbbm{k})\\
    E_{ij} & \longmapsto \sum_{\ell=1}^r E_{i\ell}\otimes E_{\ell j},
\end{align*}
and the $\mathbbm{k}$-linear homomorphism
\begin{align*}
    \varepsilon\colon \mathcal{M}^c_{n\times n}(\mathbbm{k}) & \longrightarrow \mathbbm{k}\\
    E_{ij} & \longmapsto \delta_{ij},
\end{align*}
where $\delta_{ij}$ is the Kronecker symbol.
The maps $\Delta$ and $\varepsilon$ are coassociative and counital in the sense that the comatrices $\mathcal{M}_{n\times m}^c(\mathbbm{k})$ constitute the morphisms of a category $\mathrm{Mat}^c(\mathbbm{k})$ enriched in $\mathrm{Mod}_\mathbbm{k}^\mathrm{op}$, the opposite category of $\mathbbm{k}$-modules. In fact, the underlying category of $\mathrm{Mat}^c(\mathbbm{k})$ is anti-equivalent to the usual category $\mathrm{Mat}(\mathbbm{k})$ of matrices.
Denote by $\mathcal{M}_{n\times m}^c(C)=C\otimes \mathcal{M}_{n\times m}^c(\mathbbm{k})$ the \textit{$(n,m)$-comatrices with values in $C$.}
Combining with the comultiplication and counit of $C$, the $\mathbbm{k}$-modules $\mathcal{M}_{n\times m}^c(C)$ extend the previous operations to $\mathcal{M}_{n\times m}^c(C)\rightarrow  \mathcal{M}^c_{n\times r}(C)\otimes \mathcal{M}^c_{r\times m}(C)$ and $\mathcal{M}_{n\times m}^c(C)\rightarrow \mathbbm{k}$, which are again coassociative and counital.
When $n=m$, we denote the square comatrices by $\mathcal{M}_n^c(C)$; these are $\mathbbm{k}$-coalgebras and $\mathcal{M}_n^c(\mathbbm{k})$ is anti-equivalent to the usual matrix ring $\mathcal{M}_n(\mathbbm{k})$.
Moreover, by \cite[12.20]{brzezinski2003corings}, the coalgebras $C$ and $\mathcal{M}_n^c(C)$ are Morita--Takeuchi equivalent for all $n\geq 1$ (Definition \ref{defn: Morita-Takeuchi equivalence}).
\end{defn}

Given a cotrace $T\colon V\rightarrow C$, it can be extended to a $\mathbbm{k}$-linear homomorphism $T_n\colon V\rightarrow \mathcal{M}^c_n(C)$ by assigning $v\mapsto T(v)\otimes I_n$, where $I_n$ is the identity matrix of size $n$. 
These homomorphisms $T_n\colon V\rightarrow \mathcal{M}^c_n(C)$ are cocyclic in the sense that the following diagram commutes:
\begin{equation}\label{eq: coyclic cotraces on matrices}
\begin{tikzcd}
	& {\mathcal{M}_n^c(C)} & {\mathcal{M}^c_{n\times m}(C)\otimes \mathcal{M}^c_{m\times n}(C)} \\
	V \\
	& {\mathcal{M}_m^c(C)} & {\mathcal{M}^c_{m\times n}(C)\otimes \mathcal{M}^c_{n\times m}(C).}
	\arrow["{T_n}", from=2-1, to=1-2]
	\arrow["\Delta", from=1-2, to=1-3]
	\arrow["\tau", from=1-3, to=3-3]
	\arrow["{T_m}"', from=2-1, to=3-2]
	\arrow["\Delta"', from=3-2, to=3-3]
\end{tikzcd}
\end{equation}
The commutative diagram captures the dual property $T_n(MN)=T_m(NM)$ of traces.

\begin{lemma}
A cotrace $T\colon V\rightarrow C$ is uniquely determined by a collection of $\mathbbm{k}$-linear homomorphisms $T_n\colon  V\rightarrow \mathcal{M}^c_n(C)$ satisfying \eqref{eq: coyclic cotraces on matrices} for which  $T_n(v)=T(v)\otimes I_n$  for  all  $v\in V$ and $n\geq 1$.
\end{lemma}

Just as a trace $T\colon R\rightarrow A$ can be extended to $T_M\colon \Endo_R(M)\rightarrow A$ for $M$ a finitely generated and projective right $R$-module, we now would like to further extend the cotraces.

\begin{defn}\label{def: finitely cogenerated}
We say a left $C$-comodule $M$ is \textit{finitely cogenerated} if there is a $C$-colinear monomorphism $M\hookrightarrow C^{\oplus n}$ for some $n\geq 0$.
\end{defn}

A left $C$-comodule $M$ is \textit{injective} in the usual abelian category sense: any $C$-colinear monomorphisms $M\hookrightarrow N$ splits.
It is immediate to obtain the following result.

\begin{lemma}
For any finitely cogenerated and injective left $C$-comodule $M$, there is another finitely cogenerated and injective left $C$-comodule $N$ such that $M\oplus N\cong C^{\oplus n}$ for some $n\geq 0$.
\end{lemma}

Finitely cogenerated and injective left $C$-comodules are the dualizable comodules in a sense that we make precise in Section \ref{section: duality}.
From Definition \ref{def: quasi-finite comodules} and Example \ref{ex: dualizable comodules with cofree}, given any finitely cogenerated left $C$-comodule $M$, the functor $M\otimes-\colon \bimod{}{\mathbbm{k}}\rightarrow \bicomod{C}{}$ admits a left adjoint $h_C(M,-)\colon\bicomod{C}{}\rightarrow \bimod{}{\mathbbm{k}}$ called the \textit{cohom} functor. 
In other words, given any left $C$-comodule $N$, the cohom $h_C(M,N)$ is the universal $\mathbbm{k}$-module providing a natural $\mathbbm{k}$-linear isomorphism
\[
\bimod{}{\mathbbm{k}}\Big( h_C(M, N), V\Big)\cong \bicomod{C}{}\Big(N, M\otimes V\Big),
\]
for any $\mathbbm{k}$-module $V$. In particular, for $V=\mathbbm{k}$, we obtain the isomorphism $h_C(M,N)^*\cong  {}_C\hom(N, M)$.
Given any left  $C$-comodule $N$, there is a $C$-colinear homomorphism $\eta\colon N\rightarrow M\otimes h_C(M,N)$ called the \textit{coevaluation on $N$ with $M$}, adjoint of the identity on $h_C(M,N)$.

\begin{defn}[{\cite[12.9]{brzezinski2003corings}}]\label{def: coendomorphism coalgebra}
    Let $M$ be a finitely cogenerated left $C$-comodule. The \textit{coendomorphism coalgebra on $M$} denoted $e_C(M)$ is the $\mathbbm{k}$-module $h_C(M,M)$ endowed with the comultiplication
    \[
    h_C(M,M)  \longrightarrow h_C(M, M) \otimes h_C(M,M),
    \]
    which is the adjoint of the $C$-colinear homomorphism
    \[
    \begin{tikzcd}
        M \ar{r}{\eta} & M\otimes h_C(M, M) \ar{r}{\eta\otimes 1} & M\otimes h_C(M,M)\otimes h_C(M, M).
    \end{tikzcd}
    \]
    The counit $h_C(M,M)\rightarrow \mathbbm{k}$ is the adjoint of the identity of the $C$-colinear homomorphism $M\stackrel{\cong}\rightarrow M\otimes \mathbbm{k}$.
\end{defn}

We have an isomorphism of $\mathbbm{k}$-coalgebras $e_C(C^{\oplus n})\cong \mathcal{M}_n^c(C)$ \cite[12.20]{brzezinski2003corings}.  More generally, we obtain the identification $h_C(C^{\oplus n}, C^{\oplus m})\cong \mathcal{M}_{n\times m}^c(C)$. The linear dual $e_C(M)^*$ is anti-equivalent to the $\mathbbm{k}$-algebra ${}_C\Endo(M)$, see \cite[12.9]{brzezinski2003corings}.

\begin{lemma}\label{lem: extending  fin cogenerated to cofree}
    Given a finitely cogenerated and injective left $C$-comodule $M$,  there exists a $\mathbbm{k}$-linear homomorphism $\mathcal{M}_n(C)\rightarrow e_C(M)$ that becomes precisely the following monomorphism after taking the linear dual:
    \begin{align*}
        {}_C\Endo(M) & \longrightarrow {}_C\Endo(C^{\oplus n})\\
        \left( M\stackrel{f}\rightarrow M\right) & \longmapsto \left(M\oplus N \stackrel{f\oplus 0}\rightarrow M\oplus N \right), 
    \end{align*}
    for some choice of $N$ such that $M\oplus N\cong C^{\oplus n}$.
\end{lemma}

\begin{proof}
    As $M$ is finitely cogenerated, there exists $n \geq 0$ such that $M \hookrightarrow C^{\oplus n}$. Since $M$ is injective, we can choose a $C$-colinear section $C^{\oplus n}\rightarrow M$.
    Define $h_C(C^{\oplus n}, C^{\oplus n})=\mathcal{M}_n(C)\rightarrow e_C(M)=h_C(M,M)$ as the adjoint the composition
    \[
    C^{\oplus n} \longrightarrow M\otimes h_C(M,C^{\oplus n}) \longrightarrow C^{\oplus n}\otimes  h_C(M, C^{\oplus n}) \longrightarrow C^{\oplus n}\otimes  h_C(M, M).
    \]
    The first map is induced by coevaluation on $C^{\oplus n}$, the second map is induced by $M\hookrightarrow C^{\oplus n}$, and the third map is induced by applying $h_C(M,-)$ on the section $C^{\oplus n}\rightarrow M$.
\end{proof}

Therefore, given a finitely cogenerated and injective left $C$-comodule $M$, any cotrace $T\colon V\rightarrow C$ extends to $T_M\colon V\rightarrow e_C(M)$ as the composite
\[
V\stackrel{T_n}\longrightarrow \mathcal{M}_n(C) \longrightarrow e_C(M),
\]
for some $n\geq 1$.
Applying the linear dual yields a $\mathbbm{k}$-linear homomorphism \[T_M\colon  {}_C\Endo(M)\cong  e_C(M)^*\longrightarrow V^*.\] 
As $C^{\oplus n}$ is the cofree $C$-comodule on $\mathbbm{k}^{\oplus n}$, we obtain the identification
\[
{}_C\Endo(C^{\oplus n})\cong \hom_\mathbbm{k}(C^{\oplus n}, \mathbbm{k}^{\oplus n}) \cong \hom_\mathbbm{k}(C, \mathcal{M}_n(\mathbbm{k})),
\]
which sends the extended endomorphism $f\oplus 0$ on $C^{\oplus n}$ of a $C$-colinear endomorphism $f\colon M\rightarrow M$ to a $\mathbbm{k}$-linear homomorphism $C\rightarrow \mathcal{M}_n(\mathbbm{k})$ that we denote $c\mapsto f_c$.
Unwinding the definitions, we obtain the following formula for $T_M\colon {}_C\Endo(M)\longrightarrow V^*$,
\begin{equation*}
    T_M(f)(v)=\sum_{(T(v))}\varepsilon\left( T(v)_{(1)}\right)\tr\left(f_{T(v)_{(2)}}\right),
\end{equation*}
for any $v\in V$ and $C$-colinear endomorphism $f\colon M\rightarrow M$. 
Here $\varepsilon\colon C\rightarrow \mathbbm{k}$ denotes the counit, and the comultiplication $C\rightarrow C\otimes C$ is denoted $c\mapsto\sum_{(c)} c_{(1)}\otimes c_{(2)}$ using Sweedler notation.

\begin{defn}\label{def: the hattori-stallings cotrace}
Let $M$ be a finitely cogenerated and injective left $C$-comodule.
The \textit{Hattori--Stallings cotrace on $M$} is the cotrace induced by universal cotrace $\cotr\colon \coHH_0(C)\hookrightarrow C$ on the endomorphisms of $M$
\[
\cotr_M\colon {}_C\Endo(M)\longrightarrow \coHH_0(C)^*.
\]
Explicitly, it is given by the formula
\[
\cotr_M(f)(c)=\sum_{(c)}\varepsilon\left( c_{(1)}\right)\tr\left(f_{c_{(2)}}\right),
\]
for all $c\in \coHH_0(C)$.
The \textit{Hattori--Stallings corank on $M$} is the cotrace on the identity.
\end{defn}

\begin{rem}
    A priori, the definition of the Hattori--Stallings cotrace may depend on the choice of the section $C^{\oplus n}\rightarrow M$, but we shall see in Proposition \ref{prop: hs-cotrace is a bicategorical trace} that the cotrace is a bicategorical  trace in a shadow framework and thus has a coordinate-free formulation, just as the usual Hattori--Stallings trace does. 
\end{rem}

\begin{rem}
We can provide a very explicit formulation of the Hattori--Stallings corank on a finitely cogenerated injective left $C$-comodule $M$.
There exists a comodule $N$ such that $M\oplus N\cong C^{\oplus n}$. 
Denote by $e^M\colon C^{\oplus n}\rightarrow C^{\oplus n}$ the $C$-colinear projection $M\oplus N\rightarrow M$.
Denote the composition
 \[
 \begin{tikzcd}
     C \ar{r}{\iota_j} & C^{\oplus n} \ar{r}{e^M} & C^{\oplus n} \ar{r}{p_i} & C
 \end{tikzcd}
 \]
by $e^M_{ij}\colon C\rightarrow
 C$, where $\iota_j$ is the $j$-th inclusion and $p_i$ is the $i$-th projection. 
Then, by adjunction, the map $e^M$ corresponds to a $\mathbbm{k}$-linear homomorphism 
\begin{align*}
    C & \longrightarrow \mathcal{M}_n(\mathbbm{k})\\
    c & \longmapsto [\varepsilon(e^M_{ij}(c))],
\end{align*}
where $\varepsilon\colon C\rightarrow \mathbbm{k}$ is the counit.
By our choice of $e^M$, there exists $n_M\geq 1$ such that $\cotr_M(\id_M)\in (\coHH_0(C))^*$ is determined by
\begin{equation*}\label{eq: cotrace formula}
\cotr_M(\id_M)(c)=\sum_{(c)}\varepsilon(c_{(1)})\left(\sum_{i=1}^n\varepsilon(e^M_{ii}(c_{(2)}))\right)=\sum_{i=1}^{n_M} \sum_{(c)} \varepsilon(c_{(1)})\varepsilon(c_{(2)})=\sum_{i=1}^{n_M} \varepsilon(c)
\end{equation*}
for any $c\in \coHH_0(C)$, where $\Delta\colon C\rightarrow C\otimes C$ is denoted $\Delta(c)=\sum_{(c)}c_{(1)}\otimes c_{(2)}$ using Sweedler notation.
\end{rem}

\begin{rem}\label{rem: Quillen}
In \cite{quillencochain}, Quillen already established $\coHH_0(C)$ as the universal home for cotrace maps with universal cotrace $\coHH_0(C)\hookrightarrow C$ as we do here.
He identified $\coHH_0(C)$ with cyclic homology $\operatorname{HC}_*$ in certain cases. In \cite{ddaniel}, Quillen utilized the universal cotrace to extend trace maps on a ring $R$ to traces $K_0(A)\rightarrow \operatorname{HC}_{2n}(A)$, where $A$ is a quotient algebra of $R$. 
However, Quillen did not extend cotraces $V\rightarrow C$ to cotraces $V\rightarrow \mathcal{M}_n^c(C)$ nor $V\rightarrow e_C(M)$. Our Hattori--Stallings cotrace is an extension of the universal cotrace $\coHH_0(C)\hookrightarrow C$.
\end{rem}

Let $\mathrm{I}(C)$ be the isomorphism class of finitely cogenerated injective left $C$-comodules. It is a commutative monoid with respect to the direct sum.  
Any cotrace $T\colon V\rightarrow C$, defines an additive homomorphism with its rank:
\begin{align*}
\mathrm{I}(C) &\longrightarrow V^*\\
[M] & \longmapsto T_M(\id_M).
\end{align*}
Denote $K_0^c(C)$ the group completion of $\mathrm{I}(C)$.
Since the abelian category $\bicomod{C}{}^\mathrm{fc,inj}$ is split exact, then $K_0^c(C)$ is indeed equivalent to the $\pi_0$-group of the algebraic $K$-theory spectrum of Definition \ref{def: the first K-theory}.
Therefore, the above additive map factors through a $\mathbb{Z}$-linear homomorphism:
\[
T\colon K_0^c(C)\longrightarrow V^*.
\]
Choosing $T$ to be the universal cotrace $\cotr\colon \coHH_0(C)\hookrightarrow C$ yields the following result.

\begin{thm}\label{thm: the corank as a theorem}
    Let $C$ be a $\mathbbm{k}$-coalgebra.
    The Hattori--Stallings corank defines a homomorphism of abelian groups:
    \[
    K_0^c(C)\longrightarrow \coHH_0(C)^*.
    \]
\end{thm}

\begin{ex}
    If $C=\mathbbm{k}$, then $K^c(\mathbbm{k})$ is equivalent to the algebraic $K$-theory $K(\mathbbm{k})$ of $\mathbbm{k}$ as a ring, and thus $K_0^c(\mathbbm{k})\cong \mathbb{Z}^{\times n}$ as $\mathbbm{k}$ is a finite product of fields, and $\coHH_0(C)^*=\mathbbm{k}$, and the corank $\mathbb{Z}^n\rightarrow \mathbbm{k}$ is induced by the unique ring homomorphism from $\mathbb{Z}$ on each field in the direct summand of $\mathbbm{k}$.
\end{ex}

\begin{rem}[Eilenberg Swindle]
   Just as in the case for rings, we needed to restrict to finitely cogenerated comodules, as otherwise $K_0^c(C)=0$. Indeed, let $\mathbbm{k}^{\oplus \infty}$ be the $\mathbbm{k}$-vector space with countably infinite basis. Let $C^{\oplus \infty}=C\otimes \mathbbm{k}^{\oplus \infty}$ be an infinitely cogenerated cofree left $C$-comodule. 
   Then we obtain isomorphisms of left $C$-comodules
   \begin{align*}
       C^{\oplus \infty} & = C\otimes \mathbbm{k}^{\oplus \infty}\\
       & \cong  C\otimes \big( \mathbbm{k}^{\oplus n} \oplus \mathbbm{k}^{\oplus n} \oplus \cdots \big)\\
       &  \cong (C\otimes \mathbbm{k}^{\oplus n}) \oplus (C\otimes \mathbbm{k}^{\oplus n}) \oplus \cdots \\
       & \cong C^{\oplus n} \oplus C^{\oplus n} \oplus \cdots.
   \end{align*}
   Therefore if $M\oplus N\cong C^{\oplus n}$, then we obtain
   \begin{align*}
       M\oplus C^{\oplus \infty} & \cong M\oplus (N\oplus M) \oplus (N\oplus M) \oplus \cdots\\
       & \cong (M\oplus N)\oplus (M\oplus N)\oplus \cdots\\
       & \cong C^{\oplus \infty}.
   \end{align*}
   Hence, in $K_0^c(C)$ we would get $[M]=0$.
\end{rem}

\begin{prop}\label{prop: K-zero is Z for conil}
    If $C$ is a conilpotent coaugmented $\mathbbm{k}$-coalgebra,  then $K_0^c(C)\cong \mathbb{Z}$.
\end{prop}

\begin{proof}
  By \cite[2.6]{leo}, injective $C$-comodules are all cofree when $C$ is a conilpotent coaugmented coalgebra.
  Thus $\mathrm{I}(C)\cong \mathbb{N}$, and the result follows.
\end{proof}

\begin{prop}\label{prop: K(C) is a ring spectrum}
    If $C$ is a cocommutative $\mathbbm{k}$-coalgebra, then $K_0^c(C)$ is a commutative ring.
    In fact, if  $C$ is cocommutative, then $K^c(C)$ is a ring spectrum.
\end{prop}

\begin{proof}
    If $C$ is cocommutative, then the cotensor product forms a symmetric monoidal structure on $C$-comodules.
    The cotensor product of injective $C$-comodules is injective \cite[Proposition 1]{doi1981homological}. The cotensor product of finitely cogenerated $C$-comodules is finitely cogenerated. Therefore $K_0^c(C)$ is a commutative ring via cotensor product of $C$-comodules.
    As the the cotensor product preserves monomorphisms in each variable, then the Waldhausen category ${}_C\comod^\mathrm{fc,inj}$ is symmetric monoidal and thus $K^c(C)$ is a ring spectrum \cite[2.8]{blum-man}.  
\end{proof}

\subsection{The colinear Hattori--Stallings trace}\label{subsec: colinear trace}
Another approach is to obtain a coalgebraic refinement of the Hattori--Stallings trace as we now explain.
We can extend the notion of a trace $T\colon R\rightarrow A$ to a trace on any $(R,R)$-bimodule $P$ as a $\mathbb{Z}$-linear homomorphism $T\colon P\rightarrow A$ such that $T(rx)=T(xr)$ for all $r\in R$, $x\in P$. 
A universal trace is then given by $\tr\colon P\rightarrow \HH_0(R, P):=R\otimes_{R\otimes R^{\op}} P$. 
If $M$ is a right $R$-module, then $\hom_R(M,R)\otimes_\mathbb{Z} M$ is an $(R,R)$-bimodule. 

\begin{defn}[{\cite[2.3.4]{ponto2008fixed}}]\label{def: twisted hattori-stallings trace}
Let $M$ be a finitely generated projective right $R$-module, and $P$ an $(R,R)$-bimodule.
The \textit{twisted Hattori--Stallings trace} $\hom_R(M,M\otimes_R P)\rightarrow \HH_0(R,P)$ is defined on a right $R$-linear homomorphism $f\colon M\rightarrow M\otimes_R P$, referred to as an \textit{endomorphism on $M$ twisted by $P$}, as the image of $f$ under the composition
\[
\begin{tikzcd}
    \hom_R(M, M\otimes_R P) & \ar{l}[swap]{\cong} (M\otimes_R P )\otimes_R \hom_R(M,R) \ar{r}{\ev} & P \ar{r}{\tr} & \HH_0(R, P).
\end{tikzcd}
\]
\end{defn}

Given a $\mathbbm{k}$-coalgebra $C$, we now choose $R=\mathbbm{k}$ and $P=C$ as a $(\mathbbm{k}, \mathbbm{k})$-bimodule, then the twisted Hattori--Stallings trace becomes the $\mathbbm{k}$-linear homomorphism $\tr\colon \hom_\mathbbm{k}(M, M\otimes C)\rightarrow \HH_0(\mathbbm{k}, C)=C$.
If we require $M$ to be a right $C$-comodule, finitely generated as a $\mathbbm{k}$-module, with coaction $\rho\colon M\rightarrow M\otimes  C$, we define $\tr^C\colon \Endo_C(M)\rightarrow C$ by $\tr^C(f)=\tr(\rho\circ f)$.

\begin{prop}\label{prop: colinear is a cotrace}
    The $\mathbbm{k}$-linear homomorphism $\tr^C\colon\Endo_C(M)\rightarrow C$ is a cotrace. In particular, it factors uniquely through the universal cotrace:
    \[
\begin{tikzcd}[row sep=large]
    & \coHH_0(C) \ar[hook]{d}{\cotr}\\
    \Endo_C(M) \ar[dashed]{ur}{\exists ! \tr^C} \ar{r}[swap]{\tr^C} & C.
\end{tikzcd}
\]
\end{prop}

\begin{defn}\label{def: colinear Hattori--Stallings trace}
Given a right $C$-comodule $M$ that is finitely generated as a $\mathbbm{k}$-module, we call the induced  $\mathbbm{k}$-linear homomorphism $\tr^C\colon \Endo_C(M)\rightarrow \coHH_0(C)$ the \textit{colinear Hattori--Stallings trace on $M$}.
Explicitly, given a basis $(e_1, \dots, e_n)$ of $M$, and a $C$-colinear homomorphism $f\colon M\rightarrow M$, we have the formula:
\[
\tr^C(f)=\sum_{i=1}^n \sum_{(e_i)} e_i^*(f({e_i}_{(0)})){e_i}_{(1)}
\]
where we denoted $m\mapsto \sum_{(m)} m_{(0)}\otimes m_{(1)}$ the  coaction  $M\rightarrow M\otimes C$.
The \textit{colinear Hattori--Stallings rank on $M$} is the colinear trace on the identity.
\end{defn}

\begin{proof}[Proof of Proposition \ref{prop: colinear is a cotrace}]
 Denote $\Delta\colon C\rightarrow C\otimes C$ the comultiplication on $C$.
 We need to show that for all colinear endomorphisms $f$, we have $\Delta(\tr^C(f))=\tau(\Delta(\tr^C(f)))$. 
 As $f$ is  colinear and $M$  is a right $C$-comodule,  we obtain  the commutative diagram:
\begin{equation}\label{eq: f is colinear}
\begin{tikzcd}[row sep=2.25em]
	M & M & {M\otimes C} \\
	{M\otimes C} & {M\otimes C} & {M\otimes C^{\otimes 2}}
	\arrow["f", from=1-1, to=1-2]
	\arrow["\rho"', from=1-1, to=2-1]
	\arrow["\rho", from=1-2, to=2-2]
	\arrow["{f\otimes 1}"', from=2-1, to=2-2]
	\arrow["\rho", from=1-2, to=1-3]
	\arrow["{1\otimes \Delta}", from=1-3, to=2-3]
	\arrow["{\rho\otimes 1}"', from=2-2, to=2-3]
\end{tikzcd}
\end{equation}
 As $M$ is a right $C$-comodule, its linear  dual $M^*$ is a left $C$-comodule with coaction:
 \begin{align*}
     \lambda\colon M^* & \longrightarrow \hom_k(M,C) \stackrel{\cong}\leftarrow C\otimes M^*\\
     \alpha  & \longmapsto \left(M\rightarrow M\otimes C \stackrel{\alpha}\rightarrow C\right)
 \end{align*}
The coactions are compatible on $M^*\otimes M$ in the sense that the following diagram commutes:
\begin{equation}\label{eq: compatibility of M^* coaction}
\begin{tikzcd}[row sep=2.25em]
	{M^*\otimes M} & {C\otimes M^*\otimes M} \\
	{M^*\otimes M\otimes C} & C & {C^{\otimes 2}\otimes M^* \otimes M} \\
	& {M^*\otimes M \otimes C^{\otimes 2}} & {C^{\otimes 2}.}
	\arrow["{\lambda\otimes 1}", from=1-1, to=1-2]
	\arrow["{1\otimes \mathrm{ev}}", from=1-2, to=2-2]
	\arrow["{\mathrm{ev}\otimes 1}"', from=2-1, to=2-2]
	\arrow["{1\otimes \rho}"', from=1-1, to=2-1]
	\arrow["{\lambda\otimes 1 \otimes 1}", from=1-2, to=2-3]
	\arrow["{1\otimes 1 \otimes \rho}"', from=2-1, to=3-2]
	\arrow["{\mathrm{ev}\otimes 1 \otimes 1}"', from=3-2, to=3-3]
	\arrow["{1\otimes 1\otimes\mathrm{ev}}", from=2-3, to=3-3]
\end{tikzcd}
\end{equation}
 Using the usual convention for Sweedler notation, and \eqref{eq: f is colinear}  and \eqref{eq: compatibility of M^* coaction}, we obtain 
\begin{align*}
    \Delta(\tr^C(f)) & = \sum_i \sum_{(e_i)} e_i^*(f({e_i}_{(0)})){e_i}_{(1)} \otimes {e_i}_{(2)}\\
    & = \sum_i \sum_{(f(e_i))} e_i^*({f(e_i)}_{(0)}) {f(e_i)}_{(1)}\otimes {f(e_i)}_{(2)}\\
    & = \sum_i \sum_{(e_i^*)} (e_i^*)_{(0)}(f(e_i))(e_i^*)_{(2)} \otimes (e_i^*)_{(1)}\\
    & = \sum_i \sum_{(e_i)} e_i^*(f({e_i})_{(0)}){e_i}_{(2)}\otimes {e_i}_{(1)}\\
    & = \tau(\Delta(\tr^C(f))). \qedhere
\end{align*}
\end{proof}

Let $\mathrm{M}(C)$ be the isomorphism class of right $C$-comodules that are finitely generated as $\mathbbm{k}$-modules. 
By \cite[II.9.1.3, IV.8.4]{Kbook}, the $\pi_0$-group of the algebraic $K$-theory $G^c(C)=K(\comod_C^\mathrm{fg})$ of Definition \ref{def: the second algebraic K-theory} is isomorphic to the free abelian group on $\mathrm{M}(C)$ under the relation $[M]=[M']+[M'']$ whenever $0\rightarrow M'\rightarrow M \rightarrow M''\rightarrow 0$ is an exact sequence in $\comod_C^\mathrm{fg}$. The colinear Hattori--Stallings rank defines an additive homomorphism $\mathrm{M}(C)\rightarrow \coHH_0(C)$. We wish to show it factors through $G^c_0(C)$. 

\begin{lemma}[Additivity]
    Given a morphism of short exact sequences in $\comod_C^\mathrm{fg}$:
    \[
    \begin{tikzcd}
        0 \ar{r} & M' \ar{d}{f'} \ar{r} & M \ar{d}{f}\ar{r} & M'' \ar{d}{f''} \ar{r} & 0\\
        0 \ar{r} & M' \ar{r} & M \ar{r} & M'' \ar{r} & 0
    \end{tikzcd}
    \]
    we obtain $\tr^C(f)=\tr^C(f')+\tr^C(f'')$ in $\coHH_0(C)$.
\end{lemma}

\begin{proof}
    This follows from $\tr^C(f)=\tr(\rho\circ f)$ and the fact that the usual twisted trace is additive.
    In more details, we have $M\cong M'\oplus M''$ as $\mathbbm{k}$-modules. 
    Denote the coactions as follows:
    \begin{align*}
        M &\rightarrow M\otimes C & M' & \rightarrow M'\otimes C & M'' & \rightarrow M''\otimes C \\
        m & \mapsto \sum_{(m)} m_{(0)}^M\otimes m_{(1)}^M & m & \mapsto \sum_{(m)} m_{(0)}^{M'}\otimes m_{(1)}^{M'} & m & \mapsto  \sum_{(m)} m_{(0)}^{M''}\otimes m_{(1)}^{M''}.
    \end{align*}
    Choose a basis $(e_1,\dots, e_n, e_{n+1}, \dots, e_{n+m})$ of $M$ such that $(e_1, \dots, e_n)$ is a basis of $M'$ and $(e_{n+1}, \dots, e_{n+m})$ is a basis of $M''$.
    As $f$,$f'$ and $f''$ are colinear, we obtain for all $1\leq i \leq n$ and $n+1 \leq j \leq n+m$:
    \begin{align*}
        \sum_{(e_i)} f'({e_i}_{(0)}^{M'})\otimes {e_i}_{(1)}^{M'}= \sum_{(e_i)} f({e_i}_{(0)}^{M})\otimes {e_i}_{(1)}^{M} & & \sum_{(e_i)} f''({e_i}_{(0)}^{M''})\otimes {e_i}_{(1)}^{M''}=\sum_{(e_i)} f({e_i}_{(0)}^{M})\otimes {e_i}_{(1)}^{M}.
    \end{align*}
    Then:
    \begin{align*}
        \tr^C(f) & = \sum_{i=1}^{n+m} \sum_{(e_i)} e_i^*(f({e_i}_{(0)}^M)) {e_{i}}_{(1)}^M\\
        & =\sum_{i=1}^n  \sum_{(e_i)}  e_i^*(f({e_i}_{(0)}^{M})){e_{i}}_{(1)}^{M} + \sum_{i=n+1}^{n+m} \sum_{(e_i)} e_i^*(f({e_i}_{(0)}^M)){e_{i}}_{(1)}^M \\
        & =\sum_{i=1}^n  \sum_{(e_i)}  e_i^*(f'({e_i}_{(0)}^{M'})){e_{i}}_{(1)}^{M'} + \sum_{i=n+1}^{n+m} \sum_{(e_i)} e_i^*(f''({e_i}_{(0)}^{M''})){e_{i}}_{(1)}^{M''}\\
        & = \tr^C(f')+\tr^C(f'').\qedhere
    \end{align*}
\end{proof}

\begin{thm}\label{thm: the second trace on K-theory}
    Let $C$ be a $\mathbbm{k}$-coalgebra.
    The colinear Hattori--Stallings rank defines a homomorphism of abelian groups:
    \[
    G^c_0(C)\longrightarrow \coHH_0(C).
    \]
\end{thm}

\begin{proof}
    Considering trace on identity yields a homomorphism $\mathbb{Z}[\mathrm{M}(C)]\rightarrow \coHH_0(C)$.
    Since the colinear trace $\tr^C$ is additive, then the previous homomorphism must factor through $G^c_0(C)$.
\end{proof}

Any right $C$-comodule $M$ can be given a left $C^*$-module with the action given by the coaction followed by evaluation:
\[
C^*\otimes M \rightarrow C^*\otimes M \otimes C\cong C^*\otimes C \otimes M \rightarrow \mathbbm{k}\otimes M\cong M.
\]
This defines a fully faithful functor $\bicomod{}{C}\hookrightarrow \bimod{C^*}{}$ that is an equivalence of categories if and only if $C$ is finitely generated as a $\mathbbm{k}$-module \cite[4.7]{brzezinski2003corings}.
In that case, the ring $C^*$ is also finitely generated as a $\mathbbm{k}$-module, and a left $C^*$-module is finitely generated over $C^*$ if and only if is finitely generated over $\mathbbm{k}$.
This results into an equivalence $\bicomod{}{C}^\mathrm{fg}\simeq \bimod{C^*}{}^\mathrm{fg}$ between right $C$-comodules and left $C^*$-modules that are finitely generated as $\mathbbm{k}$-modules.
Considering $\bimod{C^*}{}^\mathrm{fg}$ as a Waldhausen category with cofibrations all monomorphisms, we obtain $K(\bimod{C^*}{}^\mathrm{fg})\simeq G(C^*)$, the $G$-theory of the ring $C^*$.
The Cartan map $K(C^*)\rightarrow G(C^*)$ combined with the colinear Hattori--Stallings rank defines a homomorphism $K_0(C^*)\rightarrow \coHH_0(C)\cong \HH_0(C^*)^*\cong \HH_0(C^*)$, and this is the usual Hattori--Stallings rank $K_0(C^*)\rightarrow \HH_0(C^*)$ on the ring $C^*$. 
We record our observation as follows choosing $C^*=R$, then $R$ is a $\mathbbm{k}$-algebra that is finitely generated as a $\mathbbm{k}$-module, i.e. an Artinian $\mathbbm{k}$-algebra. In that case, we see $R^*$ becomes a $\mathbbm{k}$-coalgebra.

\begin{prop}\label{prop: G-theory}
    Let $R$ be an Artinian $\mathbbm{k}$-algebra.
    There is an equivalence of algebraic $K$-theory spectra:
    \[
    G(R)\simeq G^c(R^*),
    \]
    and the combination of the Cartan maps $K(R)\rightarrow G(R)$ and colinear Hattori--Stallings rank on $R^*$ recovers the Dennis trace $K_0(R)\rightarrow \coHH_0(R^*)\cong \HH_0(R)$.
\end{prop}

\begin{rem}
If $R$ is a commutative Artinian $\mathbbm{k}$-algebra, then $R=R_1\times \dots\times R_n$ is a finite product of Artinian local rings $R_i$ \cite[X.7.7.(ii)]{lang}. By \cite[V.4.2.1]{Kbook}, when $\mathbbm{k}$ is a field, we get $G(R_i)\simeq K(\mathbbm{k})$. Therefore, for $R$ a commutative Artinian algebra over a field $\mathbbm{k}$, we obtain:
\[
G^c(R^*)\simeq \prod_{i=1}^n K(\mathbbm{k}).
\]
\end{rem}

\section{Bicategories of bicomodules}\label{section: bicategory of comodules}

In this section, we introduce the bicategories of bicomodules in Definition \ref{def: bicategory of discrete bicomodules} and Definition \ref{def: bicategory of derived bicomodules}.
For convenience, we recall the definition of a bicategory, further details of which may be found in \cite{ponto2008fixed, Ponto_2012, campbell2019topological}.

\begin{defn}
A \textit{bicategory} $\mathcal{B}$ consists of following.
\begin{itemize}
    \item A collection of objects, $ob(\mathcal{B})$, called \emph{0-cells}. We write $C\in \mathcal{B}$ instead of $C\in ob(\mathcal{B})$.
    \item Categories $\mathcal{B}(C,D)$ for each pair of objects $C, D \in \mathcal{B}$.  The objects in these categories are referred to as \emph{1-cells} and the morphisms as \emph{2-cells}. The composition is referred to as \emph{vertical composition}.
    \item Unit functors $U_C \in \mathcal{B}(C,C)$ for all $C \in \mathcal{B}$. 
    \item \textit{Horizontal composition} functors for $C,D,E \in \mathcal{B}$
    $$\odot\colon \mathcal{B}(C,D) \times \mathcal{B}(D,E) \to \mathcal{B}(C,E),$$ which are not required to be strictly associative or unital.
    \item Natural isomorphisms for $M \in \mathcal{B}(C,D)$, $N \in \mathcal{B}(D,E)$, $P \in \mathcal{B}(E,F)$, and $Q \in \mathcal{B}(F,G)$ given $C, D, E, F, G \in \mathcal{B}$:
 \[        a\colon (M \odot N) \odot P \xlongrightarrow{\cong} M \odot (N \odot P),\]
       \[ \ell\colon U_C \odot M \xlongrightarrow{\cong} M, \quad\quad \quad r\colon M \odot U_D \xlongrightarrow{\cong} M,\]
    which satisfy the triangle identity
    \[
    \begin{tikzcd}
     (M \odot U_D) \odot N \arrow{dr}[swap]{r \odot \id_N} \arrow{rr}{a} &&M \odot (U_D \odot N) \arrow{dl}{\id_M \odot \ell}\\
    & M \odot N ,
    \end{tikzcd}
    \]
    and the pentagon identity
    \[
    \begin{tikzcd}[row sep= small, column sep=small]
    & &(M \odot N) \odot (P \odot Q) \arrow{dr}{a} &[-2em]\\
    &((M \odot N) \odot P) \odot Q \arrow{ur}{a} \arrow{d}{a \odot \id_Q} & &M \odot (N \odot (P \odot Q))\\
   [+2em] &(M \odot (N \odot P)) \odot Q \arrow{rr}{a}& & M \odot ((N \odot P) \odot Q). \arrow{u}{\id_M \odot a}
    \end{tikzcd}
    \]
\end{itemize}
\end{defn}

A bicategory with one object is precisely a monoidal category. The bicategory of categories in which $0$-cells are categories, $1$-cells are functors, and $2$-cells are natural transformations is a classical example. We recall other examples below that will be motivating in our context. 

\begin{ex}\label{ex: bicat of modules classical}
We can define a bicategory whose 0-cells are rings and whose 1- and 2-cells come from the category of $(R,S)$-bimodules for rings $R$ and $S$. In this case the horizontal composition is given by the tensor product of bimodules. The derived version of this bicategory has 0-cells which are rings and 1- and 2-cells from the derived category of $(R,S)$-bimodules, where horizontal composition is given by the derived tensor product $\otimes^\mathbb{L}$.
Recall that if $M$ is a right $S$-module, and $N$ is a left $S$-module, then $M\otimes_S^\mathbb{L} N$ is equivalent to the two-sided bar construction, $\mathsf{Bar}(M,S,N)$. 
\end{ex}

\begin{ex}\label{ex: bicat of modules derived}
When the 0-cells are given by ring spectra and the 1- and 2-cells come from the homotopy category of $(R,S)$-bimodules for ring spectra $R$ and $S$, then the horizontal composition is given by the derived smash product of spectra. Note that as in the previous example, $\mathsf{Bar}(M,S,N) \simeq M \wedge^\mathbb{L}_S N$, see \cite[IV.7.5]{elmendorf1995rings},
and so we may also consider this horizontal composition as the two-sided bar construction.
\end{ex}

In this paper, we will introduce two bicategories formed by bicomodules instead of bimodules. Their horizontal compositions will be given by the (derived) cotensor product.

\begin{prop}
Let $(\Cc, \otimes, \mathbb{I})$ be a symmetric monoidal category, and let $C$ be a flat coalgebra in $\Cc$. Then $(\bicomod{C}{C}(\Cc), \square_C, C)$ is a monoidal category. Further, it is symmetric monoidal if $C$ is cocommutative.
\end{prop}

\begin{defn}\label{def: bicategory of discrete bicomodules}
    Define $\bicomod{}{}$ to be the bicategory whose $0$-cells are coalgebras in $\bimod{}{\mathbbm{k}}$, and whose $1$-cells, $2$-cells, and vertical compositions are given by the category $\bicomod{C}{D}(\bimod{}{\mathbbm{k}})$. 
The unit $U_C$ is the $(C,C)$-bicomodule $C$, and horizontal composition is given by the cotensor product:
\begin{align*}
    \odot: \mathcal{B}(C,D) \times \mathcal{B}(D,E) &\to \mathcal{B}(C,E)\\
    (M,N) &\mapsto M \odot N := M \square_D N.
\end{align*}
For $(C,D)$-bicomodule $M$, $(D,E)$-bicomodule $N$, and $(E,F)$-bicomodule $P$, the natural isomorphisms
   \[
        a\colon (M \odot N) \odot P \xlongrightarrow{\cong} M \odot (N \odot P), \quad \quad 
        \ell\colon U_C \odot M \xlongrightarrow{\cong} M, \quad \quad 
        r\colon M \odot U_D \xlongrightarrow{\cong} M,
    \]
follow as in \cite{doi1981homological} from the natural isomorphisms 
\[
(M \square_D N) \square_E P \cong M \square_D (N \square_E P), \quad \quad  C \square_C M \cong M, \quad \quad M \square_D D \cong M.
\]
More specifically, these maps are given by
\begin{align*}
    a\colon (M \square_D N) \square_E P &\stackrel{\cong}\longrightarrow M \square_D (N \square_E P) &  \ell\colon C \square_C M &\stackrel{\cong}\longrightarrow M &  r\colon M \square_D D &\stackrel{\cong}\longrightarrow M\\
    (m \otimes n) \otimes p &\longmapsto m \otimes (n \otimes p) &  c\otimes m &\longmapsto  \varepsilon_C(c) m &  m \otimes d &\longmapsto   \varepsilon_D(d)m.
\end{align*}
\end{defn}

We now consider the derived case. We first need a homotopy theory of bicomodules.
Let $\Mc=\Ck$ be the category of non-negative chain complexes over $\mathbbm{k}$. There is a model structure on $\Ck$ in which weak equivalences are quasi-isomorphisms, cofibrations are monomorphisms, and fibrations are positive levelwise epimorphisms, see \cite{hovey}. 
It is a combinatorial symmetric monoidal model category with its usual tensor product, and a simplicial model category via the Dold-Kan correspondence. In addition, every object is cofibrant and fibrant.

\begin{prop}
    Let $C$ and $D$ be dg-coalgebras over $\mathbbm{k}$. There exists a simplicial combinatorial model structure on $\bicomod{C}{D}(\Ck)$ in which
\begin{itemize}
    \item the weak equivalences $\mathsf{W}$ are the morphisms of $(C, D)$-bicomodules that are quasi-isomorphisms in $\Ck$;
    \item the cofibrations are the morphisms of $(C, D)$-bicomodules that are monomorphisms in $\Ck$.
\end{itemize}
In particular, every object is cofibrant.
\end{prop}

\begin{proof}
    The model structure is left-induced in the sense of \cite{hkrs} via the forget\-ful-cofree adjunction
\[
\begin{tikzcd}[column sep= large]
\bicomod{C}{D}\left(\Ck\right) \ar[shift left=2]{r}{U}[swap]{\perp} & \Ck. \ar[shift left=2]{l}{C\otimes -\otimes D}
\end{tikzcd}
\]
Indeed, this follows from \cite[2.12]{connectivecomod} and \cite[6.3.7]{hkrs} by Remark \ref{rem: bicomodules are right comodules (ordinary)}.
\end{proof}

The above result is true more generally  for unbounded chain complexes over any ring. However, it is not known if this forms an algebraic model for homotopy coherent comodules in this generality.

\begin{defn}
A (connective) \emph{dg-coalgebra over $\mathbbm{k}$} is a coalgebra $C$ in $\Ck$.
We say it is \emph{simply connected} if $C_0\cong\mathbbm{k}$ and $C_1=0$.
\end{defn}

When $C$ and $D$ are simply connected, we show the model structure from the previous proposition defines the correct homotopy type, see Theorem \ref{thm: bicomodules connective}.
Fibrant objects in $\bicomod{C}{D}(\Ck)$ are retracts of Postnikov towers, see more details in Appendix \ref{chap: appendix}.

In order to define horizontal composition for the bicategory in the derived setting, we now need to derive the cotensor product. However, it is more subtle than the usual tensor product of modules because the cotensor product is neither a left nor a right adjoint in general. Nevertheless, one can right derive the cotensor product using methods of \cite{connectivecomod, pertower}.
We leave the details in Appendix \ref{chap: appendix}, but essentially, we show that the cotensor product of fibrant bicomodules is again fibrant (Proposition \ref{prop: cotensor of fibrants}) and that the cotensor preserves weak equivalences on fibrant objects (Proposition \ref{Prop: cotensor preserve weak equivalence if fibrant}). Moreover, just as the derived tensor product can be interpreted as a two-sided bar construction, the derived cotensor product can be interpreted as a two-sided cobar construction, as we shall now explain.

\begin{defn}\label{def: cobar cosimplicial}
Let $C$ and $D$ be simply connected dg-coalgebras over $\mathbbm{k}$, with comultiplication and counit of $C$ given by $\Delta\colon  C\rightarrow C\otimes C$ and $\varepsilon:C\rightarrow \mathbbm{k}$ respectively.
Let $M$ be a $(D,C)$-bicomodule in $\Ck$, and let $\rho\colon M\rightarrow M\otimes C$ denote its right coaction over $C$.
The \emph{two-sided cosimplicial cobar construction} $\ccobars{M}$ of $M$ is the cosimplicial object in $\bicomod{D}{C}(\Ck)$:
\[
\begin{tikzcd}
M\otimes C \ar[shift left]{r}\ar[shift right]{r}  &   \ar{l} M\otimes C \otimes C \ar[shift left=2]{r}\ar{r} \ar[shift right=2]{r} &  \ar[shift right]{l} \ar[shift left]{l} \cdots, 
\end{tikzcd}
\]
defined as follows.
\begin{itemize}
\item For all $n\geq 0$, $\Omega^{n}(M, C, C)=M\otimes C^{\otimes n+1}$.
\item The zeroth coface map $d^0\colon \Omega^{n}(M, C, C)\rightarrow \Omega^{n+1}(M, C, C)$ is given by $d^0 = \rho\otimes \id_{C^{n+1}}$.
\item For $1\leq i \leq n+1$, the $i^{th}$ coface map $d^i\colon \Omega^{n}(M, C, C) \rightarrow \Omega^{n+1}(M, C, C)$ is given by \[d^i= \id_M \otimes \id_{C^{\otimes i-1}} \otimes \Delta \otimes \id_{C^{\otimes n+1-i}}.\]
\item For all $0\leq j\leq n$, the $j^{th}$ codegeneracy map $s^j\colon \Omega^{n+1}(M, C, C) \rightarrow \Omega^{n}(M, C, C)$ is given by
\[
s^j=\id_M\otimes \id_{C^{\otimes j}} \otimes \varepsilon \otimes \id_{C^{\otimes n+1-j}}.
\]
\end{itemize}
Since $\bicomod{D}{C}(\Ck)$ is a simplicial model category, homotopy limits over cosimplicial diagrams are computed as in \cite[18.1.8]{hir}. 
We denote the homotopy limit of the cosimplicial diagram $\ccobars{M}$ in $\bicomod{D}{C}(\Ck)$ by $\ccobar{M}$, and we say it is the\emph{ two-sided cobar resolution of $M$.}
By \cite[2.5]{connectivecomod}, we have $M\simeq \ccobar{M}$ as a $(D,C)$-bicomodule if $M$ is fibrant as a left $D$-comodule. 
Notice that each object in the cosimplicial diagram $\ccobars{M}$ is a fibrant right $C$-comodule by Lemma \ref{lem: tensor preserves fibrant (cofree case)}.
Thus $\ccobar{M}$ is a fibrant $(D,C)$-bicomodule by \cite[18.5.2]{hir} if $M$ is a fibrant left $D$-comodule. 
\end{defn}

\begin{defn}\label{def: cosimplicial cobar general}
Let $C$, $D$ and $E$ be simply connected dg-coalgebras over $\mathbbm{k}$. Let $M$ be a $(D,C)$-bicomodule and $N$ be a $(C, E)$-bicomodule.
We define the\emph{ two-sided cosimplicial cobar construction of $M$ and $N$} to be the cosimplicial object $\cobars{M}{N}$ in $\bicomod{D}{E}(\Ck)$ given by $\cobars{M}{C}\square_C N$. We write $\cobar{M}{N}$ for the homotopy limit of $\cobars{M}{N}$ in $\bicomod{D}{E}(\Ck)$. 
As noted in \cite[2.5]{connectivecomod}, it is equivalent to compute the homotopy limit in $\Ck$.
\end{defn}

\begin{prop}\label{equivalence}
Let $C$, $D$ and $E$ be simply connected dg-coalgebras over $\mathbbm{k}$. Let $M$ be a $(D,C)$-bicomodule and $N$ be a fibrant $(C, E)$-bicomodule.
Then we obtain an equivalence $\cobar{M}{C}\square_C N\simeq \cobar{M}{N}$ as $(D,E)$-bicomodules.
\end{prop}

\begin{proof}
This argument is similar to discussion in \cite[4.2]{connectivecomod} (the cocommutative requirement there was not needed).
\end{proof}

\begin{rem}\label{rem: topological cobar vs algebraic cobar}
    By the Dold-Kan correspondence, since $C$ is a simply connected dg-coal\-gebra, the two-sided cosimplicial cobar resolution $\Omega(M, C, N)$ from Definition \ref{def: cobar cosimplicial} is quasi-isomorphic to $\underline{\Omega}(M,C,N)$, the \emph{conormalized cobar  resolution of $M$ and $N$ over $C$}, which we now define. 
    We first establish the following notation conventions. 
    \begin{itemize}
    \item Given $V$ a graded $\mathbbm{k}$-module, we define
    \[ T(V)=\bigoplus_{n\geq 0} V^{\otimes n}.\]
    Elements in the summands are denoted $v_1 \vert \cdots \vert v_n$, where $v_i\in V$.
    \item Let $s^{-1}$ denote the desuspension functor on graded $\mathbbm{k}$-modules where, for $V=\bigoplus_{i\in \mathbb{Z}} V_i$, we define $(s^{-1} V)_i=V_{i+1}$.
    Given a homogeneous element $v$ in $V$, we write $s^{-1} v$ for the corresponding element in $s^{-1}V$.

    \item We denote the kernel of the counit $\varepsilon\colon C\rightarrow \mathbbm{k}$ by $\underline{C}$, often referred to as the coideal of $C$.
\end{itemize}
The conormalized cobar resolution of $C$ is the chain complex
\(
\underline{\Omega}C:=(T(s^{-1}\underline{C}), d_\Omega)
\)
where, if $d$ denotes the differential on $C$, then
\begin{eqnarray*}
d_\Omega(s^{-1} c_1 \vert \cdots \vert s^{-1} c_n) & = & \sum_{j=1}^n \pm s^{-1}c_1 \vert \cdots \vert s^{-1}(dc_j)\vert \cdots \vert s^{-1}c_n\\
& & + \sum_{j=1}^n \sum_{(c_j)}\pm s^{-1}c_1 \vert \cdots \vert s^{-1} {c_{j}}_{(1)} \vert s^{-1}{c_j}_{(2)} \vert \cdots \vert s^{-1}c_n,
\end{eqnarray*}
where $\Delta(c_j)=\sum_{(c_j)}{c_{j}}_{(1)}\otimes {c_j}_{(2)}$ denotes the (reduced) comultiplication of $\underline{C}$ on the element $c_j$ using the Sweedler notation. To determine the signs above, one needs to apply the Koszul rule.
More generally, we define the conormalized cobar resolution of $M$ and $N$ over $C$ to be the chain complex
\(
\underline{\Omega}(M,C,N):=(M\otimes T(s^{-1}\underline{C}) \otimes N, \delta),
\)
where up to Koszul sign, the differential $\delta$ is defined as
\begin{eqnarray*}
\delta(m \otimes s^{-1} c_1 \vert \cdots \vert s^{-1}c_n \otimes n) & = & dm \otimes s^{-1}c_1 \vert \cdots \vert s^{-1}c_n \otimes n \\
& & \pm \sum_{(m)} m_{(0)}\otimes s^{-1} m_{(1)}\vert s^{-1} c_1 \vert \cdots \vert s^{-1} c_n \otimes n\\
& & \pm m\otimes d_\Omega\Big(s^{-1} c_1 \vert \cdots \vert s^{-1} c_n\Big) \otimes n\\ 
& & \pm m\otimes s^{-1}c_1 \vert \cdots \vert s^{-1}c_n \otimes dn\\
& & \pm \sum_{(n)} m\otimes s^{-1} c_1 \vert \cdots \vert s^{-1} c_n \vert s^{-1}n_{(1)} \otimes n_{(0)},
\end{eqnarray*}
where $d$ denotes either the differential on $M$ or $N$, and $m\mapsto \sum_{(m)} m_{(0)}\otimes m_{(1)}$ denotes the coaction of $M\rightarrow M\otimes C$ applied to an element $m$, and $n\mapsto \sum_{(n)} n_{(1)}\otimes n_{(0)}$ denotes the coaction of $N\rightarrow C\otimes N$ applied to an element $n$.

By the Dold-Kan correspondence, there is an isomorphism of categories between cosimplicial objects and non-positive cochain complexes of bicomodules given by the conormalization functor
\[
N^*\colon\left(\bicomod{D}{E}(\Ck)\right)^\Delta \stackrel{\cong}\longrightarrow \mathsf{coCh}^{\leq 0}\left(\bicomod{D}{E}(\Ck) \right).
\]
Therefore, we denote $N^*(\Omega^\bullet(M, C, N))$ by $\underline{\Omega}^\bullet(M, C, N)$. It is a double complex, and one can show that its total complex is precisely $\underline{\Omega}(M,C, N)$.
A word of warning however: a double complex in general has two possible totalizations, one given using coproducts and one using products (see \cite[1.2.6]{weibel}). 
The product-total complex of $\underline{\Omega}^\bullet(M, C, N)$ is quasi-isomorphic to the homotopy limit of $\Omega^\bullet(M,C,N)$, i.e., $\Omega(M,C,N)$ (see \cite[4.23]{ulrich} for instance). 
On the other hand, the coproduct-total complex of $\underline{\Omega}^\bullet(M, C, N)$ is $\underline{\Omega}(M,C,N)$. Of course, if the double complex is bounded, these are equal. 
For instance, when $C$ is simply-connected, then $\underline{\Omega}^{-q}(M,C,N)_p=0$ for $0\leq p \leq 2q-1$, and is therefore bounded. Thus when $C$ is simply-connected, we obtain a quasi-isomorphism
\[
\Omega(M,C,N)\simeq \underline{\Omega}(M,C, N).
\]
But in general, if $C$ is not simply-connected, no such claim can be made.
\end{rem}

\begin{rem}\label{rem: cotor is cohomology of cobar}
Let $C$ be a simply-connected coalgebra in $\Ck$.
    Let $M$  and $N$ be left and right $C$-comodules respectively. As noted in \cite[A1.2.12]{ravenel}, for all $i\geq 0$ we have an isomorphism
\[\cotor_C^i(M, N)\cong H^i\left(\underline{\Omega}^\bullet(M,C,N)\right).\] 
\end{rem}

Therefore, on the homotopy category of comodules, we have a derived cotensor product, which we denote by $\widehat{\square}$. In fact, $M\dcotensor_C N$ is quasi-isomorphic to the two-sided cobar resolution $\Omega(M,C, N)$, see Corollary \ref{cor: cotensor of fibrant is equal to cobar}. The derived cotensor preserves any homotopy coherent coactions, see \cite{connectivecomod}.

Having established the necessary definitions for each of the components, we now define the appropriate bicategory for the derived bicomodule setting. 

\begin{defn}\label{def: bicategory of derived bicomodules}
    Define $\dcomod$ to be the bicategory whose $0$-cells are simply connected dg-coalgebras over $\mathbbm{k}$, and whose 1-cells, 2-cells, and vertical compositions are given by the homotopy category of the model category $\bicomod{C}{D}(\Ck)$. The unit $U_C$ is the fibrant $(C,C)$-bicomodule $C$, and horizontal composition is given by the derived cotensor product of bicomodules. Given a fibrant $(C,D)$-bicomodule $M$ and a fibrant $(D,E)$-bicomodule $N$, their horizontal composition $M\odot N$ is the fibrant $(C,E)$-bicomodule 
    \[
    M\square_D N \simeq M\dcotensor_D N \simeq \Omega(M,D, N).
    \]
    If $P$ is a fibrant $(E,F)$-bicomodule, then we define the natural isomorphisms:
    \[
    a\colon (M\square_D N)\square_E P \stackrel{\cong}\longrightarrow M\square_D (N\square_E P),
    \]
    \[
    \ell\colon C\square_C M \stackrel{\cong}\longrightarrow M, \quad r\colon M\square_D D \stackrel{\cong}\longrightarrow M,
    \]
    as follows. The isomorphism $a$ follows from  the natural isomorphism $(M\otimes N)\otimes P\cong M\otimes (N\otimes P)$, and therefore automatically respects the pentagon identity.
    The natural isomorphisms $\ell$ and $r$ are induced by the counits $C\rightarrow \mathbbm{k}$ and $D\rightarrow \mathbbm{k}$ respectively, and the triangle identities follow for the cotensor products since they hold for the tensor product.

    Since we shall need it in next section, we provide an explicit definition of $a$, $\ell$, and $r$ from above, where we instead use the cobar construction as a model for the derived cotensor product.
    The associative equivalence
    \[
a\colon \Omega(\Omega(M, D, N), E, P) \stackrel{\simeq}\longrightarrow \Omega(M ,D, \Omega(N, E, P))
    \]
    is induced by an isomorphism of cosimplicial objects $\Omega^\bullet(\Omega^\bullet(M, D, N), E, P)\cong \Omega^\bullet(M ,D, \Omega^\bullet(N, E, P))$, using Corollary \ref{cor: bicosimplicial of cobar} from Appendix \ref{chap: appendix}. Indeed, the $(i, j)$-spot corresponds to the $(j,i)$-spot up to isomorphism
    \[
(M\otimes D^{\otimes j}\otimes N)\otimes E^{\otimes i}\otimes P \stackrel{\cong}\longrightarrow M\otimes D^{\otimes j}\otimes (N\otimes E^{\otimes i}\otimes P)
    \]
    from the usual associative isomorphism. The equivalences $\ell: \Omega(C,C, M)\stackrel{\simeq}\rightarrow M$ and $r\colon \Omega(M, D, D)\stackrel{\simeq}\rightarrow M$ are defined as in Definition \ref{def: cobar cosimplicial}, and are induced by the counits. For instance, the map $C^{\otimes n+1}\otimes M\rightarrow M$ is induced by repeatedly applying the counit on $C$.
\end{defn}

\section{CoHochschild homology as a shadow}\label{section: coHH as a shadow}

Here we show our main results, namely that coHochschild homology provides a shadow structure on  the previously defined bicategories of bicomodules.

\begin{defn}[\cite{ponto2008fixed, Ponto_2012}]
    A \emph{shadow functor} for a bicategory $\mathcal{B}$ consists of functors
    $$\lang -\rang_C\colon \mathcal{B}(C,C) \longrightarrow \mathbf{T} $$
    for every $C \in \mathcal{B}$ and some fixed category $\mathbf{T}$ equipped with a natural isomorphism for $M \in \mathcal{B}(C,D)$, $N \in \mathcal{B}(D,C)$: 
\begin{align*}
    \theta\colon \lang M \odot N \rang_C \xlongrightarrow{\cong} \lang N \odot M \rang_D.
\end{align*}
For $P \in \mathcal{B}(C,C)$, these functors must satisfy the following commutative diagrams
\[
\begin{tikzcd}
    \lang (M \odot N) \odot P \rang_C \ar{r}{\theta} \ar{d}[swap]{\lang a \rang} & \lang P \odot (M \odot N) \rang_C \ar{r}{\lang a \rang} & \lang (P \odot M) \odot N\rang_C\\
    \lang M \odot (N \odot P)\rang_C \ar{r}{\theta} & \lang (N \odot P) \odot M\rang_D \ar{r}{\lang a \rang} & \lang N \odot (P \odot M)\rang_D, \ar{u}[swap]{\theta}
\end{tikzcd}
\]
\[
\begin{tikzcd}
    \lang P \odot U_C \rang_C \ar{r}{\theta} \ar{dr}[swap]{\lang r \rang} & \lang U_C \odot P \rang_C \ar{d}{\lang \ell\rang } \ar{r} & \lang P \odot U_C\rang_C\ar{dl}{\lang r \rang}\\
    &\lang P \rang_C.
\end{tikzcd}
\]
In this case, we say $\mathcal{B}$ is a \emph{shadowed bicategory}, and we write $(\mathcal{B}, \lang - \rang)$ for the bicategory and its shadow.
\end{defn}

\begin{ex}
We can now consider shadows for the bicategories that we introduced in Examples \ref{ex: bicat of modules classical} and \ref{ex: bicat of modules derived}. For a ring $R$, the $0^{th}$ Hochschild homology, $\HH_0(R,-)$, is a shadow on the bicategory with 1- and 2-cells from the category of $(R,R)$-bimodules to the category of abelian groups $\mathrm{Ab}$:
\begin{align*}
    \lang-\rang_R\colon \bimod{R}{R} &\to \mathrm{Ab}\\
    M &\mapsto R \otimes_{R \otimes R^{\op}} M \cong \HH_0(R, M).
\end{align*}
More generally, a Dennis-Waldhausen Morita argument shows that Hochschild homology, $\HH(R,-)$, is a shadow in the derived setting \cite{waldhausen}. Further, \cite{blumberg2012localization} shows that topological Hochschild homology, $\THH(R, -)$, is a shadow to the homotopy category of spectra:
\begin{align*}
    \lang-\rang_R\colon \Ho( \bimod{R}{R}) &\to \Ho(\Sp)\\
    M &\mapsto \THH(R, M).
\end{align*}
\end{ex}

Since (topological) Hochschild homology provides a bicategorical shadow for the setting of modules, we want to consider the analogue of this construction for the context of comodules. 
Work of Doi defines coHochschild homology, denoted $\coHH$, as an invariant of coalgebras analogous to Hochschild homology. 

\begin{defn}[\cite{doi1981homological}]\label{def: cohochschild homology underived}
    For a commutative ring $R$, a coassociative, counital $R$-coalgebra $C$, and a $(C,C)$-bicomodule $M$, build the cochain complex $\mathcal{H}(M,C)$:
    $$\cdots \xlongleftarrow{} M \otimes_R C \otimes_R C \xlongleftarrow{} M \otimes_R C \xlongleftarrow{} M \xlongleftarrow{} 0,$$
    as follows. Let $\mathcal{H}^r(M, C)=M\otimes_R C^{\otimes r}$ for $r\geq 0$
    with coboundary map $\delta^r\colon\mathcal{H}^r(M, C)\rightarrow \mathcal{H}^{r+1}(M, C)$ defined by
    \[\delta^r = \sum_{i=0}^{r+1} (-1)^i d_i,\] for
    $d_i$ given by
    \begin{align*}
        d_i &= \begin{cases}
    \rho \otimes \id_C^{\otimes r} &i=0\\
    \id_M \otimes \id_C^{\otimes i-1} \otimes \Delta \otimes \id_C^{\otimes (r-i)} &1 \le i \le r\\
    \Tilde{t} \circ (\lambda \otimes \id_C^{\otimes r}) &i=r+1,
    \end{cases}
    \end{align*}
    where $\rho\colon M\rightarrow M\otimes_R C$ denotes the right coaction, $\lambda\colon M\rightarrow C\otimes_R M$ denotes the left coaction, and $\Tilde{t}$ is the map that twists the first factor to the last.  Then the $q^{th}$-\textit{coHochschild homology of $C$ with coefficients in $M$} is given by the cohomology of the cochain complex
    $$\coHH_q(M, C) := H^q(\mathcal{H}(M,C)).$$
\end{defn}

\begin{rem}
    We can see that $\mathcal{H}(M,C)=\underline{\Omega}(M, C^e, C)$ as in Remark \ref{rem: topological cobar vs algebraic cobar}, where $C^e=C\otimes_R C^\op$, and in particular we get that $\coHH_q(M,C)\cong \cotor^q_{C^e}(M, C)$, see Definition \ref{def: cotor} and Remark \ref{rem: cotor is cohomology of cobar}.
    In particular $\coHH_0(M,C)\cong M\square_{C^e}C$.
\end{rem}

\begin{thm}\label{thm: coHH zero is a shadow}
The $0^{th}$ coHochschild homology $\coHH_0$ is a shadow on the bicategory $\bicomod{}{}$.  That is, it gives a family of functors
\begin{align*}
    \coHH_0(-, C)\colon \bicomod{C}{C} &\to \mathsf{Mod}_\mathbbm{k}\\
    M &\mapsto \coHH_0(M, C)
\end{align*}
that satisfy the required shadow properties.
\end{thm}

\begin{proof}
Let $C$ and $D$ be coalgebras over $\mathbbm{k}$.
Given $M$ a $(C, D)$-bicomodule and $N$ a $(D, C)$-bicomodule, we need to show that we have an isomorphism
\[
\theta\colon \coHH_0(M\square_D N, C)\longrightarrow \coHH_0(N\square_C M, D).
\]
Using Sweedler notation, we denote the coactions on $M$ as follows:
\begin{align*}
 M & \longrightarrow C\otimes M,  & M & \longrightarrow M\otimes D\\
 m & \longmapsto \sum_{(m^C)} m^C_{(1)}\otimes m^C_{(0)} & m &\longmapsto\sum_{(m^D)} m^D_{(0)} \otimes m^D_{(1)},
\end{align*}
and the coactions on $N$ as follows:
\begin{align*}
    N & \longrightarrow D\otimes N, & N & \longrightarrow N\otimes C\\
    n & \longmapsto \sum_{(n^D)}n^D_{(1)}\otimes n^D_{(0)} & n& \longmapsto \sum_{(n^C)}n^C_{(0)}\otimes n^C_{(1)}.
\end{align*}
Recall that $\coHH_0(M\square_DN, C)$ is defined as the kernel of
\[
\delta^0\colon M\square_D N \rightarrow (M\square_D N)\otimes C.
\]
The desired isomorphism will be induced by
\begin{align*}
\tau\colon M\otimes N &\longrightarrow N\otimes M\\
m\otimes n & \longmapsto n\otimes m.
\end{align*}
We need to verify that if we restrict $\tau$ to $\coHH_0(M\square_DN, C)$, we indeed corestrict to $\coHH_0(N\square_CM, D)$. In other words, we need to verify that if $m\otimes n\in \coHH_0(M\square_DN, C)$, then $n\otimes m \in \coHH_0(N\square_CM, D)$.
Notice that because $m\otimes n\in \coHH_0(M\square_DN, C)$, it follows that
\begin{enumerate}
    \item\label{eq: m otimes n cond 1} $\displaystyle \sum_{(m^D)} m^D_{(0)}\otimes m^D_{(1)}\otimes n =\sum_{(n^D)} m\otimes n^D_{(1)}\otimes n^D_{(0)}$, since $m\otimes n\in M\square_D N$;
    \item\label{eq: m otimes n cond 2} $\displaystyle \sum_{(n^C)}m\otimes n^C_{(0)}\otimes n^C_{(1)}=\sum_{(m^C)}m^C_{(0)}\otimes n \otimes m^C_{(1)}$, since $m\otimes n\in \mathbbm{k}er(\delta^0)$.
\end{enumerate}
To see that $n\otimes m \in \coHH_0(N\square_CM, D)$, notice that $n\otimes m\in N\square_C M$ follows from (\ref{eq: m otimes n cond 2}) above, while $n\otimes m\in \mathbbm{k}er(\delta^0)$ follows from (\ref{eq: m otimes n cond 1}) above. Therefore we have obtained the desired homomorphism, $\theta$. 
An analogous argument defines its inverse, and thus $\theta$ is an isomorphism.

Next we must show that for a $(C,C)$-bicomodule $P$, the following diagram is commutative:
\[
\begin{tikzpicture}[baseline= (a).base]
\node[scale=1] (a) at (1,1){
\begin{tikzcd}
	{\coHH_0((M \square_D N) \square_C P, C)} & {\coHH_0(P \square_C (M \square_D N), C)} & {\coHH_0((P \square_C M) \square_D N, C)} \\
	{\coHH_0(M \square_D (N \square_C P),C)} & {\coHH_0((N \square_C P) \square_C M, D)} & {\coHH_0(N \square_C (P \square_C M), D)}
	\arrow["{\lang a \rang}"', from=1-1, to=2-1]
	\arrow["\theta"', from=2-1, to=2-2]
	\arrow["\theta"', from=2-3, to=1-3]
	\arrow["{\lang a \rang}"', from=2-2, to=2-3]
	\arrow["{\lang a \rang}", from=1-2, to=1-3]
	\arrow["\theta", from=1-1, to=1-2]
\end{tikzcd}
};  
\end{tikzpicture}
\]
We check its commutativity directly by applying the definition of $\theta$ given above and $a$ from Definition \ref{def: bicategory of discrete bicomodules}:
\[
\begin{tikzcd}
(m\otimes n)\otimes p \ar[mapsto]{r}\ar[mapsto]{d} & p \otimes (m \otimes n) \ar[mapsto]{r} & (p\otimes m)\otimes n. \\
m\otimes (n\otimes p) \ar[mapsto]{r} & (n\otimes p)\otimes m \ar[mapsto]{r} & n\otimes (p\otimes m) \ar[mapsto]{u}
\end{tikzcd}
\]
Similarly, we need to check the commutativity of the diagram below:
\[\begin{tikzcd}
	{\coHH_0(P \square_C C,C)} && {\coHH_0(C \square_C P,C)} && {\coHH_0(P \square_C C,C)} \\
	\\
	&& {\coHH_0(P, C)}
	\arrow["{\lang \ell \rang}", from=1-3, to=3-3]
	\arrow["\theta", from=1-1, to=1-3]
	\arrow["{\lang r \rang}"', from=1-1, to=3-3]
	\arrow["\theta", from=1-3, to=1-5]
	\arrow["{\lang r \rang}", from=1-5, to=3-3]
\end{tikzcd}\]
This follows again by applying the definitions of $\theta$, $r$, and $\ell$ (see Definition \ref{def: bicategory of discrete bicomodules}):
\[
\begin{tikzcd}
 p\otimes c \ar[mapsto]{r} \ar[mapsto]{dr} & c\otimes p \ar[mapsto]{d} \ar[mapsto]{r} & p\otimes c \ar[mapsto]{dl}\\
 & \varepsilon(c)p
\end{tikzcd}
\]
This proves that the $0^{th}$ coHochschild homology is a shadow in this bicategorical setting.
\end{proof}

Having established that $\coHH_0$ is a bicategorical shadow on $\bicomod{}{}$, we now consider the derived setting.
Recall that $\underline{\Omega}$ of Remark \ref{rem: topological cobar vs algebraic cobar} is not invariant under quasi-isomor\-phisms in general. 
Therefore particular care is required in order to show that $\coHH$ is a shadow in the derived setting. To do so, we must instead consider coalgebras in chain complexes that are simply connected.
Extending the definitions of \cite{HScothh} and \cite{bohmann2018computational}, the first author introduced in \cite[2.8]{klanderman2022computations} the notion of coHochschild homology with coefficients for any model category with a symmetric monoidal structure.

\begin{defn}\label{def: coHH sarah def}
Let $(\mathsf{M}, \otimes, \bI)$ be a symmetric monoidal category with a model structure, and let $C \in \mathsf{M}$ be a coalgebra with coassociative comultiplication $\Delta: C \to C \otimes C$ and counit $\varepsilon: C \to \bI$.  Further, let $M$ be a $(C,C)$-bicomodule with left and right coactions $\lambda: M \to C \otimes M$ and $\rho: M \to M \otimes C$ respectively.
Define $\coTHH^{\mathsf{M}}(M, C)^\bullet$ to be the  cosimplicial object with $r$-simplices $\coTHH^{\mathsf{M}}(M, C)^r = M \otimes C^{\otimes r}$,
with coface maps
    \begin{align*}
        d_i &= \begin{cases}
    \rho \otimes \id_C^{\otimes r} &i=0\\
    \id_M \otimes \id_C^{\otimes i-1} \otimes \Delta \otimes \id_C^{\otimes (r-i)} &1 \le i \le r\\
    \Tilde{t} \circ (\lambda \otimes \id_C^{\otimes r}) &i=r+1,
    \end{cases}
    \end{align*}
where $\Tilde{t}$ is the map that twists the first factor to the last, and with codegeneracy maps $s_i: M \otimes C^{\otimes (r+1)} \to M \otimes C^{\otimes r}$, for $0 \le i \le r$,
\[
s_i = \id_M \otimes \id_C^{\otimes i} \otimes \varepsilon \otimes \id_C^{\otimes r-i}.
\]
This gives a cosimplicial object of the form
\[\begin{tikzcd}
	\cdots & {M \otimes C\otimes C} & {M \otimes C} & M.
	\arrow[shift left=1, from=1-4, to=1-3]
	\arrow[shift right=1, from=1-4, to=1-3]
	\arrow[from=1-3, to=1-4]
	\arrow[shift right=2, from=1-3, to=1-2]
	\arrow[from=1-3, to=1-2]
	\arrow[shift left=2, from=1-3, to=1-2]
	\arrow[shift right=1, from=1-2, to=1-3]
	\arrow[shift left=1, from=1-2, to=1-3]
	\arrow[shift right=3, from=1-2, to=1-1]
	\arrow[shift left=2, from=1-1, to=1-2]
	\arrow[shift right=1, from=1-2, to=1-1]
	\arrow[from=1-1, to=1-2]
	\arrow[shift left=1, from=1-2, to=1-1]
	\arrow[shift right=2, from=1-1, to=1-2]
	\arrow[shift left=3, from=1-2, to=1-1]
\end{tikzcd}\]
The \textit{coHochschild homology in $\mathsf{M}$ of the coalgebra $C$ with coefficients in $M$} is then defined by
\[
\coTHH^\mathsf{M}(M,C) = \mathsf{holim}_\Delta \big(\coHH^\mathsf{M}(M,C)^\bullet\big).
\]
\end{defn}

\begin{rem}
In \cite{bicomodthh}, the authors give an $\infty$-categorical definition of topological Hochschild homology with coefficients, extending the approach of \cite{tch}.
In \cite{dualitySW}, a dual approach of \cite{tch} for topological coHochschild homology was given. Mimicking this approach, we can define a relative cyclic cobar construction  $\mathsf{CoBar}^\bullet_\mathcal{C}(-, -)\colon \mathsf{CoAlg}_\mathcal{AB}(\mathcal{C})\rightarrow \mathsf{Fun}(\N(\Delta), \mathcal{C})$ on a symmetric monoidal $\infty$-category $\mathcal{C}$, where $\mathsf{CoAlg}_\mathcal{AB}(\mathcal{C})$ denotes the $\infty$-category of pairs $(M,C)$ consisting of an $\mathbb{E}_1$-coalgebra $C$ and a bicomodule $M$ over $C$.
Essentially, the construction $\mathsf{CoBar}^\bullet_\mathcal{C}(C, M)$ is making precise the diagram in $\mathcal{C}$:
\[
\begin{tikzcd}
 \cdots  & \ar[shift left=3]{l}\ar[shift right=3]{l}\ar[shift left]{l}\ar[shift right]{l}
 M\otimes C \otimes C  &
 \ar[shift left=2]{l}\ar[shift right=2]{l}
 M\otimes C \ar{l}  & M\ar[shift left]{l}\ar[shift right]{l}.
\end{tikzcd}
\]
Denote $\coTHH^\mathcal{C}(M, C)$ the totalization in $\mathcal{C}$ of the cosimplicial object obtained by the cyclic cobar construction $\mathsf{CoBar}^\bullet_\mathcal{C}(M, C)$. Per usual, we denote $\coHH^\mathcal{C}(C,C)$ by $\coHH^\mathcal{C}(C)$.
If $\Mc$ is a combinatorial symmetric monoidal model category with class of weak equivalences denoted by $\mathsf{W}$, then one could also obtain a symmetric monoidal $\infty$-category $\N(\Mc_c)[\mathsf{W}^{-1}]$ obtained by Dwyer--Kan localization. Then one could consider $\coHH^{\N(\Mc_c)[\mathsf{W}^{-1}]}(C,M)$.
Suppose both $C$ and $M$ are cofibrant in $\Mc$. Then, similar to \cite[2.13]{dualitySW} we obtain a weak equivalence
    \[
\coHH^\Mc(M,C)\simeq \coHH^{\N(\Mc_c)[\mathsf{W}^{-1}]}(M,C).
    \]
\end{rem} 

\begin{ex}
    Our main example of interest is $\Mc=\Ck$. In this case, when $C$ is simply connected, the definition of coHochschild homology coincides with \cite{hess2009cohochschild}, see also Remark \ref{rem: topological cobar vs algebraic cobar} and \cite[2.11]{bohmann2018computational}.
\end{ex}

In what follows, we write $\coHH^\Ck$ simply as $\coHH$.
We show in Proposition \ref{prop: coHH is given by cobar construction over Ce} that $\coHH(M,C)\simeq \Omega (M, C^e, C)$ for $C$ simply connected and $M$ a fibrant $(C,C)$-bicomodule.

\begin{thm}\label{theorem: coHH is shadow (derived)}
CoHochschild homology defines a shadow on the bicategory of derived bicomodules over simply connected coalgebras  $\dcomod$.
\end{thm}

\begin{proof}
Let $C$ and $D$ be simply connected coalgebras in $\Ck$.
Given $M$ a fibrant $(C,D)$-bicomodule, and $N$ a fibrant $(D,C)$-bicomodule, we need to show that we have a quasi-isomorphism:
\[
\theta\colon\coHH(M\dcotensor_D N, C) \longrightarrow \coHH(N\dcotensor_C M, D).
\]
As $M$ and $N$ are fibrant, the derived cotensor $M\dcotensor_D N$ is modeled by $M\square_D N$ or $\Omega(M, D, N)$. We choose the cobar construction as a model: it automatically gives us the desired quasi-isomorphism at the cost of some combinatorics.
To provide this equivalence, we apply Corollary \ref{cor: bicosimplicial of cohh} from Appendix \ref{chap: appendix}. In particular, we claim that $\theta$ is induced by an isomorphism of bicosimplicial objects:
\[
\theta\colon\coHH(\Omega^\bullet(M, D, N), C)^\bullet \stackrel{\cong}\longrightarrow \coHH(\Omega^\bullet(N, C, M), D)^\bullet.
\]
This is the dual to the Dennis-Waldhausen Morita argument.
Indeed, we define $\theta$ by the usual shuffling isomorphism in the bicosimplicial map below, in which the rows correspond to the two-sided cobar construction, while the columns correspond to the coHochschild complex. For ease of understanding, the diagrams below have been color-coded to illustrate this isomorphism via shuffling. 
\[\begin{tikzcd}
	\textcolor{rgb,255:red,214;green,153;blue,92}{M\otimes N} & \textcolor{rgb,255:red,45;green,200;blue,63}{M\otimes N\otimes C} & \textcolor{rgb,255:red,92;green,92;blue,214}{M\otimes N \otimes C^{\otimes 2}} & \cdots \\
	\textcolor{rgb,255:red,214;green,92;blue,92}{M\otimes D\otimes N} & \textcolor{rgb,255:red,48;green,192;blue,192}{M\otimes D\otimes N\otimes C} & \textcolor{rgb,255:red,255;green,118;blue,5}{M\otimes D\otimes N\otimes C^{\otimes 2}} & \cdots \\
	\textcolor{rgb,255:red,214;green,92;blue,214}{M\otimes D^{\otimes 2}\otimes N} & \textcolor{rgb,255:red,219;green,6;blue,73}{M\otimes D^{\otimes 2}\otimes N \otimes C} & \textcolor{rgb,255:red,245;green,204;blue,0}{M\otimes D^{\otimes 2}\otimes N \otimes C^{\otimes 2}} & \cdots \\
	\vdots & \vdots & \vdots \\
	\\
	\textcolor{rgb,255:red,214;green,153;blue,92}{N\otimes M} & \textcolor{rgb,255:red,214;green,92;blue,92}{N\otimes M\otimes D} & \textcolor{rgb,255:red,214;green,92;blue,214}{N\otimes M \otimes D^{\otimes 2}} & \cdots \\
	\textcolor{rgb,255:red,45;green,200;blue,63}{N\otimes C\otimes M} & \textcolor{rgb,255:red,48;green,192;blue,192}{N\otimes C\otimes M\otimes D} & \textcolor{rgb,255:red,219;green,6;blue,73}{N\otimes C\otimes M\otimes D^{\otimes 2}} & \cdots \\
	\textcolor{rgb,255:red,92;green,92;blue,214}{N\otimes C^{\otimes 2}\otimes M} & \textcolor{rgb,255:red,255;green,118;blue,5}{N\otimes C^{\otimes 2}\otimes M\otimes D} & \textcolor{rgb,255:red,245;green,204;blue,0}{N\otimes C^{\otimes 2}\otimes M\otimes D^{\otimes 2}} & \cdots \\
	\vdots & \vdots & \vdots
	\arrow[shift left=1, from=1-1, to=1-2]
	\arrow[shift right=1, from=1-1, to=1-2]
	\arrow[shift left=1, from=1-1, to=2-1]
	\arrow[shift right=1, from=1-1, to=2-1]
	\arrow[from=1-2, to=1-3]
	\arrow[shift left=2, from=1-2, to=1-3]
	\arrow[shift right=2, from=1-2, to=1-3]
	\arrow[shift right=1, from=1-2, to=2-2]
	\arrow[shift left=1, from=1-2, to=2-2]
	\arrow[shift left=1, from=2-1, to=2-2]
	\arrow[shift right=1, from=2-1, to=2-2]
	\arrow[shift right=1, from=1-3, to=2-3]
	\arrow[shift left=1, from=1-3, to=2-3]
	\arrow[from=2-2, to=2-3]
	\arrow[shift left=2, from=2-2, to=2-3]
	\arrow[shift right=2, from=2-2, to=2-3]
	\arrow[shift right=2, from=2-1, to=3-1]
	\arrow[from=2-1, to=3-1]
	\arrow[shift left=2, from=2-1, to=3-1]
	\arrow[shift left=1, from=3-1, to=4-1]
	\arrow[shift right=1, from=3-1, to=4-1]
	\arrow[shift right=3, from=3-1, to=4-1]
	\arrow[shift left=3, from=3-1, to=4-1]
	\arrow[shift right=1, from=3-1, to=3-2]
	\arrow[shift left=1, from=3-1, to=3-2]
	\arrow[shift left=2, from=2-2, to=3-2]
	\arrow[shift right=2, from=2-2, to=3-2]
	\arrow[from=2-2, to=3-2]
	\arrow[from=2-3, to=3-3]
	\arrow[shift left=2, from=2-3, to=3-3]
	\arrow[shift right=2, from=2-3, to=3-3]
	\arrow[from=3-2, to=3-3]
	\arrow[shift left=2, from=3-2, to=3-3]
	\arrow[""{name=0, anchor=center, inner sep=0}, shift right=2, from=3-2, to=3-3]
	\arrow[shift left=1, from=1-3, to=1-4]
	\arrow[shift left=3, from=1-3, to=1-4]
	\arrow[shift right=1, from=1-3, to=1-4]
	\arrow[shift right=3, from=1-3, to=1-4]
	\arrow[shift left=1, from=2-3, to=2-4]
	\arrow[shift left=3, from=2-3, to=2-4]
	\arrow[shift right=1, from=2-3, to=2-4]
	\arrow[shift right=3, from=2-3, to=2-4]
	\arrow[shift left=1, from=3-3, to=3-4]
	\arrow[shift right=1, from=3-3, to=3-4]
	\arrow[shift left=3, from=3-3, to=3-4]
	\arrow[shift right=3, from=3-3, to=3-4]
	\arrow[shift left=1, from=3-2, to=4-2]
	\arrow[shift right=1, from=3-2, to=4-2]
	\arrow[shift right=3, from=3-2, to=4-2]
	\arrow[shift left=3, from=3-2, to=4-2]
	\arrow[shift left=3, from=3-3, to=4-3]
	\arrow[shift right=3, from=3-3, to=4-3]
	\arrow[shift left=1, from=3-3, to=4-3]
	\arrow[shift right=1, from=3-3, to=4-3]
	\arrow[shift left=1, from=7-1, to=7-2]
	\arrow[shift right=1, from=7-1, to=7-2]
	\arrow[from=7-1, to=8-1]
	\arrow[shift left=1, from=6-1, to=6-2]
	\arrow[shift right=1, from=6-1, to=6-2]
	\arrow[shift left=1, from=6-1, to=7-1]
	\arrow[shift right=1, from=6-1, to=7-1]
	\arrow[shift right=2, from=7-1, to=8-1]
	\arrow[shift left=2, from=7-1, to=8-1]
	\arrow[shift left=1, from=6-2, to=7-2]
	\arrow[shift right=1, from=6-2, to=7-2]
	\arrow[from=6-2, to=6-3]
	\arrow[""{name=1, anchor=center, inner sep=0}, shift left=2, from=6-2, to=6-3]
	\arrow[shift right=2, from=6-2, to=6-3]
	\arrow[shift left=2, from=7-2, to=7-3]
	\arrow[shift right=2, from=7-2, to=7-3]
	\arrow[from=7-2, to=7-3]
	\arrow[shift left=2, from=8-2, to=8-3]
	\arrow[shift right=2, from=8-2, to=8-3]
	\arrow[from=8-2, to=8-3]
	\arrow[shift left=2, from=7-2, to=8-2]
	\arrow[shift right=2, from=7-2, to=8-2]
	\arrow[from=7-2, to=8-2]
	\arrow[shift left=1, from=8-1, to=8-2]
	\arrow[shift right=1, from=8-1, to=8-2]
	\arrow[from=7-3, to=8-3]
	\arrow[shift left=2, from=7-3, to=8-3]
	\arrow[shift right=2, from=7-3, to=8-3]
	\arrow[from=6-3, to=7-3]
	\arrow[shift left=2, from=6-3, to=7-3]
	\arrow[shift right=2, from=6-3, to=7-3]
	\arrow[shift left=3, from=6-3, to=6-4]
	\arrow[shift right=3, from=6-3, to=6-4]
	\arrow[shift left=1, from=6-3, to=6-4]
	\arrow[shift right=1, from=6-3, to=6-4]
	\arrow[shift left=1, from=7-3, to=7-4]
	\arrow[shift right=1, from=7-3, to=7-4]
	\arrow[shift left=3, from=7-3, to=7-4]
	\arrow[shift right=3, from=7-3, to=7-4]
	\arrow[shift left=1, from=8-3, to=8-4]
	\arrow[shift right=1, from=8-3, to=8-4]
	\arrow[shift left=3, from=8-3, to=8-4]
	\arrow[shift right=3, from=8-3, to=8-4]
	\arrow[shift left=3, from=8-1, to=9-1]
	\arrow[shift left=1, from=8-1, to=9-1]
	\arrow[shift right=1, from=8-1, to=9-1]
	\arrow[shift right=3, from=8-1, to=9-1]
	\arrow[shift left=1, from=8-2, to=9-2]
	\arrow[shift right=3, from=8-2, to=9-2]
	\arrow[shift right=1, from=8-2, to=9-2]
	\arrow[shift left=3, from=8-2, to=9-2]
	\arrow[shift left=1, from=8-3, to=9-3]
	\arrow[shift right=1, from=8-3, to=9-3]
	\arrow[shift right=3, from=8-3, to=9-3]
	\arrow[shift left=3, from=8-3, to=9-3]
	\arrow["\theta"{description}, shorten <=12pt, shorten >=12pt, from=0, to=1]
\end{tikzcd}\]
The $(i, j)$-spot in $\coHH(\Omega^\bullet(M, D, N), C)^\bullet$ is isomorphic to the $(j, i)$-spot in $\coHH(\Omega^\bullet(N, C, M), D)^\bullet$:
\[
\theta\colon M\otimes D^{\otimes i}\otimes N\otimes C^{\otimes j} \stackrel{\cong}\longrightarrow N\otimes C^{\otimes j}\otimes M\otimes D^{\otimes i},
\]
which incorporates a Koszul sign. Moreover, the map $\theta$ is compatible with the cofaces and codegeneracies.
 By \cite[18.5.3]{hir}, we obtain the equivalence
\[\theta\colon\coTHH(\Omega(M,D,N), C) \stackrel{\simeq}\longrightarrow \coTHH(\Omega(N,C,M), D),\]
using the model of homotopy limit as in \cite[18.1.8]{hir}, providing the desired quasi-isomorphism $\theta$.

Given a fibrant $(C,C)$-bicomodule $P$, we next must show that the following diagram commutes:
\[
\begin{tikzpicture}[baseline= (a).base]
\node[scale=0.73] (a) at (1,1){
\begin{tikzcd}[row sep= 50]
	{\coTHH(\Omega(\Omega(M,D,N), C, P), C)} & {\coTHH(\Omega(P,C,\Omega(M,D,N)), C)} & {\coTHH(\Omega(\Omega(P,C,M),D,N), C)} \\
	{\coTHH(\Omega(M, D, \Omega(N, C, P)),C)} & {\coTHH(\Omega(\Omega(N,C,P),C,M), D)} & {\coTHH(\Omega(N,C, \Omega(P,C, M)), D).}
	\arrow["{\lang a \rang}"', from=1-1, to=2-1]
	\arrow["\theta", from=1-1, to=1-2]
	\arrow["\theta"', from=2-1, to=2-2]
	\arrow["{\lang a \rang}", from=1-2, to=1-3]
	\arrow["{\lang a \rang}"', from=2-2, to=2-3]
	\arrow["\theta"', from=2-3, to=1-3]
\end{tikzcd}
};  
\end{tikzpicture}
\]
This follows a similar argument as above. We need to consider an isomorphism of ``tri-cosimplicial" isomorphisms in which we keep track of the swapping of the grading:
\[
\begin{tikzpicture}[baseline= (a).base]
\node[scale=0.63] (a) at (1,1){
\begin{tikzcd}[row sep= 50]
    \big((M\otimes D^{\otimes  k}\otimes N)\otimes C^{\otimes j}\otimes P\big)\otimes C^{\otimes i} \ar{r}{\theta} \ar{d}[swap]{\lang a\rang}& \big( P\otimes C^{\otimes i}\otimes (M\otimes  D^{\otimes k}\otimes  N)\big)\otimes C^{\otimes j} \ar{r}{\lang a \rang} & \big((P\otimes  C^{\otimes i}\otimes M) \otimes D^{\otimes k}\otimes N\big) \otimes C^{\otimes j}. \\
    \big( M\otimes D^{\otimes k} \otimes (N\otimes C^{\otimes j}\otimes P)\big)\otimes C^{\otimes i} \ar{r}[swap]{\theta} & \big((N\otimes C^{\otimes j}\otimes  P)\otimes C^{\otimes i}\otimes M\big)\otimes D^{\otimes k} \ar{r}[swap]{\lang a \rang} & \big(N\otimes C^{\otimes j}\otimes  (P\otimes C^{\otimes i}\otimes M)\big)\otimes D^{\otimes k} \ar{u}[swap]{\theta}
\end{tikzcd}
};  
\end{tikzpicture}
\]

We then need to check if the following diagram commutes:
\[
\begin{tikzpicture}[baseline= (a).base]
\node[scale=0.85] (a) at (1,1){
\begin{tikzcd}
	{\coTHH(\Omega(P,C, C),C)} && {\coTHH(\Omega(C, C, P),C)} && \coHH({\Omega(P, C, C),C)} \\
	\\
	&& {\coTHH(P, C)}.
	\arrow["{\lang \ell \rang}", from=1-3, to=3-3]
	\arrow["\theta", from=1-1, to=1-3]
	\arrow["{\lang r \rang}"', from=1-1, to=3-3]
	\arrow["\theta", from=1-3, to=1-5]
	\arrow["{\lang r \rang}", from=1-5, to=3-3]
\end{tikzcd}};  
\end{tikzpicture}
\]
We will prove the left triangle is commutative, but the argument for the right triangle will follow analogously. We apply the definition of $\theta$, $\ell$, and $r$ on the bicosimplicial objects:
\[
\begin{tikzcd}
    (P\otimes C^{\otimes \bullet}\otimes C)\otimes C^{\otimes \ast} \ar{d}[swap]{\lang r \rang}\ar{r}{\theta} & (C\otimes C^{\otimes \ast} \otimes P)\otimes C^{\otimes \bullet} \ar{d}{\lang \ell\rang } \\
    P\otimes C^{\otimes \ast} \ar{r}{\cong} & P\otimes C^{\otimes \bullet}.
\end{tikzcd}
\]
Here we kept track of the different gradings by  $\bullet$ and $\ast$. Therefore, once we apply the homotopy limit, we obtain the desired result. 
\end{proof}

\section{Dualizable bicomodules}\label{section: duality}

As traces are always defined on dualizable objects, we investigate the notion of duality in our bicategories of bicomodules.

\begin{defn}
Let $\mathcal{B}$ be a bicategory. A $1$-cell $M\in \mathcal{B}(C,D)$  is \emph{right dualizable} if there is another $1$-cell $M^\dual$ in $\mathcal{B}(D,C)$, with $2$-cells $\eta\colon U_C\rightarrow M\odot M^\dual$ and $\varepsilon\colon M^\dual\odot M\rightarrow U_D$, called \emph{coevaluation} and \emph{evaluation} respectively, such that the compositions in $\mathcal{B}(C,D)$ and $\mathcal{B}(D,C)$ respectively are the identity $2$-cells:
\[
\begin{tikzcd}
 	M \cong {U_C\odot M}\ar{r}{ \eta\odot \id} & {(M\odot M^\dual)\odot M \cong M\odot(M^\dual\odot M)} \ar{r}{\id\odot \varepsilon} & {M\odot U_D} \cong M,
\end{tikzcd}
\]
\[
\begin{tikzcd}[column sep=small]
 	M^\dual \cong {M^\dual\odot U_C}\ar{r}{\id\odot \eta} & {M^\dual\odot (M\odot M^\dual)\cong (M^\dual\odot M)\odot M^\dual} \ar{r}{\varepsilon\odot \id}  & {U_D\odot M^\dual} \cong M^\dual.
\end{tikzcd}
\]
We call $M^\dual$ the \emph{right dual} of $M$ and say that $(M, M^\dual)$ form a \emph{dual pair} in $\mathcal{B}$. 
\end{defn}

For instance, in the bicategory of bimodules, an $(R, S)$-bimodule $M$ is right dualizable if and only if it is finitely generated and projective as a right $S$-module, and its dual is given by its linear dual $\hom_S(M, S)$, see \cite[6.1]{Ponto_2012}. We shall obtain a very similar result for bicomodules with the subtlety that our bicategory is not ``closed" (see \cite[4.1.4]{ponto2008fixed}), and thus we cannot recognize dualizable objects (as in \cite[4.3.3]{ponto2008fixed}) since we are not provided with an internal hom. Our bicategories will almost be ``co-closed" thanks to the introduction of a cohom functor.
 We shall make use of the notion of ``quasi-finite" comodules, as introduced by Takeuchi in \cite{takeuchi}. More modern reviews, in more general settings, can be found in \cite{brzezinski2003corings} and \cite{khaled}.

Let ${}_C\bicomod{}{}$ be the category of left $C$-comodules. If $M$ is a left $C$-comodule and $V$ is a $\mathbbm{k}$-module, then $M\otimes V$ is a left $C$-comodule. 
This defines a functor:
\begin{align*}
    M\otimes - \colon \bimod{}{\mathbbm{k}} & \longrightarrow {}_C\bicomod{}{} \\
    V & \longmapsto M\otimes V.
\end{align*}

\begin{defn}\label{def: quasi-finite comodules}
 We say a left $C$-comodule $M$ is \emph{quasi-finite} if the functor $M\otimes - \colon \bimod{}{\mathbbm{k}}  \rightarrow {}_C\bicomod{}{} $ is a right adjoint. 
 In this case, we denote its left adjoint by $h_C(M,-)\colon {}_C\bicomod{}{}\rightarrow \bimod{}{\mathbbm{k}}$, and refer to it as the \textit{cohom} functor. 
 In other words, given any left $C$-comodule $N$, the cohom $h_C(M,N)$ is the universal $\mathbbm{k}$-module providing a natural $\mathbbm{k}$-linear isomorphism
\begin{equation}\label{eq: cohom adjunction}
\bimod{}{\mathbbm{k}}\Big( h_C(M, N), V\Big)\cong \bicomod{C}{}\Big(N, M\otimes V\Big),
\end{equation}
for any $\mathbbm{k}$-module $V$. In particular, for $V=\mathbbm{k}$, we obtain the isomorphism
\[
h_C(M,N)^*\cong  {}_C\hom(N, M).
\]
\end{defn}

\begin{ex}\label{ex: dualizable comodules with cofree}
An example of a quasi-finite left $C$-comodule $M$ is given by cofree comodules $M=C\otimes V$, where $V$ is a finitely generated $\mathbbm{k}$-module.
Since we have
\begin{align*}
\bimod{}{\mathbbm{k}}\Big( h_C(C\otimes V, N), W\Big) &\cong \bicomod{C}{}\Big (N, C\otimes V \otimes W \Big)  \\
& \cong \bimod{}{\mathbbm{k}}\Big( N, V\otimes W\Big)\\
& \cong \bimod{}{\mathbbm{k}}\Big( V^*\otimes N, W\Big),
\end{align*}
for any $\mathbbm{k}$-module $W$, we obtain $h_C(C\otimes V, N)\cong V^*\otimes N$.
Choosing $V=\mathbbm{k}^{\oplus n}$ so that $M=C^{\oplus n}$, then $h_C(C^{\oplus n}, N)\cong N^{\oplus n}$.
More generally, any finitely cogenerated comodule (Definition \ref{def: finitely cogenerated}) is quasi-finite.
\end{ex}

The unit of the adjunction \eqref{eq: cohom adjunction} provides the $\mathbbm{k}$-linear coevaluation
\begin{align*}
\eta\colon N\longrightarrow M\otimes h_C(M, N).
\end{align*}
If $D$ is another coalgebra, and $N$ is a $(C,D)$-bicomodule, then $h_C(M,N)$ is a right $D$-comodule, and the map $\eta$ is a morphism of $(C,D)$-bicomodules, see \cite[1.7, 1.19]{takeuchi}. Therefore, for any quasi-finite left $C$-comodule $M$, we obtain that $h_C(M,C)$ is a right $C$-comodule, and we have a $(C,C)$-bicomodule homomorphism $\eta\colon C\rightarrow M\otimes h_C(M, C)$. 

Let $T$ be a left $C$-comodule, $M$ a quasi-finite left $C$-comodule, and $N$ a $(C,C)$-bicomodule. 
Denote by $\partial\colon h_C(M, N\square_C T)\rightarrow h_C(M,N)\square_C T$  the adjoint to the left $C$-colinear homomorphism
\[
\eta \square 1 \colon N\square_C T \longrightarrow \Big(M\otimes h_C(M, N)\Big)\square_C T \cong M\otimes \Big(h_C(M,N)\square_C T\Big).
\]
It is noted that, in \cite[1.14]{takeuchi}, the $\mathbbm{k}$-linear homomorphism $\partial$ above is an isomorphism whenever $M$ is an injective quasi-finite left $C$-comodule. 
If we choose $N=C$ and $T=M$, 
combining the isomorphism with the map $h_C(M,M)\rightarrow \mathbbm{k}$ of Definition \ref{def: coendomorphism coalgebra}, we obtain the evaluation
\[
\begin{tikzcd}
\varepsilon\colon h_C(M,C)\square_C M& h_C(M,M) \ar{l}{\cong}[swap]{\partial}\ar{r} & \mathbbm{k} .
\end{tikzcd}
\]

\begin{ex}\label{ex: dualizable quasi-finite case}
Consider the bicategory $\bicomod{}{}$ of Definition \ref{def: bicategory of discrete bicomodules}.
A comodule $M$ in $\bicomod{C}{\mathbbm{k}}$ is right duali\-zable if it is quasi-finite and injective as a left $C$-comodule. 
The dual of $M$ is the right $C$-comodule $h_C(M,C)$, together with the coevaluation $\eta\colon C\rightarrow M\otimes h_C(M,C)$ and evaluation $\varepsilon\colon h_C(M,C)\square_C M\rightarrow \mathbbm{k}$ defined above. The desired triangle identities follow from the adjunction of the cohom functor \eqref{eq: cohom adjunction}.
\end{ex}

\begin{ex}\label{ex: dualizable comodules with cofree bicategory}
In Example \ref{ex: dualizable quasi-finite case}, choose $M=C\otimes V$, where $V$ is a dualizable $\mathbbm{k}$-module (i.e. finitely generated).
We obtain $M^\dual=h_C(C\otimes V, C)\cong V^*\otimes C$. 
We can then explicitly describe the coevaluation and evaluation. 
Let us denote the comultiplication and counit of $C$ by $\Delta_C\colon C\rightarrow C\otimes C$ and $\varepsilon_C\colon C\rightarrow \mathbbm{k}$ respectively.
Similarly, denote the coevaluation and evaluation of the dualizable $\mathbbm{k}$-module $V$ by $\eta_V\colon\mathbbm{k}\rightarrow V\otimes V^*$ and $\varepsilon_V\colon V^*\otimes V\rightarrow \mathbbm{k}$ respectively.
Then the coevaluation $\eta\colon C\rightarrow M\otimes M^\dual\cong (C\otimes V)\otimes (V^* \otimes C)$ is the composite
\[
\begin{tikzcd}
C \ar{r}{\Delta_C}& C\otimes C\cong C\otimes \mathbbm{k} \otimes C\ar{r}{\id\otimes \eta_V\otimes \id} & [2em] C\otimes (V\otimes V^*)\otimes C. 
\end{tikzcd}
\]
Explicitly, if we pick $(e_1, \ldots, e_n)$ for a basis of $V$, and $(e^*_1, \ldots, e_n^*)$ for the dual basis of $V^*$, and write $\Delta_C(c)=\sum_{(c)} c_{(1)}\otimes c_{(2)}$, we obtain the formula
\[
\eta(c)=\sum_{i=1}^n \sum_{(c)} c_{(1)}\otimes  e_i\otimes e^*_i \otimes c_{(2)}.
\]
The evaluation $\varepsilon\colon M^\star\square_C M\cong V^*\otimes C\otimes V\rightarrow \mathbbm{k}$ is then the composite
\[
\begin{tikzcd}
 V^*\otimes C\otimes V \ar{r}{\id\otimes \varepsilon_C \otimes \id} & [2em] V^*\otimes \mathbbm{k} \otimes V \cong V^*\otimes V \ar{r}{\varepsilon_V} & \mathbbm{k}.
\end{tikzcd}
\]
Explicitly, we obtain the formula for $f\in V^*, c\in C, v\in V$:
\[
\varepsilon(f\otimes c\otimes v)=\varepsilon_C(c)f(v).
\]
Choosing in particular $V=\mathbbm{k}^{\oplus n}$ so that $M=C^{\oplus n}$ and $M^\star\cong C^{\oplus n}$ and $V\otimes V^*\cong  \mathcal{M}_n(\mathbbm{k})$, we  get $C^{\oplus n}\otimes C^{\oplus n}\cong \mathcal{M}_n^c(C)\otimes C$ and the coevaluation can be formulated as the $(C,C)$-bicolinear morphism:
\begin{align*}
    \eta\colon C & \longrightarrow \mathcal{M}_n^c(C)\otimes C\\
    c & \longmapsto  \sum_{(c)} c_{(1)}\otimes I_n \otimes c_{(2)}.
\end{align*}
Notice that $M^\star \square_C M\cong \mathcal{M}_n^c(C)$ as $\mathbbm{k}$-modules, and thus the evaluation can be formulated as a $\mathbbm{k}$-linear homomorphism:
\begin{align*}
    \varepsilon\colon \mathcal{M}_n^c(C) & \longrightarrow \mathbbm{k}\\
    c\otimes f & \longmapsto \varepsilon_C(c)\tr(f).
\end{align*}
\end{ex}

\begin{ex}\label{ex: dualizable comodule (general)}
We now generalize Example \ref{ex: dualizable quasi-finite case}.
A comodule $M$ in $\bicomod{C}{D}$ is right dualizable whenever it is quasi-finite and injective as a left $C$-comodule. 
Define $M^\dual=h_C(M,C)$, which is a $(D,C)$-bicomodule. Indeed, by \cite[1.8]{takeuchi}, for any $(C,D)$-bicomodule $M$ such that $M$ is quasi-finite as a left $C$-comodule, and any left $C$-comodule $N$, we have that $h_C(M,N)$ is a left $D$-comodule with coaction $h_C(M,N)\rightarrow D\otimes h_C(M,N)$ given by the adjoint to the left $C$-colinear map
\[
\begin{tikzcd}
N \ar{r}{\eta} & M\otimes h_C(M,N) \ar{r}{\rho\otimes \id} & M\otimes D\otimes h_C(M,N).
\end{tikzcd}
\]
Moreover, by \cite[1.7]{takeuchi}, the $C$-bicolinear map $\eta\colon C\rightarrow M\otimes h_C(M,C)$ factors through the $(C,C)$-bicomodule $M\square_D h_C(M,C)$, and the induced map $\eta\colon C\rightarrow M\square_D h_C(M,C)$ remains $C$-bicolinear by \cite[1.9]{takeuchi}. Therefore this defines the desired $C$-bicolinear coevaluation $\eta:C\rightarrow M\square_D M^\dual$.
In fact, by \cite[1.10]{takeuchi}, the cohom functor $h_C(M,-)\colon \bicomod{C}{}\rightarrow\bimod{}{\mathbbm{k}}$ can be promoted to a functor $h_C(M,-)\colon \bicomod{C}{}\rightarrow\bicomod{D}{}$ whenever the quasi-finite left $C$-comodule $M$ is endowed with a right $D$-coaction, and it is the left adjoint of the functor
\begin{align*}
\bicomod{D}{} & \longrightarrow \bicomod{C}{}\\
L & \longrightarrow M\square_D L,
\end{align*}
i.e., we obtain an equivalence
\begin{equation}\label{eq:cohom adjunction promoted}
\bicomod{D}{}\Big( h_C(M,N), L \Big) \cong \bicomod{C}{}\Big(N, M\square_D L\Big)
\end{equation}
for any left $C$-comodule $N$ and left $D$-comodule $L$. In particular, if we choose $M=N$ and $L=D$, the adjoint of the identity map on $M$ provides a left $D$-colinear map $h_C(M,M)\rightarrow D$, which is in fact $D$-bicolinear by \cite[1.11]{takeuchi}.
Moreover, by \cite[1.15]{takeuchi}, the map
\[
\begin{tikzcd}
 h_C(M,M)\cong h_C(M, C\square_C M) \ar{r}{\partial}[swap]{} & h_C(M,C)\square_C M,
\end{tikzcd}
\]
is $D$-bicolinear. When $M$ is injective as a left $C$-comodule, the map $\partial$ is an isomorphism. Therefore, the desired evaluation $\varepsilon\colon M^\dual\square_C M\rightarrow D$ is given by composing the previous $D$-bicolinear maps
\[
\begin{tikzcd}
 h_C(M,C)\square_C M \ar{r}{\partial^{-1}}[swap]{\cong} & h_C(M,  C\square_C M)\cong h_C(M,M)\ar{r} & D.
\end{tikzcd}
\]
Just as in Example \ref{ex: dualizable quasi-finite case}, the triangle identities follow from the adjunction of the cohom functor \eqref{eq:cohom adjunction promoted}.
\end{ex}

\begin{ex}\label{ex: dualizable on other side}
A particular case of Example \ref{ex: dualizable comodule (general)} is when $C=\mathbbm{k}$. A bicomodule $M$ in $\bicomod{\mathbbm{k}}{D}$ is right dualizable when $M$ is dualizable as a $\mathbbm{k}$-module (i.e., finitely generated), and its right dual is $M^*=\hom_\mathbbm{k}(M,\mathbbm{k})$, the usual linear dual. This appeared in the proof of Proposition \ref{prop: colinear is a cotrace}.
Explicitly, its left $D$-coaction is given by:
\begin{align*}
   \lambda\colon M^* & \longrightarrow \hom_\mathbbm{k}(M,D) \cong D\otimes M^*\\
    \left(M\stackrel{\alpha}\rightarrow \mathbbm{k}\right) & \longmapsto \left(M\stackrel{\rho}\rightarrow M\otimes D \stackrel{\alpha\otimes \id}\rightarrow D\right),
\end{align*}
where $\rho\colon M\rightarrow M\otimes D$ is the right $D$-coaction of $M$.
By \cite[10.11]{brzezinski2003corings}, we have an isomorphism of $\mathbbm{k}$-modules $M\square_D M^*\cong \hom_D(M,M)$. 
Therefore, we can interpret the coevaluation $\eta\colon\mathbbm{k}\rightarrow M\square_D M^*$ as:
\begin{align*}
    \mathbbm{k} & \longrightarrow \hom_D(M, M)\\
    1 & \longmapsto \id_M.
\end{align*}
Notice that the following diagram commutes:
\[
\begin{tikzcd}
 M^*\otimes M \ar{r}{\lambda\otimes \id}\ar{d}[swap]{\id\otimes \rho} & \hom_\mathbbm{k}(M,D)\otimes M \ar{d}\\
 M^*\otimes M \otimes D \ar{r} & D,
\end{tikzcd}
\]
where the unlabeled maps are defined by evaluating. This defines the $D$-bicolinear evaluation homomorphism $\varepsilon\colon M^*\otimes M\rightarrow D$:
\[
\varepsilon(\alpha \otimes  m)=\sum_{(m)} \alpha(m_{(0)})\otimes m_{(1)}
\]
where we denoted $\rho(m)=\sum_{(m)} m_{(0)}\otimes m_{(1)}$ for all $m\in M$.
\end{ex}

We can further generalize the examples above to the differential graded case. Arguments for quasi-finite comodules are entirely categorical (see \cite{khaled}). However, the notion has not been documented before in this context, and therefore we introduce them here.

\begin{defn}
Let $C$ be a simply connected dg-coalgebra over $\mathbbm{k}$. 
A left dg-$C$-comodule $M$ is \emph{quasi-finite} if the functor
\begin{align*}
\Ck & \longrightarrow \bicomod{C}{}\\
V & \longmapsto M\otimes V
\end{align*}
admits a left adjoint, denoted $h_C(M, -)\colon \bicomod{C}{}\rightarrow \Ck$, called the \emph{cohom functor}.
In other words, we obtain an equivalence
\[
\Ck\Big( h_C(M, N), V\Big)\cong \bicomod{C}{}\Big(N, M\otimes V\Big),
\]
for any left dg-$C$-comodule $N$ and $\mathbbm{k}$-chain complex $V$.
\end{defn}

\begin{ex}
If $V$ is a perfect $\mathbbm{k}$-chain complex, then just as in Example \ref{ex: dualizable comodules with cofree}, we have that $C\otimes V$ is quasi-finite, and $h_C(C\otimes V, N)\cong V^*\otimes N$. If $C=\mathbbm{k}$, then a left dg $C$-comodule is just a chain complex, and quasi-finite comodules are precisely the perfect chain complexes.
\end{ex}

If $M$ is also endowed with a right $D$-coaction, then $h_C(M,N)$ is a right $D$-comodule and we obtain the adjunction
\[
\begin{tikzcd}
 \bicomod{C}{} \ar[shift left=2]{r}{h_C(M,-)} & [2em] \bicomod{D}{} \ar[shift left =2]{l}{M\square_D-}[swap]{\perp}.
\end{tikzcd}
\]
Therefore, just as in the discrete case, the right dual of a $(C,D)$-bicomodule $M$ should be the $(D,C)$-bicomodule $h_C(M,C)$. 
Notice, even if $M$ is a fibrant $(C,D)$-bicomodule, there is no reason to expect $h_C(M,C)$ is fibrant as a $(D,C)$-bicomodule. However, this is not needed. Indeed, by Proposition \ref{Prop: cotensor preserve weak equivalence if fibrant}, if $M$ is a fibrant $(C,D)$-bicomodule, then
\begin{align*}
    M\square_D h_C(M,C) & \simeq M\square_D \Omega(D, D, h_C(M,C))\\
    & \simeq \Omega(M, D, h_C(M,C))\\
    & \simeq  M\dcotensor_D h_C(M,C).
\end{align*}
Similarly, we obtain $h_C(M,C)\square_C M\simeq h_C(M,C)\dcotensor_C M$.

Just as in the discrete case, the $C$-bicolinear map $\eta\colon C\rightarrow M\square_D h_C(M,C)$ from the unit of the cohom adjunction induces the desired coevaluation on the dual pair $(M, h_C(M,C))$.

We now describe how to obtain the evaluation.
Recall that a $\mathbbm{k}$-chain complex $V$ is perfect if and only if the natural map $V\otimes \hom_\mathbbm{k}(V, \mathbbm{k})\rightarrow \hom_\mathbbm{k}(V, V)$ is an isomorphism. Similarly, as in Example \ref{ex: dualizable quasi-finite case}, for any quasi-finite left $C$-comodule $M$, we have a natural $D$-bicolinear map $\partial\colon h_C(M, M)\rightarrow h_C(M, C)\square_C M$. We saw in the discrete case that if $M$ is injective and quasi-finite, then $\partial$ is an isomorphism. 

\begin{defn}\label{def: coperfect dg-comodules}
A fibrant left $C$-comodule $M$ is said to be \emph{coperfect} if it is quasi-finite and the induced map $\partial\colon h_C(M, M)\rightarrow h_C(M, C)\square_C M$ is an isomorphism.
\end{defn}

Given a coperfect $(C,D)$-bicomodule $M$, define the evaluation as
\[
\begin{tikzcd}
 \varepsilon\colon h_C(M,C)\square_C M \ar{r}{\partial^{-1}}[swap]{\cong} &  h_C(M,M)\ar{r} & D.
\end{tikzcd}
\]
Combining our arguments above, we obtain the desired triangle identities by adjunction of the cohom functor, and thus the following result.

\begin{prop}
Let $C$ and $D$ be simply connected dg-coalgebras. 
A (fibrant) $(C,D)$-bicomodule $M$ is right dualizable if it is coperfect as a left $C$-comodule. Its right dual is given by $h_C(M,C)$.
\end{prop}

\begin{ex} 
Suppose $D=\mathbbm{k}$, and $M=C\otimes V$, where $V$ is a perfect chain complex. Then $M$ is quasi-finite as a left $C$-comodule and $M^\dual=h_C(M,C)\cong V^*\otimes C$. 
Moreover, $M$ is coperfect because
\[
V^*\otimes M\cong h_C(C\otimes V ,M) \rightarrow h_C(C\otimes V, C)\square_C M\cong V^*\otimes M
\]
is an isomorphism. 
\end{ex}

\begin{ex}
    If $C=\mathbbm{k}$, then if $M$ is a right $D$-comodule such that $M$ is a perfect chain complex, then $M$ is dualizable and $M^*=\hom_\mathbbm{k}(M,\mathbbm{k})$. 
\end{ex}

\section{Bicategorical traces}\label{section: traces}
Every shadowed bicategory defines a notion of traces on its dualizable objects. This extends the notion of traces on symmetric monoidal categories, as reviewed in \cite{traceinSMon}. 
We show the Hattori--Stallings cotrace ${}_C\Endo(M)\rightarrow \coHH_0(C)^*$ (Definition \ref{def: the hattori-stallings cotrace}) and the colinear Hattori--Stallings trace $\Endo_C(M)\rightarrow \coHH_0(C)$ (Definition \ref{def: colinear Hattori--Stallings trace}) are bicategorical traces on the shadowed bicategory $(\bicomod{}{}, \coHH_0)$.
We first recall the general definition of a bicategorical trace.

\begin{defn}[{\cite{Ponto_2012}}]
Let $(\mathcal{B}, \lang - \rang)$ be a shadowed bicategory. Let $M\in \mathcal{B}(C, D)$ be a right dualizable $1$-cell.
The \emph{trace} of a $2$-cell $f:M\rightarrow M$, denoted $\tr^\mathcal{B}(f)$, is the composite
\[
\begin{tikzpicture}[baseline= (a).base]
\node[scale=1] (a) at (1,1){
\begin{tikzcd}
 {\lang U_C \rang}_C \ar{r}{\lang \eta \rang} & \lang M\odot M^\dual \rang_C \ar{r}{\lang f\odot \id \rang} & \lang M\odot M^\dual \rang_C \ar{r}{\theta}[swap]{\cong} & [-1em]\lang M^\dual\odot M\rang_D \ar{r}{\lang \varepsilon \rang} & \lang U_D \rang_D. 
\end{tikzcd}
};  
\end{tikzpicture}
\]
More generally, given $1$-cells $P\in \mathcal{B}(D, D)$ and $Q\in \mathcal{B}(C,C)$, the trace of a $2$-cell $f:Q\odot M \rightarrow M\odot P$ is the composite
\[
\begin{tikzpicture}[baseline= (a).base]
\node[scale=1] (a) at (1,1){
\begin{tikzcd}
 {\lang Q \rang}_C \ar{r}{\lang \id\odot \eta \rang} & \lang Q\odot M\odot M^\dual \rang_C \ar{r}{\lang f\odot \id \rang} & \lang M\odot P \odot M^\dual \rang_C \ar{r}{\theta}[swap]{\cong} & [-1em]\lang M^\dual\odot M\rang_D \ar{r}{\lang \varepsilon \odot \id \rang} & \lang P \rang_D.
\end{tikzcd}
};  
\end{tikzpicture}
\]
The \emph{Euler characteristic} of $M$, denoted $\chi(M),$ is the trace of its identity $2$-cell $\tr^\mathcal{B}(\id_M)$.
\end{defn}

Bicategorical traces enjoy many useful properties that are recorded in \cite[Section  7]{Ponto_2012}. Notably, we obtain the following cyclicity property.

\begin{prop}[{\cite[7.3]{Ponto_2012}}]\label{prop: cyclic trace}
 If $M$ and $N$ are right dualizable $1$-cells in a shadowed bicategory $\mathcal{B}$ and if $f\colon M\rightarrow N$ and $g\colon N\rightarrow M$ are $2$-cells in $\mathcal{B}$, then $\tr^\mathcal{B}(f\circ g)=\tr^\mathcal{B}(g\circ f)$.   
\end{prop}

\begin{ex}[{\cite[4.2.2]{ponto2008fixed}}]\label{ex: Hattori-Stallings trace as a trace of shadow}
Let $R$ be a ring and let $M$ be a finitely generated projective right $R$-module. Let $f$ be an $R$-linear endomorphism on $M$. The Hattori--Stalling trace (Definition \ref{def: hattori-stallings definition}) can be regarded as the image of $1$ in the bicategorical trace $\mathsf{tr}(f)\colon \mathbb{Z}\rightarrow \HH_0(R)$:
\[
\begin{tikzcd}
    \mathbb{Z} \ar{r}{\eta} & M\otimes_R \hom_R(M, R) \ar{r}{f\otimes 1} & M\otimes_R \hom_R(M,R) \cong \HH_0(R,  \hom_R(M,R)\otimes_\mathbb{Z} M) \ar{r}{\lang\varepsilon\rang} & \HH_0(R).
\end{tikzcd} 
\]
\end{ex}

Given $\mathbbm{k}$-coalgebras $C$ and $D$, and $M$ a right dualizable $(C,D)$-bicomodule with dual $M^\star$, for any $(C,D)$-bicolinear homomorphism $f\colon M\rightarrow M$, we obtain a trace $\tr_C^D(f)\colon \coHH_0(C)\rightarrow \coHH_0(D)$ defined as the composition:
\[
\begin{tikzcd}
    \coHH_0(C) \ar{r}{\lang \eta \rang} & [-1em]\coHH_0(M\square_D M^\star, C) \ar{r}{\lang f\square 1\rang} & \coHH_0(M\square_DM^\star, C) \cong \coHH_0(M^\star\square_C M, D) \ar{r}{\lang \varepsilon \rang} & [-1em]\coHH_0(D).
\end{tikzcd}
\]
We show how the bicategorical trace recovers the Hattori--Stallings cotrace of Definition \ref{def: the hattori-stallings cotrace}.
Recall from Example \ref{ex: dualizable quasi-finite case} that a finitely cogenerated and injective left $C$-comodule is a right dualizable $(C,\mathbbm{k})$-bicomodule.

\begin{prop}\label{prop: hs-cotrace is a bicategorical trace}
    Let $C$ be a $\mathbbm{k}$-coalgebra. Let $M$ be a finitely cogenerated and injective left $C$-comodule. 
    The Hattori--Stallings cotrace ${}_C\Endo(M)\rightarrow \coHH_0(C)^*$ sends a $C$-colinear homomorphism $f\colon M\rightarrow M$ to the bicategorical trace $\tr_C^\mathbbm{k}(f)$:
    \[
\begin{tikzcd}
    \coHH_0(C) \ar{r}{\lang \eta \rang} & [-1em]\coHH_0(M\otimes h_C(M,C), C) \ar{r}{\lang f\otimes 1\rang} & \coHH_0(M\otimes h_C(M,C), C) \cong h_C(M,C)\square_C M \ar{r}{ \varepsilon } & [-1em]\mathbbm{k}.
\end{tikzcd}
\]
\end{prop}

\begin{proof}
    Suppose first $M=C^{\oplus n}$. 
    Then the  bicategorical trace  $\tr_C^\mathbbm{k}(f)$ can be re-written as the composition:
 \[
\begin{tikzcd}
    \coHH_0(C) \ar{r}{\lang \eta \rang} & [-1em]\coHH_0(\mathcal{M}_n^c(C)\otimes C, C) \ar{r}{\lang f\otimes 1\rang} & \coHH_0(\mathcal{M}_n^c(C)\otimes C, C) \cong \mathcal{M}^c_n(C) \ar{r}{ \varepsilon } & [-1em]\mathbbm{k}.
\end{tikzcd}
\]
Recall given $f\colon M\rightarrow M$, we denoted the corresponding $\mathbbm{k}$-linear   homomorphism $C\rightarrow \mathcal{M}_n(\mathbbm{k})$ by $c\mapsto f_c$. 
    
    Given $c\in \coHH_0(C)\subseteq C$, following  the formulas of Example \ref{ex: dualizable comodules with cofree bicategory}, we obtain that the image of $c$ under the composition is given by:
    \[
    c\mapsto \sum_{(c)} c_{(1)}\otimes I_n \otimes c_{(2)} \mapsto \sum_{(c)} c_{(1)}\otimes f_{c_{(2)}}\otimes c_{(3)} \mapsto \sum_{(c)} c_{(1)}\otimes f_{c_{(2)}} \mapsto \sum_{(c)}\varepsilon(c_{(1)})\tr(f_{c_{(2)}}).
    \]
   Therefore $\tr_C^\mathbbm{k}(f)(c)=\cotr_M(f)(c)$ for $M=C^{\oplus n}$. 
    We can conclude for a general finitely cogenerated and injective comodule $M$ by applying Lemma \ref{lem: extending  fin cogenerated to cofree}.
\end{proof}

Similarly, we can show the bicategorical trace recovers the colinear Hattori--Stallings trace of Definition \ref{def: colinear Hattori--Stallings trace}.
From Example \ref{ex: dualizable on other side}, we know that right $C$-comodules that are finitely generated as $\mathbbm{k}$-modules are right dualizable $(\mathbbm{k}, C)$-bicomodules.

\begin{prop}\label{prop: colinear hs is bicategorical}
    Let $C$ be a $\mathbbm{k}$-coalgebra.
    Let $M$ be a right $C$-comodule that is finitely generated as a $\mathbbm{k}$-module. 
    The colinear Hattori--Stallings trace $\Endo_C(M)\rightarrow \coHH_0(C)$ sends a $C$-colinear homomorphism $f\colon M\rightarrow M$ to the image of $1\in \mathbbm{k}$ in the bicategorical trace $\tr_\mathbbm{k}^C(f)$:
    \[
    \begin{tikzcd}
        \mathbbm{k} \ar{r}{\eta} & M\square_C M^* \ar{r}{f\square 1} & M\square_C M^* \cong \coHH_0(M^*\otimes M, C) \ar{r}{\lang  \varepsilon \rang}  & \coHH_0(C).
    \end{tikzcd}
    \]
\end{prop}

\begin{proof}
We denote the right coaction $M\rightarrow M\otimes C$ by $m\mapsto \sum_{(m)} m_{(0)}\otimes m_{(1)}$.
 Choosing a basis $(e_1, \dots, e_n)$ of $M$, and applying the formulas from Example \ref{ex: dualizable on other side}, we obtain that the image of $1$ in the composition is given by:
 \[
 1 \mapsto \sum_{i=1}^n e_i\otimes e_i^* \mapsto \sum_{i=1}^n  f(e_i)\otimes e_i^* \mapsto \sum_{i=1}^n e_i^*\otimes  f(e_i) \mapsto \sum_{i=1}^n \sum_{(f(e_i))} e_i^*(f(e_i)_{(0)})\otimes f(e_i)_{(1)}.
 \]
As $f$ is $C$-colinear, we obtain in particular:
\[
\sum_{(f(e_i))} f(e_i)_{(0)}\otimes f(e_i)_{(1)} = \sum_{(e_i)} f({e_i}_{(0)}) \otimes {e_i}_{(1)}.
\]
Therefore, we can conclude $\tr_\mathbbm{k}^C(f)(1)=\tr^C(f)$ as in Definition \ref{def: colinear Hattori--Stallings trace}.
\end{proof}

\begin{rem}
In recent independent work of Justin Barhite  \cite{justin}, Hochschild cohomology $\HH^0$ is shown to be a so-called \textit{coshadow} on the bicategory of bimodules over $\mathbbm{k}$-algebras, and there are associated \textit{bicategorical cotraces}. 
If $C$ is a $\mathbbm{k}$-coalgebra that is finitely generated as $\mathbbm{k}$-module, then the bicategory of bimodules over $C^*$ is intimately related to the bicategory of bicomodules over $C$ since $M\square_C N\cong {}_{C^*}\hom_{C^*}(C^*, M\otimes N)$ by \cite[10.10]{brzezinski2003corings}.
This induces an equivalence $\coHH_0(N\otimes M, C)\cong \HH^0(C^*, M\otimes N)$.
From our perspective however, our cotrace comes from a shadow and not a coshadow, and we do not assume that our bicategory is closed. 
One can build a functor $(\bicomod{}{}^\mathrm{fg})^\mathrm{co-op}\rightarrow \bimod{}{}$ from the co-opposite bicategory (inverting both 1 and 2 cells) of finitely generated bicomodules to the bicategory of bimodules, and we expect that it sends the shadow $\coHH_0$ to the coshadow $\HH^0$.
Nonetheless, it does not seem that the language of coshadows and bicategorical cotraces is the appropriate vocabulary for the Hattori--Stallings cotrace, nor the colinear Hattori--Stallings trace.
\end{rem}

\begin{ex}
We can generalize the Hattori--Stallings cotrace for chain complexes.
Let $C$ be a simply connected dg-coalgebra over $\mathbbm{k}$, 
let $M$ be a coperfect left $C$-comodule, and let $f$ be an $C$-colinear endomorphism on $M$.
Then we obtain the cotrace $\coHH(C)\rightarrow \mathbbm{k}$.
Since the homomorphism is in $\Ck$, and $\mathbbm{k}$ is concentrated in degree zero, then the above trace is entirely determined by $\coHH_0(C)$.
Unfortunately, as $C$ is simply connected, $\coHH_0(C)=\mathbbm{k}$, and this cotrace is simply the alternating sum of the usual trace of $f$ on each degree.
Similarly, we can also generalize the colinear Hattori--Stallings trace for chain complexes. Letting $V$ denote a perfect chain complex with a right $C$-comodule structure, given a $C$-colinear endomorphism $f\colon V\rightarrow V$, its colinear trace defines a chain homomorphism $\mathbbm{k}\rightarrow \coHH(C)$. It is non-trivial only in degree zero, but as $C$ is simply connected, we get that the colinear trace is the alternating sum of the trace of $f$ on each degree.
\end{ex}

\begin{ex}
Let $X$ be a simply connected CW-complex, and let $Y$ be a finite CW-complex.
Let $C_*(-;\mathbbm{k})=C_*(-)$ denote the singular chain complex over $\mathbbm{k}$ associated to a space. Recall that $C_*(X)$ is a dg-coalgebra over $\mathbbm{k}$ using Alexander--Whitney formula and the diagonal $X\rightarrow X\times X$. Then $C_*(X\times Y)\cong C_*(X)\otimes C_*(Y)$ is a coperfect $C_*(X)$-comodule. 
Let $f\colon Y\rightarrow Y$ be an endomorphism. This defines an endomorphism on $X\times Y$ that is the identity on $X$, and thus an endomorphism $\widehat{f}$ on the comodule $C_*(X\times Y)$. 
Therefore the cotrace of $\widehat{f}$ is a map $\coHH(C_*(X))\rightarrow \mathbbm{k}$ in which its image is the usual Lefschetz number of $f$.
\end{ex}

\begin{ex}
    Let $X$ be a simply connected CW-complex, and let $Y$ be a finite CW-complex with a continuous map $\rho:Y\rightarrow X$. 
    Then $C_*(Y)$ is a comodule over $C_*(X)$. Let $f$ be an endomorphism over $Y$. Then the colinear trace induced on coHochschild homology $\mathbbm{k}\rightarrow \coHH(C_*(X))$ is the usual alternating sum of the traces of $C_*(Y)\stackrel{f}\rightarrow C_*(Y)\stackrel{\rho}\rightarrow C_*(X)$.
\end{ex}

\section{Morita--Takeuchi invariance}\label{section: morita}

In bicategories, Morita equivalence is the natural notion of an equivalence in a bicategory. It extends the usual notion of equivalence of categories  and Morita equivalence between rings.

\begin{defn}[{\cite[4.1, 4.2]{campbell2019topological}}]
Let $\mathcal{B}$ be a bicategory, let $M\in \mathcal{B}(C,D)$ be a right dualizable $1$-cell, and denote its right dual by $M^\dual$.
We say the dual pair $(M, M^\dual)$ is a \emph{Morita equivalence} in  $\mathcal{B}$ if the coevaluation $\eta\colon U_C\rightarrow M\odot M^\dual$ and evaluation $\varepsilon\colon M^\dual\odot M\rightarrow D$ maps are isomorphisms in $\mathcal{B}(C,C)$ and $\mathcal{B}(D,D)$ respectively.
If such a pair exists, we say $C$ and $D$ are \emph{Morita equivalent}.
\end{defn}

\begin{prop}[{\cite[4.5]{campbell2019topological}}]\label{prop: morita invariance for shadows}
Let $(\mathcal{B}, \lang - \rang)$ be a shadowed bicategory. Let $M\in \mathcal{B}(C,D)$ be a right dualizable $1$-cell and denote its right dual by  $M^\dual$. If $(M, M^\dual)$ is a Morita equivalence, then the Euler characteristic $\chi(M)$ is an isomorphism, with inverse $\chi(M^\dual)$. In particular, if $C$ and $D$ are Morita equivalent, then \[\lang U_C\rang_C \cong \lang U_D \rang_D.\]
\end{prop}

\begin{ex}
Consider the bicategory $\bicomod{}{}$ of bicomodules over $\mathbbm{k}$ as in Definition \ref{def: bicategory of discrete bicomodules}.
Then in \cite[2.3]{takeuchi}, a Morita equivalence is referred to as a \emph{set of equivalence data}. It is then shown than a Morita equivalence in this bicategory recovers the notion of Morita-Takeuchi equivalence in $\mathbbm{k}$-modules (Definition \ref{defn: Morita-Takeuchi equivalence}).
Therefore, if $C$ and $D$ are Morita--Takeuchi equivalent, then $\coHH_0(C)\cong \coHH_0(D)$ by Proposition \ref{prop: morita invariance for shadows}. This recovers the results of \cite{moritatakinvariance} at level zero.
By \cite[3.5]{takeuchi}, if $M$ is a left $C$-comodule that is a quasi-finite injective cogenerator, and $D:=e_C(M)$, the coalgebra of coendomorphisms, then $C$ and $D$ are Morita-Takeuchi equivalent, and $M$ and $M^\dual$ form a Morita equivalence in the bicategory $\bicomod{}{}$.
\end{ex}

Consider the bicategory of derived bicomodules $\dcomod$ as in Definition \ref{def: bicategory of derived bicomodules}. 
We say two simply connected coalgebras $C$ and $D$ in $\Ck$ are \textit{homotopically Morita--Takeuchi equivalent} if they are Morita equivalent in the bicategory $\dcomod$. A major consequence of Theorem \ref{theorem: coHH is shadow (derived)} toegther with Proposition \ref{prop: morita invariance for shadows} is the following.

\begin{thm}\label{thm: coHH is morita invariant}
Let $C$ and $D$ be homotopically  Morita--Takeuchi equivalent simply connected differential graded $\mathbbm{k}$-coalgebras.
Then we obtain an isomorphism: \[\coHH(C)\cong \coHH(D).\] 
\end{thm}

\begin{proof}
    This is a combination of Theorem \ref{theorem: coHH is shadow (derived)} together with Proposition \ref{prop: morita invariance for shadows}.
\end{proof}

This extends the results of \cite{HScothh}, where we suspect that the simply-connected condition was forgotten to be mentioned.

\begin{prop}
    Suppose $(M, M^\dual)$ form a dual pair in $\dcomod$. Then we obtain a Quillen adjunction
    \[
\begin{tikzcd}
    M^\dual\square_C-\colon  \bicomod{C}{}(\Ck) \ar[shift left =2]{r} & \bicomod{D}{}(\Ck)\colon M\square_D -. \ar[shift left = 2]{l}[swap]{\perp}
\end{tikzcd}
    \]
    Moreover, it is a Quillen equivalence if and only if $(M, M^\dual)$ is a homotopical Morita-Takeuchi equivalence.
\end{prop}

\begin{proof}
    Since $M$ and $M^\dual$ are dual to each other, the maps $\eta\colon  C\rightarrow M\square_D M^\dual$ and $\varepsilon\colon M^\dual\square_C M\rightarrow D$ induce the adjunction. For instance, given a left $C$-comodule $P$, a left $D$-comodule $Q$, and a $D$-colinear map $\alpha\colon M^\dual\square_C P\rightarrow Q$, we obtain a $C$-colinear map $P\rightarrow M\square_D Q$ via the composite
    \[
\begin{tikzcd}
P\cong C\square_C P \ar{r}{\eta\square \id} & (M\square_D M^\dual)\square_C P \cong M\square_D(M^\dual\square_C P) \ar{r}{\id\square \alpha} & M\square_D Q.
\end{tikzcd}
    \]
 The axioms on $\eta$ and $\varepsilon$ guarantee that the procedure gives a correspondence:
 \[
\bicomod{D}{}(\Ck) \Big( M^\dual\square_C P, Q\Big) \cong \bicomod{C}{}(\Ck) \Big( P, M\square_D Q\Big).
 \]

Since the cotensor product is left exact, then the functor $M^\dual\square_D-$ preserves monomorphisms, i.e., cofibrations.
Since $M^\dual$ is a fibrant right $C$-comodule, then $M^\dual\square_C-$ preserves quasi-isomorphisms, i.e., weak equivalences (by Proposition \ref{Prop: cotensor preserve weak equivalence if fibrant}).
Therefore we obtain that it is a Quillen adjunction.

Suppose $(M, M^\dual)$ is a homotopical Morita--Takeuchi equivalence and $P\rightarrow P'$ is a $C$-colinear map such that the induced map
\[
M^\dual\square_C P \stackrel{\simeq}\longrightarrow M^\dual\square_C P'
\]
is a quasi-isomorphism.
Then if  we apply $M\square_D-$, since $M$ is a fibrant $D$-comodule, we obtain the quasi-isomorphism by Proposition \ref{Prop: cotensor preserve weak equivalence if fibrant}:
\[
\begin{tikzcd}[column sep= small]
 \underbrace{(M\square_D M^\dual)\square_C P}_{\simeq P} \cong M\square_D(M^\dual\square_C P)\ar{r}{\simeq} & M\square_D(M^\dual\square_C P') \cong \underbrace{(M\square_D M^\dual)\square_C P'}_{\simeq P'}.
\end{tikzcd}
\]
Therefore $P\rightarrow P'$ is a quasi-isomorphism, and thus $M^\dual\square_C-$ reflects  quasi-isomor\-phism.
Moreover, for any fibrant left $D$-comodule $Q$, we have a quasi-isomorphism
\[
Q \cong  D\square_D Q \simeq M^\dual\square_C M \square_D Q \longrightarrow Q.
\]
Thus by \cite[1.3.16]{hovey}, we obtain that the adjunction is a Quillen equivalence. 

Conversely, if we supposed the adjunction to be a Quillen equivalence, then if we apply the adjoints on $C$ and $D$, we recover that the coevaluation and evaluation $C\rightarrow M\square_D M^\dual$ and $M^\dual\square_C M \rightarrow D$ are quasi-isomorphisms by again applying \cite[1.3.16]{hovey}.
\end{proof}

\begin{ex}
Given a map of simply connected coalgebras $f:C\rightarrow D$, we recover the expected result that if $f$ is a quasi-isomorphism, then $C$ and $D$ are homotopically Morita-Takeuchi equivalent. 
In more details, recall that $(C, C)$ form a dual pair of bicomodules over $C$, and $h_C(C,C)\cong C$. 
Moreover, $C$ can be regarded as a left $D$-comodule via $f$:
\[
\begin{tikzcd}
    C\ar{r} & C\otimes C \ar{r}{f\otimes \id} & D\otimes C.
\end{tikzcd}
\]
In fact, any left $C$-comodule $P$ can be regarded as a left $D$-comodule this way, we shall denote it by $f^*(C)$. Notice that $f^*(C)\cong C\square_C P$ as a left $D$-comodule. We obtain a Quillen adjunction:
 \[
\begin{tikzcd}
    f^*\colon\bicomod{C}{}(\Ck) \ar[shift left =2]{r} & \bicomod{D}{}(\Ck)\colon C\square_D - \ar[shift left = 2]{l}[swap]{\perp},
\end{tikzcd}
    \]
    which is a Quillen equivalence if and only if $f$ is a quasi-isomorphism. This recovers the usual change of coalgebras, see  \cite[4.10]{connectivecomod}.
\end{ex}

\appendix 

\section{Homotopy theory of connective bicomodules}\label{chap: appendix}

The goal of the appendix is to show that the homotopy theory of bicomodules is endowed with a derived cotensor product which provides a homotopy coherent monoidal structure.

Let $\C^\otimes$ be a symmetric monoidal $\infty$-category, as in \cite[2.0.0.7]{HA}. Subsequently, we only refer to its underlying $\infty$-category, $\C$, as in \cite[2.1.2.20]{HA}.
By an \emph{$\Ainf$-algebra in $\C$}, we mean an associative algebra in the sense of \cite[4.1.1.6]{HA}. Given $\Ainf$-algebras $A$ and $B$ in $\C$, denote the $\infty$-category of $(A,B)$-bimodule objects in $\C$ by $\bimod{A}{B}(\C)$, as in \cite[Section 4.3]{HA}.
From \cite[2.1]{coalgenr}, define an \emph{$\Ainf$-coalgebra} in $\C$ to be an $\Ainf$-algebra in the opposite category $\C$. Similarly as in \cite[4.1.1.7]{HA}, given an $\Ainf$-coalgebra $C$ in $\C$, we can define $C^\op$, the \emph{opposite $\Ainf$-coalgebra} of $C$.

\begin{defn}
Let $\C$ be a symmetric monoidal $\infty$-category. Let $C$ and $D$ be $\Ainf$-coalgebras in $\C$.
A \emph{$(C, D)$-bicomodule} object in $\C$ is a {$(C, D)$-bimodule} object in $\C^\op$. The $\infty$-category of $(C, D)$-bicomodules in $\C$ is defined as
\[
\bicomod{C}{D}(\C):=(\bimod{C}{D}(\C^\op))^\op.
\]
We define the $\infty$-categories of left $C$-comodules $\bicomod{C}{}(\C)$ and right $D$-comodules $\bicomod{}{D}(\C)$ similarly.
\end{defn}

\begin{rem}\label{rem: bicomodules are right comodules (infinity)}
Just as in Remark \ref{rem: bicomodules are right comodules (ordinary)}, given $\Ainf$-coalgebras $C$ and $D$ in a symmetric monoidal $\infty$-category $\C$, we obtain an equivalence of $\infty$-categories
\[
\bicomod{C}{D}(\C)\simeq \bicomod{}{C^\op\otimes D}(\C).
\]
This follows from \cite[4.6.3.11]{HA} applied to the opposite category. Notice that the statement requires the monoidal product of $\C$ to commute with totalizations. However, as noted above \cite[4.6.3.3]{HA}, this requirement is not essential and remains true for any symmetric monoidal $\infty$-category $\C$.
\end{rem}

\begin{rem}
Let $\mathsf{C}$ be a symmetric monoidal category. 
Let $\mathsf{C}^\otimes$ be its operator category as in \cite[2.0.0.1]{HA}. Then its nerve $\N(\mathsf{C}^\otimes)$ is a symmetric monoidal $\infty$-category whose underlying $\infty$-category is $\N(\mathsf{C})$, see \cite[2.1.2.21]{HA}. Let $C$ be a coalgebra in $\mathsf{C}$. It can be regarded as an $\mathbb{A}_\infty$-coalgebra in $\N(\mathsf{C})$, see \cite[2.3]{coalgenr}. If $D$ is also a coalgebra in $\Cc$, then we obtain an equivalence of $\infty$-categories
\[
\N\left(\bicomod{C}{D}(\Cc)\right) \simeq \bicomod{C}{D}\left(\N(\Cc)\right).
\]
\end{rem}

Let $\Mc$ be a symmetric monoidal model category as in \cite[4.2.6]{hovey}, with a class of weak equivalence denoted $\Wc$. We assume every object is cofibrant. Its Dwyer-Kan localization is the $\infty$-category denoted $\N(\Mc)[\Wc^{-1}]$ following \cite[1.3.4.15]{HA} whose homotopy category is the homotopy category of $\Mc$. It is endowed with a symmetric monoidal structure via the derived tensor product, see \cite[4.1.7.6]{HA}.
Suppose now that the model structure on $\Mc$ is combinatorial. Given algebras $A$ and $B$ in $\Mc$, the category of bimodules $\bimod{A}{B}(\Mc)$ is endowed with a model structure whose class of weak equivalences, denoted $\Wc'$, are the morphisms of bimodules which are weak equivalences regarded as in $\Mc$. By \cite[4.3.3.17]{HA}, we obtain an equivalence of $\infty$-categories
\[
\N\Big(\bimod{A}{B}(\Mc)\Big) [\Wc'^{-1}] \simeq \bimod{A}{B}\left(\N(\Mc)[\Wc^{-1}]\right).
\]
The same statement for bicomodules is challenging if we insist on considering a combinatorial monoidal model category (and not ``co-combinatorial"). Following \cite{connectivecomod}, we show in Theorem \ref{thm: bicomodules connective} when the equivalence above does hold for bicomodules.
We denote the Dwyer-Kan localization of $\Ck$ by $\DDk$; it is equivalent to the symmetric monoidal $\infty$-category of connective $H\mathbbm{k}$-modules in spectra.

\begin{thm}\label{thm: bicomodules connective}
Let $C$ and $D$ be simply connected dg-coalgebras over $\mathbbm{k}$.  Then the natural functor
\[
\N\left( \bicomod{C}{D}(\Ck)\right) \left[ \mathsf{W}^{-1} \right] \longrightarrow \bicomod{C}{D}(\DDk)
\]
is an equivalence of $\infty$-categories.
\end{thm}

\begin{proof}
The result was proved in \cite[3.1]{connectivecomod} for right comodules. We can deduce the result for bicomodules using Remarks \ref{rem: bicomodules are right comodules (ordinary)} and \ref{rem: bicomodules are right comodules (infinity)}.
In more details, if $C$ and $D$ are simply connected, then so is $C\otimes D$:
\[
(C\otimes D)_0=C_0\otimes D_0\cong\mathbbm{k}\otimes \mathbbm{k}\cong \mathbbm{k}, \hspace{15pt} (C\otimes D)_1=(C_1\otimes D_0)\oplus (C_0\otimes D_1)=0.
\]
Therefore we can apply \cite[3.1]{connectivecomod} to the model category of right $(C^\op\otimes D)$-comodules. In particular, we obtain the following diagram of $\infty$-categories:
\[
\begin{tikzcd}
\N\left( \bicomod{C}{D}(\Ck)\right) \left[ \mathsf{W}^{-1} \right] \ar{r}{\simeq}\ar{d} & \N\left( \bicomod{}{C^\op\otimes D}(\Ck)\right) \left[ \mathsf{W'}^{-1} \right]\ar{d}{\simeq}\\
\bicomod{C}{D}(\DDk) \ar{r}{\simeq} & \bicomod{}{C^\op\otimes D}(\DDk).
\end{tikzcd}
\]
Here, we let $\mathsf{W}'$ denote the class of right $(C^\op\otimes D)$-comodule morphisms that are quasi-isomorphisms in $\Ck$. By \cite[3.1]{connectivecomod}, the right vertical map on the diagram above is an equivalence of $\infty$-categories. By Remark \ref{rem: bicomodules are right comodules (ordinary)}, the top horizontal map is an equivalence. By Remark \ref{rem: bicomodules are right comodules (infinity)}, the bottom horizontal map is an equivalence. Thus the left vertical map is an equivalence of $\infty$-categories.
\end{proof}

We now describe fibrant objects in $\bicomod{C}{D}(\Ck)$ using Postnikov towers. We say a tower $\{M(n)\}$ in $\bicomod{C}{D}(\Ck)$ \emph{stabilizes in each degree} if for all $n\geq 0$, and all $0\leq i \leq n$, the maps in the tower induce isomorphisms of $\mathbbm{k}$-modules
\[
M(n+1)_i\cong M(n+2)_i\cong M(n+3)_i\cong \cdots.
\]
Although in general non-finite limits in $\bicomod{C}{D}(\Ck)$ do not correspond to the underlying limit in $\Ck$, limits of towers that stabilizes in each degree in $\bicomod{C}{D}(\Ck)$ do correspond to their underlying limits, see \cite[4.15]{pertower}.
Furthermore, usually the functor $M\square_C-\colon\bicomod{C}{}(\Ck)\rightarrow \Ck$ does not preserve non-finite limits, but it does for towers that stabilize in each degree, see \cite[4.21]{pertower}. 

\begin{prop}[{\cite[4.17]{pertower}}]\label{prop: postnikov tower}
Let $C$ and $D$ be simply connected dg-coalgebras over $\mathbbm{k}$. Let $M$ be a $(C,D)$-bicomodule in $\Ck$. 
There exists a tower $\{ M(n)\}$ in $\bicomod{C}{D}(\Ck)$ that stabilizes in each degree defined as follows:
\begin{itemize}
    \item $M(0)=0$;
    \item $M(1)=C\otimes M\otimes D$;
    \item if the $(C,D)$-bicomodule $M(n)$ is defined together with a cofibration $M\hookrightarrow M(n)$ that induces an isomorphism $H_i(M)\cong H_i(M(n))$ for $0\leq i\leq n$, then $M(n+1)$ and the cofibration $M\hookrightarrow M(n+1)$ are defined by the pullback (both in $\bicomod{C}{D}(\Ck)$ and in $\Ck$):
    \[
    \begin{tikzcd}
    M \ar[hook, dashed]{dr} \ar[bend left]{drr}{0} \ar[bend right, hook]{ddr} & &\\
     &  M(n+1)\ar{d} \ar{r} \pull & C\otimes P\otimes D \ar{d} \\
     &  M(n) \ar{r} & C\otimes Q\otimes D,
    \end{tikzcd}
    \]
    where the right vertical map is the $(C,D)$-cofree map induced by an epimorphism $P\rightarrow Q$ in $\Ck$.
    
\end{itemize}
Moreover $\holim_n M(n) \simeq \lim_n M(n)\cong M$, and the (homotopy) limits are determined in $\Ck$.
\end{prop}

We now show that fibrant objects in the category of bicomodules behave well with respect to the cotensor product.

Consider the general case of a symmetric monoidal category, $(\Cc, \otimes, \mathbb{I})$, that is abelian, and consider flat coalgebras $C$ and $D$ in $\Cc$. Then the category of bicomodules $\bicomod{C}{D}(\Cc)$ remains abelian, and the forgetful functor $U\colon \bicomod{C}{D}(\Cc)\rightarrow \Cc$ preserves and reflects exact sequences, kernels, cokernels, monomorphisms, and epimorphisms.

\begin{lemma}[{\cite[4.5]{connectivecomod}}]\label{lem: bousfield}
Let $C$ and $D$ be dg-coalgebras over $\mathbbm{k}$. 
Let $f\colon M\rightarrow N$ be an epimorphism in $\bicomod{C}{D}(\Ck)$, and let $F$ be its kernel. Then $f$ is a fibration in $\bicomod{C}{D}(\Ck)$ if and only if $F$ is fibrant in $\bicomod{C}{D}(\Ck)$.
\end{lemma}

\begin{lemma}\label{lem: fib bicomodule==> fib left and right}
Let $C$ and $D$ be simply connected dg-coalgebras over $\mathbbm{k}$. If $M$ is a fibrant $(C,D)$-bicomodule, then $M$ is also a fibrant left $C$-comodule and a fibrant right $D$-comodule.
\end{lemma}

\begin{proof}
Since $M$ is a fibrant $(C,D)$-bicomodule, it is a retract of the limit $\widetilde{X}$ of its Postnikov tower $\{M(n)\}$. Since for all chain complexes $V$, the bicomodule $C\otimes V \otimes D$ is fibrant both as a left $C$-comodule and right $D$-comodule, then by Lemma \ref{lem: bousfield}, we can conclude the desired result.
\end{proof}

\begin{lemma}\label{lem: tensor preserves fibrant (cofree case)}
Let $V$ be in $\Ck$, 
let $C$ and $D$ be simply connected dg-coalgebras over $\mathbbm{k}$, and
let $N$ be a fibrant right $D$-comodule. 
Then $(C\otimes V)\otimes N$ is a fibrant $(C,D)$-bicomodule.
\end{lemma}

\begin{proof}
We consider $\{N(m)\}$, the Postnikov tower in $\bicomod{}{D}(\Ck)$ of $N$. As $N$ is a retract of the limit $\widetilde{N}=\mathsf{lim}_m N(m)$, then it is enough to show that $C\otimes V \otimes \widetilde{N}$ is a fibrant $(C,D)$-bicomodule. Therefore we need to prove that $\{(C\otimes V)\otimes N(m)\}$ is a fibrant tower of $(C,D)$-bicomodules.
For $m=0$, we get $(C\otimes V)\otimes N(0)=0$, which is fibrant.
For $m=1$, we get $(C\otimes V)\otimes N(1)=(C\otimes V)\otimes N \otimes D$, which is a cofree $(C,D)$-bicomodule and is thus a fibrant $(C,D)$-bicomodule.
Suppose $m\geq 1$, then we obtain the pullback in $\bicomod{C}{D}(\Ck)$:
\[
\begin{tikzcd}
 (C\otimes V) \otimes N(m+1) \pull\ar{d} \ar{r} & (C\otimes V) \otimes P \otimes D \ar{d} \\
 (C\otimes V) \otimes N(m) \ar{r} & (C\otimes V) \otimes Q \otimes D.
\end{tikzcd}
\]
Since the right vertical map is a fibration in $\bicomod{C}{D}(\Ck)$, then so is the left vertical map. Thus $\{(C\otimes V)\otimes Y(m)\}$ is a fibrant tower of $(C,D)$-bicomodules.
\end{proof}

\begin{lemma}\label{lem: tensor preserves fibrant}
Let $C$ and $D$ be simply connected dg-coalgebras over $\mathbbm{k}$. If $M$ is a fibrant left $C$-comodule and $N$ is a fibrant right $D$-comodule, then $M\otimes N$ is a fibrant $(C,D)$-bicomodule.
\end{lemma}

\begin{proof}
Let $\{M(n)\}$ be the Postnikov tower in $\bicomod{C}{}(\Ck)$ of $X$. 
As $M$ is fibrant, it is a retract of $\widetilde{X}=\mathsf{lim}_n M(n)$. Therefore it is sufficient to show that $\widetilde{M}\otimes N$ is a fibrant $(C,D)$-bicomodule.
Since the tower stabilizes in each degree, we have
\[
(\mathsf{lim}_n M(n))\otimes N \cong \mathsf{lim}_n (M(n)\otimes N),
\]
and thus it is enough to show that $\{M(n)\otimes N\}$ is a fibrant tower of $(C,D)$-bicomodules.
For $n=0$, we have $M(0)\otimes N=0$, which is fibrant. 
For $n=1$, we have $M(1)\otimes N=(C\otimes M)\otimes N$, which is a fibrant $(C,D)$-bicomodule by Lemma \ref{lem: tensor preserves fibrant (cofree case)}. 
For $n\geq 1$, we obtain the pullback in $\bicomod{C}{D}(\Ck)$:
\[
\begin{tikzcd}
 M(n+1)\otimes N\ar{d}\ar{r} \pull & (C\otimes P)\otimes N\ar{d}\\
 M(n) \otimes N\ar{r} & (C\otimes Q)\otimes N.
\end{tikzcd}
\]
The right vertical map is a fibration in $\bicomod{C}{D}(\Ck)$ by Lemma \ref{lem: bousfield} since its kernel $(C\otimes K)\otimes N$ is a fibrant $(C,D)$-bicomodule by Lemma \ref{lem: tensor preserves fibrant (cofree case)}, where $K$ is the kernel of $P\rightarrow Q$. Thus the left vertical map is a fibration.
\end{proof}

\begin{prop}\label{prop: cotensor of fibrants}
Let $C$, $D$ and $E$ be simply connected dg-coalgebras over $\mathbbm{k}$. If $M$ is a fibrant $(D, C)$-bicomodule and $N$ is a fibrant $(C, E)$-bicomodule, then $M\square_C N$ is a fibrant $(D, E)$-bicomodule.
\end{prop}

\begin{proof}
Let $\{M(n)\}$ be the Postnikov tower of $M$ in $\bicomod{D}{C}(\Ck)$. Then $M$ is a retract of $\widetilde{M}\simeq\mathsf{lim}_n M(n)$. As $\widetilde{M}\square_C N=(\mathsf{lim}_n M(n)\square_C N)\cong \lim_n (M(n)\square_C N)$, it is enough to show that $\{M(n)\square_C N\}$ is a fibrant tower of $(D,E)$-bicomodules. For $n=0$, then $M(0)\square_C N=0$ and is thus fibrant. For $n=1$, we get $M(1)\square_C N\cong (D\otimes M\otimes C)\square_C N \cong (D\otimes M)\otimes N$. By Lemmas \ref{lem: fib bicomodule==> fib left and right} and \ref{lem: tensor preserves fibrant}, it is a fibrant $(D,E)$-bicomodule. As the functor $-\square_C Y$ preserves pullbacks, we obtain the pullback in $\bicomod{D}{E}(\Ck)$:
\[
\begin{tikzcd}
M(n+1)\square_C N \ar{d} \ar{r} \pull & (D\otimes P) \otimes N\ar{d}\\
M(n)\square_C N \ar{r} & (D\otimes Q) \otimes N.
\end{tikzcd}
\]
The right vertical map is a fibration by Lemma \ref{lem: bousfield} as its kernel is $(D\otimes K)\otimes N$, which is a fibrant $(D,E)$-bicomodule by Lemmas \ref{lem: fib bicomodule==> fib left and right} and \ref{lem: tensor preserves fibrant}, where $K$ is the kernel of $P\rightarrow Q$. Thus the left vertical map is a fibration.
\end{proof}

One particularly nice characterizing algebraic property of fibrant comodules is that they are \emph{coflat}, i.e., they interplay well with the cotensor product and exact sequences. The cotensor product is left-exact, as it preserves finite products and equalizers. We are interested in knowing the cases in which it preserves exactness, without any cocommutativity requirement.

\begin{defn}
Let $C$ and $D$ be flat coalgebras in a symmetric monoidal abelian category, $(\Cc, \otimes, \mathbb{I})$.
Let $M$ be a $(C, D)$-bicomodule. We say \emph{$M$ is left coflat over $C$} if given any short exact sequence in $\bicomod{}{C}$
\[
\begin{tikzcd}
0\ar{r} & N \ar{r} & N'\ar{r} & N''\ar{r} & 0,
\end{tikzcd}
\]
we obtain a short exact sequence in $\Cc$
\[
\begin{tikzcd}
0\ar{r} & N\square_C M\ar{r} & N'\square_C M\ar{r} & N''\square_C M \ar{r} & 0.
\end{tikzcd}
\]
Similarly, we say \emph{$M$ is right coflat over $D$} if given any short exact sequence in $\bicomod{D}{}$
\[
\begin{tikzcd}
0\ar{r} & N \ar{r} & N'\ar{r} & N''\ar{r} & 0,
\end{tikzcd}
\]
we obtain a short exact sequence in $\Cc$
\[
\begin{tikzcd}
0\ar{r} & M\square_D N \ar{r} & M\square_D N'\ar{r} & M\square_D N''\ar{r} & 0.
\end{tikzcd}
\]
We say the bicomodule $M$ is \emph{two-sided coflat over $(C,D)$} if it is left coflat over $C$ and right coflat over $D$.
Using Definition \ref{def: cotor}, if $M$ is right coflat over $C$, or if $N$ is left coflat over $C$, then $\cotor^i_C(M,N)=0$ for all $i\geq 1$.
\end{defn}

\begin{prop}\label{prop: fibrant==>coflat}
Let $C$ and $D$ be simply connected dg-coalgebras over $\mathbbm{k}$. If a $(C,D)$-bicomodule is fibrant, then it is a two-sided coflat bicomodule over $(C,D)$.
\end{prop}

\begin{proof}
Notice that any cofree $(C,D)$-bicomodule $C\otimes V \otimes D$ is two-sided coflat. Let us show that coflatness is preserved under extensions. Consider a short exact sequence in $\bicomod{C}{D}(\Ck)$,
\[
\begin{tikzcd}
 0 \ar{r} & M\ar{r} & N \ar{r} & P \ar{r} & 0.
\end{tikzcd}
\]
Suppose $M$ and $P$ are two-sided coflat.
Let $Q$ be a left $D$-comodule. We then obtain an exact sequence
\[
\begin{tikzcd}
\cotor^D_1(M, Q) \ar{r} & \cotor^D_1(N, Q) \ar{r} & \cotor^D_1(P, Q).
\end{tikzcd}
\]
By exactness, we get that $\cotor^D_1(N,Q)=0$ for any left $D$-comodule $Q$. Thus $N$ is coflat as a right $D$-comodule. We argue similarly to show that $N$ is coflat as a left $C$-comodule.

Now let $F$ be a fibrant $(C,D)$-bicomodule. Let $\{F(n)\}$ be its Postnikov tower in $\bicomod{C}{D}(\Ck)$ and write $\widetilde{F}=\mathsf{lim}_n F(n)$. Since $F$ is a retract of $\widetilde{F}$, then for any left $D$-comodule $Q$, we get $\cotor^D_1(F, Q)$ is a retract of $\cotor^D_1(\widetilde{F}, Q)$. Thus if $\widetilde{F}$ is coflat as a right $D$-comodule, so is $F$. We prove $\widetilde{F}$ is right coflat by induction. For $n=0$, we see that $F(0)=0$ is coflat. For $n=1$, we get that $F(1)=C\otimes F\otimes D$, a cofree bicomodule, and is thus right coflat. 
Suppose we have shown that $F(n)$ is right coflat over $D$. Then from the short exact sequence in $\bicomod{C}{D}(\Ck)$:
\[
\begin{tikzcd}
 0 \ar{r} & C\otimes K \otimes D \ar{r} & F(n+1) \ar{r} & F(n) \ar{r} & 0,
\end{tikzcd}
\]
we get that $F(n+1)$ is also right coflat over $D$. Given any short exact sequence in $\bicomod{D}{}(\Ck)$:

\[
\begin{tikzcd}
 0 \ar{r} & Q \ar{r} & Q' \ar{r} & Q'' \ar{r} & 0,
\end{tikzcd}
\]
we obtain a short exact sequence of towers in $\Ck$:
\[
\begin{tikzcd}
 0 \ar{r} & \{F(n)\square_D W\} \ar{r} & \{F(n)\square_D W'\} \ar{r} & \{ F(n)\square_D W''\} \ar{r} & 0.
\end{tikzcd}
\]
Each of these towers satisfies the Mittag-Leffler condition, and since limit of towers that stabilizes in each degree commute with cotensor product, we obtain a short exact sequence
\[
\begin{tikzcd}
 0 \ar{r} &  \widetilde{F}\square_D Q \ar{r} & \widetilde{F}\square_D Q' \ar{r} & \widetilde{F}\square_D Q'' \ar{r} & 0.
\end{tikzcd}
\]
Therefore $\widetilde{F}$ is right coflat over $D$. We can show $\widetilde{F}$ is left coflat over $C$ in a similar fashion.
\end{proof}

\begin{rem}
In fact, the result of \cite[4.7]{connectivecomod} remains true in the non-commu\-tative case. Thus one can show that a $(C,D)$-bicomodule is fibrant if and only if it is coflat as a right $(C^\op\otimes D)$-comodule. We shall not need this result here so we do not provide details, however we do mention below the relationship between two-sided coflat $(C,D)$-comodules and right coflat $(C^\op\otimes D)$-comodules.
In particular, by Lemma \ref{lem: fibrant==>two-sided coflat}, a fibrant $(C,D)$-bicomodule is always two-sided coflat.
\end{rem}

\begin{lemma}
Let $(\Cc, \otimes, \bI)$ be a symmetric monoidal category.
Let $(C, \Delta_C, \varepsilon_C)$ and $(D, \Delta_D, \varepsilon_D)$ be flat coalgebras in $\Cc$. Let $M$ be a right $(C\otimes D)$-comodule. Let $N$ be a left $D$-comodule. Then $C\otimes N$ is a left $(C\otimes D)$-comodule, and we obtain an isomorphism in $\Cc$:
\[
M\square_{C\otimes D}(C\otimes N)\cong M\square_D N.
\]
\end{lemma}

\begin{proof}
Let $\lambda:N\rightarrow D\otimes N$ be the left $D$-coaction on $N$. Then we obtain a left $(C\otimes D)$-coaction on $N$ via
\[
\begin{tikzcd}
C\otimes N\ar{r}{\id_C\otimes \lambda} & C\otimes D\otimes N \ar{r}{\Delta_C\otimes \id_{D\otimes Y}} & [20pt] C\otimes C \otimes D \otimes N\cong C\otimes D\otimes C \otimes N.
\end{tikzcd}
\]
The above coaction, denoted $\widetilde{\lambda}\colon C\otimes N\rightarrow (C\otimes D)\otimes C \otimes N$, is counital and coassociative because the left $D$-coaction on $N$ and the coalgebra structure on $C$ are both coassociative and counital.

To prove the desired isomorphism, we verify that $M\square_D N$ satisfies the universal property of the equalizer. From the right $(C\otimes D)$-coaction $\rho\colon M\rightarrow M\otimes C\otimes D$, we obtain the underlying right $C$-coaction $\rho_C$ on $M$ as the composite
\[
\begin{tikzcd}
M\ar{r}{\rho} & M\otimes C\otimes D \ar{r}{\id_{X\otimes C}\otimes \varepsilon_D} &[20pt] M\otimes C.
\end{tikzcd}
\]
Now we obtain a morphism $\alpha\colon M\square_D N\rightarrow M\square_{C\otimes D}(C\otimes N)$ by functoriality of the equalizers:
\[
\begin{tikzcd}
M\square_D N \ar{r} \ar[dashed]{d}{\exists !}[swap]{\alpha} & M\otimes N \ar{d}{\rho_C\otimes \id_N}\ar[shift left]{r}{\rho_D\otimes \id_N} \ar[shift right]{r}[swap]{\id_M\otimes \lambda} & [4em] M\otimes D \otimes N \ar{d} \\
M\square_{C\otimes D}(C\otimes N) \ar{r} &  M\otimes (C\otimes N) \ar[shift left]{r}{\rho\otimes \id_{C\otimes N}} \ar[shift right]{r}[swap]{\id_M\otimes \widetilde{\lambda}} & M\otimes (C\otimes D) \otimes (C \otimes N).
\end{tikzcd}
\]
The right unlabeled vertical arrow is induced by applying the functor $-\otimes N$ on the map
\[
\begin{tikzcd}
M\otimes D\ar{r}{\rho_C\otimes \id_D} & (M\otimes C)\otimes D \ar{r}{\id_M\otimes \Delta_C \otimes \id_D} & [2.5em] \underbrace{(M\otimes (C \otimes C)) \otimes D}_{\cong M\otimes (C \otimes D) \otimes C}. 
\end{tikzcd}
\]
Similarly, by applying the counit map $\varepsilon_C\colon C\rightarrow \mathbb{I}$ vertically on each $C$, we obtain the dashed map  on the equalizers below:
\[
\begin{tikzcd}
M\square_{C\otimes D}(C\otimes N) \ar{r}\ar[dashed]{d}{\exists !}[swap]{\beta} &  M\otimes C\otimes N \ar[shift left]{r}{\rho\otimes \id_{C\otimes N}} \ar[shift right]{r}[swap]{\id_M\otimes \widetilde{\lambda}} \ar{d}& [4em] M\otimes (C\otimes D) \otimes C \otimes N\ar{d}\\
M\square_D N \ar{r}  & M\otimes N \ar[shift left]{r}{\rho_D\otimes \id_N} \ar[shift right]{r}[swap]{\id_M\otimes \lambda} &  M\otimes D \otimes N.
\end{tikzcd}
\]
The induced morphism $\beta$ is the inverse of the morphism $\alpha$ defined above. Therefore we obtain the desired isomorphism in $\Cc$.
\end{proof}

\begin{lemma}\label{lem: fibrant==>two-sided coflat}
Let $(\Cc, \otimes, \bI)$ be a symmetric monoidal abelian category.
Let $C$ and $D$ be flat coalgebras in $\Cc$.
Let $M$ be a $(C, D)$-bicomodule in $\Cc$. If $M$ is right coflat as a right $(C^\op\otimes D)$-comodule, then it is two-sided coflat as a $(C, D)$-bicomodule.
\end{lemma}

\begin{proof}
Suppose $M$ is right coflat a $(C^\op\otimes D)$-comodule.
Let us first show that $M$ is right coflat over $D$.  Consider a short exact sequence in $\bicomod{D}{}$:
\[
\begin{tikzcd}
0\ar{r} & N \ar{r} & N'\ar{r} & N''\ar{r} & 0.
\end{tikzcd}
\]
By the previous lemma, we obtain that $C^\op\otimes N$, $C^\op\otimes N'$, and $C^\op\otimes N''$ are left $(C^\op\otimes D)$-comodules. Since $C^\op\otimes-$ preserves exactness because it is flat, we obtain a short exact sequence in $\bicomod{(C^\op\otimes D)}{}$:
\[
\begin{tikzcd}
0\ar{r} & C^\op\otimes N\ar{r} & C^\op\otimes N'\ar{r} & C^\op\otimes N''\ar{r} & 0.
\end{tikzcd}
\]
Since $M$ is a right coflat over $(C^\op\otimes D)$, we obtain that the sequence is exact in $\Cc$:
\[
\begin{tikzpicture}[baseline= (a).base]
\node[scale=.855] (a) at (1,1){
\begin{tikzcd}
0\ar{r} & M\square_{C^\op\otimes D} C^\op\otimes N \ar{r} & M\square_{C^\op\otimes D} C^\op\otimes N'\ar{r} & M\square_{C^\op\otimes D} C^\op\otimes N''\ar{r} & 0.
\end{tikzcd}
};
\end{tikzpicture}
\]
By the previous lemma, this induces the short exact sequence in $\Cc$:
\[
\begin{tikzcd}
0\ar{r} & M\square_{ D}  N \ar{r} & M\square_{D} N'\ar{r} & M\square_{ D}  N''\ar{r} & 0.
\end{tikzcd}
\]
Thus $M$ is right coflat over $D$ as desired. We prove that $M$ is left coflat over $C$ in a similar fashion.
\end{proof}

\begin{ex}
The converse is not true: a two-sided coflat $(C, D)$-bicomodule need not to be a right coflat $(C^\op\otimes D)$-comodule, and thus need not be a fibrant $(C,D)$-bicomodule. This can already be seen for modules over an algebra (even commutative). Indeed, consider the polynomial ring, $\mathbbm{k}[x]$, as a $\mathbbm{k}$-algebra. It is a two-sided flat $\mathbbm{k}[x]$-module, but it is not flat as a right $(\mathbbm{k}[x]\otimes \mathbbm{k}[x])$-module. Indeed, note first that $\mathbbm{k}[x,y]\cong\mathbbm{k}[x]\otimes \mathbbm{k}[x]$. Consider the following short exact sequence of left $k[x,y]$-modules
\[
\begin{tikzcd}
 0\ar{r} & \mathbbm{k}[x,y] \ar{r}{\cdot y} & \mathbbm{k}[x, y]\ar{r} & \mathbbm{k}[x] \ar{r} & 0.
\end{tikzcd}
\]
If we apply $\mathbbm{k}[x]\otimes_{\mathbbm{k}[x,y]}-$, we obtain the exact sequence
\[
\begin{tikzcd}
 \mathbbm{k}[x]\ar{r}{0} & \mathbbm{k}[x]\ar{r} & \mathbbm{k}[x] \ar{r} & 0.
\end{tikzcd}
\]
The above sequence is not a short exact sequence however, which shows that $\mathbbm{k}[x]$ is not flat as a right $\mathbbm{k}[x,y]$-module.
\end{ex}

We are now equipped to prove a crucial result in this paper: the cotensor product with a fibrant comodule preserves quasi-isomorphisms.

\begin{prop}\label{Prop: cotensor preserve weak equivalence if fibrant}
Let $C$ and $D$ be simply connected dg-coalgebras over $\mathbbm{k}$. 
Let $M$ be a fibrant $(C,D)$-bicomodule. Then $M\square_D-\colon  \bicomod{D}{}(\Ck)\rightarrow \Ck$ and $-\square_C M:\bicomod{}{C}(\Ck)\rightarrow \Ck$ preserve weak equivalences.
\end{prop}

\begin{proof}
We use an Eilenberg--Moore spectral sequence argument.
Let $N$ be any left $D$-comodule in $\Ck$.
As in Remark \ref{rem: topological cobar vs algebraic cobar}, the conormalized cobar complex $\underline{\Omega}^\bullet(M,D,N)$ is a second quadrant double chain complex that is bounded. We grade the row cohomologically, but the columns homologically. Hence its two associated spectral sequences converge to the same page, see \cite[2.15]{mccleary}.

The first spectral sequence has its $E^1$-page induced by the cohomology of the rows and therefore
\[
E^1_{\bullet, q}=H^q(\underline{\Omega}^\bullet(M,D,N))\cong \cotor^q_D(M,N).
\]
Since $M$ is a fibrant $(C,D)$-bicomodule, then it is coflat as a right $D$-comodule. Thus $E^1_{\bullet, q}=0$ for all $q\geq 1$, and $E^1_{\bullet, 0}=M\square_D N$.
Hence the spectral sequence collapses onto its second page, $E^2_{\bullet, 0}=H_*(M\square_D N)$.

The second spectral sequence has its $E^1$-page induced by the homology of its columns and therefore
\[
E^1_{\bullet,q}=H_*(\underline{\Omega}^{q}({X},D,{Y}))=\underline{\Omega^q}(H_*(M), H_*(D), H_*(N)).
\]
Thus, as its $E^2$-page is given by the cohomology of the induced cochain complex, we obtain
\[
E^2_{\bullet, q}= \mathsf{CoTor}^q_{H_*(D)}(H_*(M), H_*(N)).
\]
It converges to the page with trivial columns except its $0^{th}$ column, which is given by the cohomology $H_*(M\square_D N)$.

Combining the arguments above, we obtain a converging Eilenberg--Moore spectral  sequence
\[
E^2_{\bullet,q}= \mathsf{CoTor}^q_{H_*(D)}(H_*(M), H_*(N)) \Rightarrow E^\infty_{\bullet, 0}=H_*(M \square_D N),
\] 
for any left $D$-comodule $N$. In particular, given a map of left $D$-comodules $N\rightarrow N'$ that is a quasi-isomorphism, we get
\[
\mathsf{CoTor}^q_{H_*(D)}(H_*(M), H_*(N))\cong \mathsf{CoTor}^q_{H_*(D)}(H_*(M), H_*(N')).
\]
Thus the map $N\rightarrow N'$ induces an isomorphism, $H_*(M \square_D N)\cong H_*(M\square_D N')$.
\end{proof}

\begin{cor}\label{cor: cotensor of fibrant is equal to cobar}
    Let $C$, $D$ and $E$ be simply connected coalgebras in $\Ck$.
    Let $M$ be a fibrant $(D,C)$-bicomodule and $N$ a fibrant $(C,E)$-bicomodule.
    Then we obtain a quasi-isomorphism of $(D,E)$-bicomodules
    \[
M\square_C N \simeq \Omega(M,C,N).
    \]
\end{cor}

\begin{cor}\label{cor: tensor commute with cobar}
Let $C$, $D$, and $E$ be simply connected coalgebras in $\Ck$.
    Let $M$ be a fibrant $(C,D)$-bicomodule, let $N$ be a left fibrant $D$-comodule and $P$ be a fibrant right $E$-comodule.
Then there is a natural quasi-isomorphism of $(C, E)$-bicomodules
\[
\Omega(M, D, N)\otimes P \simeq \Omega(M,C, N\otimes P).
\]
\end{cor}

\begin{proof}
    Since $M$ and $N$ are fibrant, we obtain
    \begin{align*}
       \Omega(M, D, N)\otimes P & \simeq (M\square_D N) \otimes P\\
       &  \cong M\square_D (N\otimes P)\\
       & \simeq \Omega(M,D, N\otimes P).
    \end{align*}
    The isomorphism follows from the fact that $P$ is flat in $\Ck$, and thus preserves equalizers.
    The last quasi-isomorphism follows from the fact that $N\otimes P$ remains a fibrant $(D, E)$-bicomodule by Lemma \ref{lem: tensor preserves fibrant}.
\end{proof}

\begin{rem}
    Corollary \ref{cor: tensor commute with cobar} is not true if the comodules are not fibrant, since in general the tensor product does not preserve infinite homotopy limits. 
\end{rem}

\begin{cor}\label{cor: bicosimplicial of cobar}
    Let $C$, $D$, $E$, and $F$ be simply connected coalgebras in $\Ck$.
    Let $M$ be a fibrant $(C,D)$-bicomodule, let $N$ be a fibrant $(D, E)$-bicomodule, and let $P$ be a fibrant $(E,F)$-bicomodule. 
    Then we obtain a quasi-isomorphism of $(C, F)$-bicomodules 
    \[
\Omega\big(\Omega(M, D, N), E, P\big) \simeq \holim_{\Delta\times \Delta} \Omega^\bullet\big(\Omega^\bullet(M,D,N),E,P\big).
    \]
\end{cor}

\begin{proof}
    By the previous Corollary, for any $i\geq 0$
    \begin{align*}
        \holim_\Delta\big( \Omega^\bullet(M,D, N)\otimes E^{\otimes i}\otimes P \big) & \simeq \holim_\Delta \big( \Omega^\bullet (M, D, N\otimes E^{\otimes i}\otimes P) \big)\\
        &  \simeq \Omega(M, D, N\otimes E^{\otimes i}\otimes P )\\
        & \simeq \Omega (M, D, N) \otimes E^{\otimes i}\otimes P\\
        & \simeq \holim_\Delta \big( \Omega^\bullet(M,D, N) \big) \otimes E^{\otimes i}\otimes P.
    \end{align*}
In other words, we have proved that we obtain a quasi-isomorphism between fibrant $(C, F)$-bicomodules (recall that we fixed a model for the homotopy limit):
\[
\holim_\Delta \big(\Omega^i(\Omega^\bullet(M,D, N), E, P)\big) \simeq \Omega^i\big(\holim_\Delta \Omega^\bullet(M,D, N), E, P \big).
\]
This above equivalence is compatible with cofaces and codegeneracies of the cobar construction, and therefore we obtain by \cite[18.5.2, 18.5.3]{hir}:
\begin{align*}
    \Omega(\Omega(M, D, N), E, P) & \simeq \holim_\Delta \Omega^\bullet \Big( \Omega(M,D, N) , E, P\Big)\\
    & \simeq \holim_\Delta \Omega^\bullet \Big( \holim_\Delta \Omega^\bullet \big( M, D, N\big), E, P \Big)\\
    &\simeq \holim_\Delta  \holim_\Delta \Omega^\bullet\big( \Omega^\bullet(M,D, N), E, P \big) \\
    & \simeq \holim_{\Delta \times \Delta} \Omega^\bullet(\Omega^\bullet(M, D, N), E, P). \qedhere 
\end{align*}
\end{proof}

\begin{cor}\label{cor: bicosimplicial of cohh}
   Let $C$, $D$, and $E$ be simply connected coalgebras in $\Ck$.
   Let $M$ be a fibrant $(C,D)$-bicomodule and $N$ a fibrant $(D,E)$-bicomodule. Then we obtain a quasi-isomorphism
   \[
\coHH(\Omega(M,D,N), C)\simeq \holim_{\Delta \times \Delta} \coHH^\bullet(\Omega^\bullet(M, D, N), E, P).
   \]
\end{cor}

\begin{proof}
    By Corollary \ref{cor: tensor commute with cobar}, for any $i\geq 0$
    \[
\holim_\Delta (\Omega^\bullet(M,D, N)\otimes C^{\otimes i}) \simeq \holim_\Delta (\Omega^\bullet(M,D, N))\otimes C^{\otimes i}.
    \]
The above equivalence is compatible with cofaces and codegeneracies of the cobar construction, and therefore we obtain \cite[18.5.2, 18.5.3]{hir}:
\begin{align*}
    \coHH(\Omega(M,D, N), C) & \simeq  \holim_\Delta \coHH^\bullet(\holim_\Delta \Omega^\bullet(M,D, N), C)\\
    & \simeq \holim_{\Delta\times \Delta} \coHH^\bullet(\Omega^\bullet(M,D, N), C). \qedhere
\end{align*}
\end{proof}

\begin{prop}\label{prop: coHH is given by cobar construction over Ce}
Let $C$ be a simply connected in $\Ck$.
Let $M$ be a fibrant $(C,C)$-bicomodule.
There is a quasi-isomorphism
$$\coTHH(M, C) \simeq \Omega(M, C\otimes C^\op, C).$$
\end{prop}

\begin{proof}
We denote the enveloping coalgebra by $C^e=C\otimes C^\op$.
Recall that we have the quasi-isomorphism $C\simeq \Omega(C, C, C)$ as $(C, C)$-bicomodules. Thus we obtain
\[ 
M \widehat{\square}_{C^e} C \simeq M \widehat{\square}_{C^e} \Omega(C, C, C).
\]
Notice that each object in the cosimplicial diagram $\ccobars{C}$ is a fibrant $(C,C)$-bicomodule by Lemma \ref{lem: tensor preserves fibrant (cofree case)}.
Thus $\ccobar{C}$ is a fibrant $(C,C)$-bicomodule by \cite[18.5.2]{hir}. 
Therefore
\[
M \widehat{\square}_{C^e} \Omega(C, C, C)\simeq M \square_{C^e} \Omega(C, C, C).
\]
Since the functor $M\square_{C^e}-$ preserves towers that stabilize in each degree, we get
\begin{align*}
    M {\square}_{C^e} \Omega (C,C,C) &\simeq \holim\Big({M} \square_{C^e} \Omega^\bullet(C,C,C)\Big).
\end{align*}
We have an isomorphism, $\coTHH^n(M, C) \cong M\square_{C^e}\Omega^n(C, C, C)$, for each $n\geq 0$, induced by
\begin{align*}
    M \otimes C^n \cong M {\square}_{C^e} (C^e \otimes C^n) &\cong M {\square}_{C^e} (C \otimes C^n \otimes C),
\end{align*}
given by the permutation of $C^{op} \cong C$ past $C^n$. 
These isomorphisms are compatible with the cosimplicial structures and thus provide an isomorphism, $\coTHH^\bullet(M, C)\cong M\square_{C^e}\Omega^\bullet(C, C, C)$.  
\end{proof}

\bibliography{coTHH}

\newcommand{\etalchar}[1]{$^{#1}$}
\begin{thebibliography}{EKMM97}

\bibitem[{Al-}02]{khaled}
Khaled {Al-Takhman}.
\newblock Equivalences of comodule categories for coalgebras over rings.
\newblock {\em J. Pure Appl. Algebra}, 173(3):245--271, 2002.

\bibitem[Bar24]{justin}
Justin Barhite.
\newblock Bicategorical traces and cotraces.
\newblock {\em Theory Appl. Categ.}, 41:Paper No. 22, 707--759, 2024.

\bibitem[BGH{\etalchar{+}}18]{bohmann2018computational}
Anna~Marie Bohmann, Teena Gerhardt, Amalie H{\o}genhaven, Brooke Shipley, and Stephanie Ziegenhagen.
\newblock Computational tools for topological co{H}ochschild homology.
\newblock {\em Topology Appl.}, 235:185--213, 2018.

\bibitem[BGS22]{loop}
Anna~Marie Bohmann, Teena Gerhardt, and Brooke Shipley.
\newblock Topological co{H}ochschild homology and the homology of free loop spaces.
\newblock {\em Math. Z.}, 301(1):411--454, 2022.

\bibitem[BHK{\etalchar{+}}15]{left1}
Marzieh Bayeh, Kathryn Hess, Varvara Karpova, Magdalena K\c{e}dziorek, Emily Riehl, and Brooke Shipley.
\newblock Left-induced model structures and diagram categories.
\newblock In {\em Women in topology: collaborations in homotopy theory}, volume 641 of {\em Contemp. Math.}, pages 49--81. Amer. Math. Soc., Providence, RI, 2015.

\bibitem[BHM93]{trace1}
M.~B\"{o}kstedt, W.~C. Hsiang, and I.~Madsen.
\newblock The cyclotomic trace and algebraic {$K$}-theory of spaces.
\newblock {\em Invent. Math.}, 111(3):465--539, 1993.

\bibitem[BM11]{blum-man}
Andrew~J. Blumberg and Michael~A. Mandell.
\newblock Derived {K}oszul duality and involutions in the algebraic {$K$}-theory of spaces.
\newblock {\em J. Topol.}, 4(2):327--342, 2011.

\bibitem[BM12]{blumberg2012localization}
Andrew~J. Blumberg and Michael~A. Mandell.
\newblock Localization theorems in topological {H}ochschild homology and topological cyclic homology.
\newblock {\em Geom. Topol.}, 16(2):1053--1120, 2012.

\bibitem[BP23]{dualitySW}
{\"O}zg{\"u}r~Haldun {Bay{\i}nd{\i}r} and Maximilien {P{\'e}roux}.
\newblock Spanier-{W}hitehead duality for topological co{H}ochschild homology.
\newblock {\em J. Lond. Math. Soc. (2)}, 107(5):1780--1822, 2023.

\bibitem[{Bun}12]{ulrich}
Ulrich {Bunke}.
\newblock {Differential cohomology}.
\newblock {\em arXiv e-prints}, page arXiv:1208.3961, August 2012.

\bibitem[BW03]{brzezinski2003corings}
Tomasz Brzezinski and Robert Wisbauer.
\newblock {\em Corings and comodules}, volume 309 of {\em London Mathematical Society Lecture Note Series}.
\newblock Cambridge University Press, Cambridge, 2003.

\bibitem[CP19]{campbell2019topological}
Jonathan~A. Campbell and Kate Ponto.
\newblock Topological {H}ochschild homology and higher characteristics.
\newblock {\em Algebr. Geom. Topol.}, 19(2):965--1017, 2019.

\bibitem[Doi81]{doi1981homological}
Yukio Doi.
\newblock Homological coalgebra.
\newblock {\em J. Math. Soc. Japan}, 33(1):31--50, 1981.

\bibitem[DP80]{doldpuppe}
Albrecht Dold and Dieter Puppe.
\newblock Duality, trace, and transfer.
\newblock In {\em Proceedings of the {I}nternational {C}onference on {G}eometric {T}opology ({W}arsaw, 1978)}, pages 81--102. PWN, Warsaw, 1980.

\bibitem[Dun97]{trace2}
Bj\o rn~Ian Dundas.
\newblock Relative {$K$}-theory and topological cyclic homology.
\newblock {\em Acta Math.}, 179(2):223--242, 1997.

\bibitem[EKMM97]{elmendorf1995rings}
A.~D. Elmendorf, I.~Kriz, M.~A. Mandell, and J.~P. May.
\newblock {\em Rings, modules, and algebras in stable homotopy theory}, volume~47 of {\em Mathematical Surveys and Monographs}.
\newblock American Mathematical Society, Providence, RI, 1997.
\newblock With an appendix by M. Cole.

\bibitem[FS98]{moritatakinvariance}
Marco~A. Farinati and Andrea Solotar.
\newblock Morita-{T}akeuchi equivalence, cohomology of coalgebras and {A}zumaya coalgebras.
\newblock In {\em Rings, {H}opf algebras, and {B}rauer groups ({A}ntwerp/{B}russels, 1996)}, volume 197 of {\em Lecture Notes in Pure and Appl. Math.}, pages 119--146. Dekker, New York, 1998.

\bibitem[GKR20]{left3}
Richard {Garner}, Magdalena {K\c{e}dziorek}, and Emily {Riehl}.
\newblock Lifting accessible model structures.
\newblock {\em J. Topol.}, 13(1):59--76, 2020.

\bibitem[GP87]{families}
Luzius Gr\"{u}nenfelder and Robert Par\'{e}.
\newblock Families parametrized by coalgebras.
\newblock {\em J. Algebra}, 107(2):316--375, 1987.

\bibitem[GPS25]{coalgKtheory}
Teena Gerhardt, Maximilien P\'eroux, and W.~Hermann~B. Sor\'e.
\newblock Coalgebraic {$K$}-theory, 2025.
\newblock ArXiv:2503.04897.

\bibitem[Hat65]{hattori}
Akira Hattori.
\newblock Rank element of a projective module.
\newblock {\em Nagoya Math. J.}, 25:113--120, 1965.

\bibitem[Heu24]{gijs}
Gijs Heuts.
\newblock Koszul duality and a conjecture of {F}rancis-{G}aitsgory, 2024.
\newblock ArXiv:2408.06173.

\bibitem[Hir03]{hir}
Philip~S. Hirschhorn.
\newblock {\em Model categories and their localizations}, volume~99 of {\em Mathematical Surveys and Monographs}.
\newblock American Mathematical Society, Providence, RI, 2003.

\bibitem[HKRS17]{hkrs}
Kathryn Hess, Magdalena K\c{e}dziorek, Emily Riehl, and Brooke Shipley.
\newblock A necessary and sufficient condition for induced model structures.
\newblock {\em J. Topol.}, 10(2):324--369, 2017.

\bibitem[Hov99]{hovey}
Mark Hovey.
\newblock {\em Model categories}, volume~63 of {\em Mathematical Surveys and Monographs}.
\newblock American Mathematical Society, Providence, RI, 1999.

\bibitem[HPS09]{hess2009cohochschild}
Kathryn Hess, Paul-Eug\`ene Parent, and Jonathan Scott.
\newblock Co{H}ochschild homology of chain coalgebras.
\newblock {\em J. Pure Appl. Algebra}, 213(4):536--556, 2009.

\bibitem[HR21a]{nima}
Kathryn {Hess} and Nima {Rasekh}.
\newblock {Shadows are Bicategorical Traces}.
\newblock {\em arXiv e-prints}, page arXiv:2109.02144, September 2021.

\bibitem[HR21b]{bicomodthh}
Geoffroy Horel and Maxime Ramzi.
\newblock A multiplicative comparison of {M}ac {L}ane homology and topological {H}ochschild homology.
\newblock {\em Ann. K-Theory}, 6(3):571--605, 2021.

\bibitem[HS16]{HSwaldausen}
Kathryn Hess and Brooke Shipley.
\newblock Waldhausen {$K$}-theory of spaces via comodules.
\newblock {\em Adv. Math.}, 290:1079--1137, 2016.

\bibitem[HS21]{HScothh}
Kathryn Hess and Brooke Shipley.
\newblock Invariance properties of co{H}ochschild homology.
\newblock {\em J. Pure Appl. Algebra}, 225(2):Paper No. 106505, 27, 2021.

\bibitem[Kla22]{klanderman2022computations}
Sarah Klanderman.
\newblock Computations of relative topological co{H}ochschild homology.
\newblock {\em J. Homotopy Relat. Struct.}, 17(3):393--417, 2022.

\bibitem[Lan02]{lang}
Serge Lang.
\newblock {\em Algebra}, volume 211 of {\em Graduate Texts in Mathematics}.
\newblock Springer-Verlag, New York, third edition, 2002.

\bibitem[Lod98]{loday1998cyclichomology}
Jean-Louis Loday.
\newblock {\em Cyclic homology}, volume 301 of {\em Grundlehren der mathematischen Wissenschaften [Fundamental Principles of Mathematical Sciences]}.
\newblock Springer-Verlag, Berlin, second edition, 1998.
\newblock Appendix E by Mar\'{\i}a O. Ronco, Chapter 13 by the author in collaboration with Teimuraz Pirashvili.

\bibitem[Lur17]{HA}
Jacob Lurie.
\newblock Higher algebra.
\newblock \url{https://www.math.ias.edu/~lurie/papers/HA.pdf}, 2017.
\newblock electronic book.

\bibitem[Mal17]{Cary}
Cary Malkiewich.
\newblock Cyclotomic structure in the topological {H}ochschild homology of {$DX$}.
\newblock {\em Algebr. Geom. Topol.}, 17(4):2307--2356, 2017.

\bibitem[McC01]{mccleary}
John McCleary.
\newblock {\em A user's guide to spectral sequences}, volume~58 of {\em Cambridge Studies in Advanced Mathematics}.
\newblock Cambridge University Press, Cambridge, second edition, 2001.

\bibitem[NS18]{tch}
Thomas Nikolaus and Peter Scholze.
\newblock On topological cyclic homology.
\newblock {\em Acta Math.}, 221(2):203--409, 2018.

\bibitem[P{\'e}r22a]{coalgenr}
Maximilien P{\'e}roux.
\newblock The coalgebraic enrichment of algebras in higher categories.
\newblock {\em J. Pure Appl. Algebra}, 226(3):Paper No. 106849, 11, 2022.

\bibitem[P{\'e}r22b]{dkcoalg}
Maximilien P{\'e}roux.
\newblock Coalgebras in the {D}wyer-{K}an localization of a model category.
\newblock {\em Proc. Amer. Math. Soc.}, 150(10):4173--4190, 2022.

\bibitem[P{\'e}r24]{pertower}
Maximilien P{\'e}roux.
\newblock A monoidal {D}old-{K}an correspondence for comodules.
\newblock {\em J. Pure Appl. Algebra}, 228(8):Paper No. 107660, 2024.

\bibitem[P{\'e}r25]{connectivecomod}
Maximilien P{\'e}roux.
\newblock Rigidificaton of connective comodules, 2025.
\newblock To appear in Proc. Amer. Soc., arXiv:2006.09398.v5.

\bibitem[Pon10]{ponto2008fixed}
Kate Ponto.
\newblock Fixed point theory and trace for bicategories.
\newblock {\em Ast\'{e}risque}, pages xii+102, 2010.

\bibitem[Pos23]{leo}
Leonid Positselski.
\newblock Homological full-and-faithfulness of comodule inclusion and contramodule forgetful functors, 2023.
\newblock arXiv2301.09561.v2.

\bibitem[PS13]{Ponto_2012}
Kate Ponto and Michael Shulman.
\newblock Shadows and traces in bicategories.
\newblock {\em J. Homotopy Relat. Struct.}, 8(2):151--200, 2013.

\bibitem[PS14]{traceinSMon}
Kate Ponto and Michael Shulman.
\newblock Traces in symmetric monoidal categories.
\newblock {\em Expo. Math.}, 32(3):248--273, 2014.

\bibitem[PS19]{perouxshipley}
Maximilien P\'{e}roux and Brooke Shipley.
\newblock Coalgebras in symmetric monoidal categories of spectra.
\newblock {\em Homology Homotopy Appl.}, 21(1):1--18, 2019.

\bibitem[Qui88]{quillencochain}
Daniel Quillen.
\newblock Algebra cochains and cyclic cohomology.
\newblock {\em Publications Math\'ematiques de l'IH\'ES}, 68:139--174, 1988.

\bibitem[Qui89]{ddaniel}
Daniel Quillen.
\newblock Cyclic cohomology and algebra extensions.
\newblock {\em $K$-Theory}, 3(3):205--246, 1989.

\bibitem[Rav86]{ravenel}
Douglas~C. Ravenel.
\newblock {\em Complex cobordism and stable homotopy groups of spheres}, volume 121 of {\em Pure and Applied Mathematics}.
\newblock Academic Press, Inc., Orlando, FL, 1986.

\bibitem[Sta65]{stallings}
John Stallings.
\newblock Centerless groups---an algebraic formulation of {G}ottlieb's theorem.
\newblock {\em Topology}, 4:129--134, 1965.

\bibitem[Swa70]{Swan}
Richard~G. Swan.
\newblock {\em {$K$}-theory of finite groups and orders}, volume Vol. 149 of {\em Lecture Notes in Mathematics}.
\newblock Springer-Verlag, Berlin-New York, 1970.

\bibitem[Tak77]{takeuchi}
Mitsuhiro Takeuchi.
\newblock Morita theorems for categories of comodules.
\newblock {\em J. Fac. Sci. Univ. Tokyo Sect. IA Math.}, 24(3):629--644, 1977.

\bibitem[Wal79]{waldhausen}
Friedhelm Waldhausen.
\newblock Algebraic {$K$}-theory of topological spaces. {II}.
\newblock In {\em Algebraic topology, {A}arhus 1978 ({P}roc. {S}ympos., {U}niv. {A}arhus, {A}arhus, 1978)}, volume 763 of {\em Lecture Notes in Math.}, pages 356--394. Springer, Berlin, 1979.

\bibitem[Wal85]{waldy}
Friedhelm Waldhausen.
\newblock Algebraic {$K$}-theory of spaces.
\newblock In {\em Algebraic and geometric topology ({N}ew {B}runswick, {N}.{J}., 1983)}, volume 1126 of {\em Lecture Notes in Math.}, pages 318--419. Springer, Berlin, 1985.

\bibitem[Wei94]{weibel}
Charles~A. Weibel.
\newblock {\em An introduction to homological algebra}, volume~38 of {\em Cambridge Studies in Advanced Mathematics}.
\newblock Cambridge University Press, Cambridge, 1994.

\bibitem[Wei13]{Kbook}
Charles~A. Weibel.
\newblock {\em The {$K$}-book}, volume 145 of {\em Graduate Studies in Mathematics}.
\newblock American Mathematical Society, Providence, RI, 2013.
\newblock An introduction to algebraic $K$-theory.

\end{thebibliography}
\bibliographystyle{alpha}

\end{document}